\documentclass[11pt,a4paper]{article}
\usepackage{amsmath}
\usepackage{amsfonts}
\usepackage{amssymb}
\usepackage{amsthm}
\usepackage{braket}
\usepackage{fullpage}
\usepackage{graphicx}
\usepackage{multicol}
\usepackage{enumitem}
\usepackage[boxruled,linesnumbered]{algorithm2e}
\SetKwInput{KwInput}{Input}
\SetKwInput{KwOutput}{Output}
\usepackage{mysec, color, stmaryrd}
\usepackage{hyperref}
\usepackage{cancel}
\usepackage[normalem]{ulem}
\usepackage{subfigure}
\usepackage{float}
\usepackage{ytableau}

\usepackage{bbm,bm}

\newcommand{\myp}{\mbox{$\:\!$}}
\newcommand{\mypp}{\mbox{$\;\!$}}

\newcommand{\myn}{\mbox{$\;\!\!$}}
\newcommand{\mynn}{\mbox{$\:\!\!$}}

\newcommand\sbullet[1][.5]{\mathbin{\vcenter{\hbox{\scalebox{#1}{$\bullet$}}}}}

\newcommand{\NN}{\mathbb{N}} 
\newcommand{\ZZ}{\mathbb{Z}} 
\newcommand{\RR}{\mathbb{R}} 
\newcommand{\CC}{\mathbb{C}} 

\newcommand{\rme}{\mathrm{e}}
\newcommand{\rmd}{\mathrm{d}}
\newcommand{\rmi}{\mathrm{i}}
\newcommand{\PP}{{\mathsf P}}
\newcommand{\QQ}{{\mathsf P\mynn}}
\newcommand{\EE}{{\mathsf E}}
\newcommand{\Var}{{\mathsf{Var}}}
\newcommand{\Cov}{{\mathsf{Cov}}}

\def\MR#1{\href{https://mathscinet.ams.org/mathscinet-getitem?mr=#1}{MR#1}}

\newcommand{\ignore}[1]{}

\numberwithin{equation}{section}

\SetKwComment{Comment}{$//$\ }{}

\title{Boltzmann Distribution on ``Short'' Integer Partitions\\ with Power Parts: Limit Laws
and Sampling}
\author{Jean C.\  Peyen\thanks{\,Corresponding author, \href{mailto:mmjcp@leeds.ac.uk}{mmjcp@leeds.ac.uk}, ORCID \href{https://orcid.org/0000-0002-7735-4005}{0000-0002-7735-4005}}, \,Leonid V.\,Bogachev\thanks{\href{mailto:L.V.Bogachev@leeds.ac.uk}{\,L.V.Bogachev@leeds.ac.uk}, ORCID \href{https://orcid.org/0000-0002-2365-2621}{0000-0002-2365-2621}},  and Paul P.\,Martin\thanks{\href{mailto:P.P.Martin@leeds.ac.uk}{\,P.P.Martin@leeds.ac.uk}, ORCID \href{https://orcid.org/0000-0002-8141-9465}{0000-0002-8141-9465}}\\[.4pc] \textit{School of Mathematics, University of Leeds, Leeds LS2 9JT, UK}}

\newtheorem{theorem}{Theorem}[section]

\newtheorem{lemma}[theorem]{Lemma}

\newtheorem{corollary}[theorem]{Corollary}

\theoremstyle{definition}
\newtheorem{definition}{Definition}[section]
\newtheorem{assumption}{Assumption}[section]

\theoremstyle{remark}
\newtheorem{remark}{Remark}[section]

\date{}

\begin{document}
\maketitle

\vspace{-1pc}
\begin{abstract}
The paper is concerned with the asymptotic analysis of a family of Boltzmann (multiplicative) distributions over the set $\check{\varLambda}^{q}$ of \emph{strict} integer partitions (i.e., with 
unequal parts) into perfect $q$-th powers. 
A combinatorial link is provided via a suitable conditioning 
by fixing the partition \emph{weight} (the sum of parts)
and \emph{length} (the number of parts), leading to uniform distribution on the corresponding subspaces of partitions. 
The Boltzmann measure is calibrated through the hyper-parameters $\langle N\rangle$ and $\langle M\rangle$ controlling the expected weight and length, respectively. 
We study ``short''  partitions, where 
the parameter $\braket{M}$ is either fixed or grows slower than for typical 
partitions in $\check{\varLambda}^{q}$. For this model, we obtain a variety of limit theorems including the asymptotics of the cumulative cardinality in the case of fixed $\braket{M}$ and a limit shape result in the case of slow growth of $\braket{M}$. 
In both cases, we also characterize the joint distribution of the weight and length, as well as the growth of the smallest and largest parts. 
Using these results we construct suitable 
sampling algorithms and analyze their performance. 
\end{abstract}


\medskip\noindent
\emph{MSC\,2020:} Primary 05A17; Secondary 05A16, 60C05, 68Q87, 82B10

\tableofcontents  

 \section{Introduction}
\subsection{Integer partitions: setting the scene}
``What, yet another paper on integer partitions?''\ Well, guilty as charged\,---\,but to be fair, this classical area of mathematics, dating back to Euler, Sylvester, MacMahon, Hardy and Ramanujan,
is fresh as ever, and seems to be a non-depletable source of fascinating problems and many beautiful results (see the authoritative monograph by Andrews \cite{Andrews} for historical comments and further references). 

Integer partitions constitute one of the most basic structures in additive number theory and combinatorics\,---\,a partition of an integer $n\in\NN$ is just a decomposition as a sum of natural parts, up to reordering (e.g., partitions $5=2+3=3+2$ are not distinguished). Perhaps, the most celebrated result in the asymptotic theory of integer partitions is the ingenious formula by Hardy and Ramanujan \cite[Sec.\mypp 1.4, p.\mypp 79, and Sec.\mypp 1.7, pp.\,84--85]{Ramanujan} for the number $p(n)$ of all partitions of a large integer $n$, with the principal term\footnote{Obtained independently by Uspensky \cite{Uspensky}; see also a streamlined treatment by  Ingham \cite{Ingham}.} reading
\begin{equation}\label{eq:HR}
p(n)\sim \frac{1}{4\myp\sqrt{3}\,n}\exp\left(\pi\mypp\sqrt{\frac{2\myp n}{3}}\myp\right)\qquad (n\to\infty).
\end{equation}
The problem was attacked in \cite{Ramanujan} using the then new \emph{circle method} (later on attributed to Hardy and Littlewood), based on the Cauchy integral formula and a careful singularity analysis of the corresponding generating function $F(z)=\prod_{n=0}^\infty (1-z^n)^{-1}$ (known already to Euler). Interestingly, the same paper also contains a general result for partitions into sums of perfect $q$-th powers \cite[Sec.\mypp 7.3, p.\mypp 111]{Ramanujan} (later on elaborated by Wright \cite{Wright}),
\begin{equation}\label{eq:HR-powers}
p^{\myp (q)}(n)\sim \frac{k(q)\,\sqrt{q/(q+1)}}{(2\pi)^{(q+1)/2}\cdot n^{1/(q+1)-3/2}}\exp\left((q+1)\,k(q)\,n^{1/(q+1)}\myp\right)\qquad (n\to\infty),
\end{equation}
where
$$
k(q):=\left\{\frac{1}{q}\,\Gamma\!\left(1+\frac{1}{q}\right)\zeta\!\left(1+\frac{1}{q}\right)\right\}^{q/(q+1)}\qquad (q\in\NN).
$$

A similar result was stated in \cite{Ramanujan} (and later on elaborated by Hua \cite{Hua}) for \emph{strict partitions}, that is, those with unequal parts,\footnote{The leading term of the asymptotic expansion of $\check{p}(n)$  was given in \cite[Sec.\mypp 7.1, p.\mypp 109]{Ramanujan} in terms of (the derivative of) the Bessel function $J_0(z)$, from which it is easy to derive an explicit formula \eqref{eq:HR-dist} using the relation $J_0'(z)=-J_1(z)$ \cite[10.6.3]{NIST} and the asymptotics of $J_1(z)$ \cite[10.7.8]{NIST}. See more direct derivations in \cite{Uspensky, Ingham}.}
\begin{equation}\label{eq:HR-dist}
\check{p}(n)\sim \frac{1}{4\cdot 3^{1/4}\,n^{3/4}}\exp\left(\pi\mypp\sqrt{\frac{n}{3}}\myp\right)\qquad (n\to\infty).
\end{equation}
Strict partitions and their generating function $\check{F}(z)=\prod_{n=0}^\infty (1+z^x)$ were considered by Euler who noticed
that $\check{p}(n)$ coincides with the number of partitions of $n$ \emph{with odd parts}, using a simple identity\footnote{See \cite{Pak} for a modern survey of bijection methods and results for various classes of integer partitions.} for the corresponding generating functions, $\check{F}(z)=F(z)/F(z^2)$, that is,
$$
\prod_{n=0}^\infty (1+z^n)=\prod_{n=0}^\infty \frac{1-z^{2n}}{1-z^n}= \prod_{n=0}^\infty (1-z^n)^{-1}\left(\prod_{n=0}^\infty (1-z^{2n})^{\myn-1}\right)^{-1}\!.
$$

These early benchmarks stimulated a growing interest in the (asymptotic) enumeration of integer partitions under various constraints, such as restrictions on the source of parts and/or their number, on permitted repetitions of parts, etc.\  \cite{Wright,Erdos,Ingham,Meinardus,HT,Lee}.

In particular, there has been intensive research into additive representations of integers with $q$-power parts, starting with $q=2$ (squares) and dating back to Hardy and Ramanujan \cite{Ramanujan} (see also \cite{Vaughan, Wright}).
In connection with combinatorial enumeration, the class of (non-strict) integer partitions with a fixed number $m$ of $q$-power parts is featured in at least two classical mathematical gems, the \emph{Waring problem} \cite{HW,Waring} and the \emph{Gauss circle problem} \cite[Sec.\:F1, pp.\,365--367]{Guy}, both originally considered for squares ($q=2$) and with $m\le 4$ or $m\le 2$ parts, respectively. The Waring  problem concerns $q$-power representability of \emph{all positive integers}\footnote{It is known that $g(q)\ge 2^q+\lfloor (3/2)^q\rfloor -2$ for any $q\in\NN$, and it is believed that in fact the equality is true\,---\,although exceptions may be possible in principle, no counter-examples have been found to date. While this version of the Waring problem (i.e., for all numbers $n\in\NN$) is almost completely settled, the asymptotic version asking for the smallest number of parts, denoted $G(q)$, sufficient to partition any sufficiently large natural number into a sum of $q$-powers, remains largely open (clearly, $G(2)=4$, and it is also known that $G(4)=16$). See further details and references in \cite{Waring}.} using at most $g(q)$ parts, whereas the Gauss circle problem focuses on the cumulative cardinality of such representations (more precisely, on error bounds for the area/volume approximation). For instance, by virtue of the Lagrange theorem it is known that $g(2)=4$, that is, any natural number can be written as a sum of at most $4$ squares, while $3$ squares may not be  enough. Moreover, Legendre's theorem gives an exact description of integers that can be represented as a sum of $3$ squares\,---\,these are numbers not congruent to $7$ ($\mathrm{mod}\ 8$); for example, for $23=7\ (\mathrm{mod}\ 8)$ the only representation is $9+9+4+1=23$. On the other hand, numbers not representable using exactly $4$ positive squares are given by the sequence comprising eight odd numbers, $1$, $3$, $5$, $9$, $11$, $17$, $29$, $41$, and all numbers of the form $\ell\cdot 4^k$ with $k\in\NN_0$ and $\ell\in\{2,6,14\}$. These two sequences overlap: for example, $16+1+1+1=9+9+1=19=3\ (\mathrm{mod}\ 8)$. Another famous result, now about sums with up to $2$ squares, is the Landau theorem \cite{Landau} stating that the fraction of numbers up to $n$ enjoying such a representation is asymptotically given by $K n/\sqrt{\log n}$, where $K\doteq 0.764223653$ is the \emph{Landau--Ramanujan constant}.

\subsection{Random integer partitions}
A more recent boost of research in this area has been due to a ``statistical''  approach focusing on asymptotic properties of typical random partitions and other decomposable combinatorial structures of large size
(see, e.g., \cite{ABT,Arratia-Tavare,Duchon,V0}). The words ``typical'' and ``random'' imply that partition ensembles of interest are endowed with suitable probability measures, such as the uniform distribution on the spaces of partitions of a given $n\in\NN$ (so that all such partitions are assumed equally likely). Amongst the first results in this direction established in a seminal paper by Erd\H{o}s and Lehner \cite{Erdos} is that the growth rate of the number of parts in the bulk of integer partitions of  large $n$ (i.e., in the sense of Law of Large Numbers) is given by $\pi^{-1}\sqrt{3\myp n/2}\,\log n$. A snapshot of subsequent advances is documented, for example, in papers  \cite{Erdos-Turan,Fristedt,Hwang, Pittel,V0} and references therein. In particular, this research has led to the discovery of so-called \emph{limit shapes} of partition ensembles, which describe a typical settlement of parts and their multiplicities within large partitions under appropriate scaling (see \cite{V0,Vershik1,Pittel,EG,Bogachev,Bogachev2,VY1,Yakubovich}). The limit shapes have a natural geometric interpretation through the so-called \emph{Young diagrams} (also known as \emph{Ferrers' graphs}), with blocks visually representing constituent parts of the integer partitions (see more detail in Section~\ref{sec:2.2} below). Incidentally, Young diagrams make it self-evident (by flipping rows and columns, called \emph{conjugation}) that the number of partitions of $n$ with at most $m$ parts is the same as the number of partitions of $n$ with the largest part not exceeding $m$, which immediately implies that the aforementioned asymptotics for the typical number of parts also hold for the largest part \cite{Erdos}.

The modern approach to the asymptotic analysis of random combinatorial structures is based on a suitable \emph{randomization} of the model parameters (collectively called \emph{poissonization}) and the subsequent \emph{conditioning} (or \emph{de-poissonization}) in order to return to the original (say, uniform) distribution  with fixed parameters (see, e.g., \cite{ABT,Arratia-Tavare,Khinchin2,Kolchin,Pitman,V0} and the vast bibliography therein). In the context of random integer partitions, this method was first successfully applied by Fristedt \cite{Fristedt}, leading to the probability measure on the space of partitions of \emph{all integers} $n\in\NN$ by assigning to each such partition a probability proportional to $z^n$, respectively, where $z\in(0,1)$ is a free parameter. Under such a measure (commonly referred to as \emph{Boltzmann distribution}), the multiplicities of candidate parts $j\in\NN$, previously restricted by the partition target $n$ (often called \emph{weight}), become independent geometric random variables with success parameter $1-z^j$, respectively. Furthermore, conditioning on the partition weight to be equal to  $n$ restores the uniform distribution on the space of all partitions of that $n$. This holds for any value of $z$, but it is helpful to calibrate the randomized model by replicating the original macroscopic properties (such as the partition weight) in terms of expectation. Crucially, for the conditioning trick to work effectively, asymptotic information is needed about the probability of the  condition, specialized here as a weighted sum of random multiplicities, thus taking a familiar form of a local limit theorem in probability theory, albeit somewhat peculiar since the number of terms in the sum is random but almost surely (a.s.)\ finite (see \cite{Fristedt,GH,Bogachev,VY1,VFY,Hwang}).

Similar ideas are well known as ``equivalence of ensembles'' in statistical physics, where the probabilistic description of the particle system of interest (e.g., ideal gas) may vary subject to optional fixation of the total energy and/or the number of particles, leading accordingly to \emph{micro-canonical}, \emph{canonical} or \emph{grand-canonical Gibbs distributions}\footnote{Cf.\ also \cite{Bogachev2}, where the term ``meso-canonical'' was proposed as better suited to the space of integer partitions with a given weight and any length (interpreted as an assembly with fixed energy and an indefinite  number of particles).} \cite{Huang,Greiner}. The usual tool to establish such equivalence is via the Darwin--Fowler method involving a complicated saddle-point asymptotic analysis of high-dimensional integrals (see \cite{Huang}). As an alternative, Khinchin \cite{Khinchin,Khinchin2} advocated a systematic use of local limit theorems of probability theory in problems of statistical
mechanics, which facilitates the analysis by invoking probabilistic insight and well-developed analytical tools.

In fact, connections with statistical physics go even deeper, whereby integer partitions serve as a model for the random partitioning 
of total energy (under a suitable choice of units) in a large assembly of indistinguishable particles, with the Boltzmann distribution arising naturally as a thermodynamic equilibrium 
\cite{Auluck,DeGr,Vershik1,Vershik2, VY2,Bogachev2,Comtet}. Specific models of relevance include an ensemble of harmonic oscillators at high temperatures and ideal quantum gases, where discrete partition structures are particularly tailored to the energy quantization \cite{Auluck,Temp1}. In this context, non-strict partitions are interpreted as \emph{bosons} following the \emph{Bose--Einstein statistics}, with no restriction on the energy level occupancy, while strict partitions model \emph{fermions}
under the \emph{Fermi--Dirac statistics} obeying the Pauli exclusion principle of not allowing more than one particle on any energy level. This analogy is quite productive\,---\,for instance, it offers an insightful combinatorial explanation of the \emph{Bose--Einstein condensation}, manifested as a measurable excess of particles at the lowest (ground) energy level at temperatures close to the absolute zero \cite{Auluck,Vershik2}. Furthermore, there is a natural physical motivation for considering integer partitions with power parts (see more detail in \cite{Vershik1,Vershik2, DeGr,Comtet,Roccia} and reference therein).

More generally, Boltzmann distributions can be defined for a large variety of decomposable combinatorial structures $\mathfrak{C}=\{\mathfrak{c}\}$, such as assemblies, multisets, selections, sequences, trees, etc.\ \cite{ABT,Duchon,Flajolet-Fusy-Pivoteau}, again assigning to each element $\mathfrak{c}\in\mathfrak{C}$ a geometric probability weight proportional to $z^{s(\mathfrak{c})}$, as long as an additive ``size function'' $\mathfrak{C}\ni \mathfrak{c}\mapsto s(\mathfrak{c})\in\RR_+$ is available (such as the partition weight). Multi-parametric versions of the Boltzmann distribution may also be considered using the weights of the form $\prod_{i=1}^k \mynn z_i^{s_i(\mathfrak{c})}$, with suitable size functions $s_1(\mathfrak{c})\dots,s_k(\mathfrak{c})$ \cite{Vershik1,
BBD}. For instance, this may be needed if we wish to control more than one macroscopic characteristic of a combinatorial object $\mathfrak{c}\in\mathfrak{C}$, such as the number of parts (length) in addition to the partition weight \cite{Vershik1} (see also Section~\ref{sec:2.4} in the present paper). Another reason may arise if we are dealing with a truly multi-dimensional structure, such as vector partitions, convex lattice polygonal lines, or digitally convex polyominoes \cite{Sinai, Vershik1,BZ4,Bureaux,Bodini}. Boltzmann distributions are very popular in computer science  as an effective tool to sample random instances of combinatorial objects\,---\,if required, with a given size, 
exact or approximate (which is achieved via rejection applied to the output of a free sampler, so as to implement the conditioning step), and with a uniform distribution of the output \cite{Duchon,Flajolet-Fusy-Pivoteau,Bodini,BeFaRa}.

Some alternative probability distributions on partition spaces, different from the Boltzmann class, are also of great interest, such as the \emph{Ewens sampling formula} \cite{Ewens,Kingman}, with important applications in population genetics and ecology, and the \emph{Plancherel measure} arising in connection with representation theory of the symmetric group $\mathfrak{S}_n$ of all permutations of the set $\{1, \dots,n\}$ \cite{Fulton}. The limit shape in the Plancherel model is known \cite{Logan-Shepp,VK1,VK2}, but here the poissonization of the partition weight leads to a \emph{determinantal ensemble} (for more detail and further references, see \cite{BOO2000,BoSu,Su}).
Finally, let us mention that integer partitions find stunning applications in many other fields, such as economics \cite{GS}, optimal transport \cite{Hohloch}, and statistics of scientific citations \cite{Yong,NBV}. AAM Final version (Copy)

\subsection{Focus of the paper}
Confronted with the awesome wealth of past research into integer partitions briefly sketched above, why should the reader take any interest in the present paper? The novelty is brought about by our focus on integer partitions under a conjunction of three constraints as follows: 
\begin{enumerate}
\item[(i)]\label{i} Firstly, the source of parts is limited to perfect $q$-th powers, with some $q\in\NN$. 
\item[(ii)]\label{ii} Secondly, partitions are assumed to be \emph{strict}, in that all parts must be distinct.
\item[(iii)]\label{iii}
Finally, we consider ``short'' partitions, where the length (i.e., the number of parts) is either fixed or grows  slowly 
as compared to the ``free'',  unrestricted regime.  
\end{enumerate}
In what follows, we denote by $\check{\varLambda}^q$
the class of integer partitions satisfying the first two conditions. Of course, each of the constraints (i) to (iii) taken alone is not novel and has been considered quite extensively; for instance, power parts were considered in \cite{Wright,Vershik1,DeGr,Comtet,Roccia}; restricted growth of length was addressed in \cite{Erdos,VY2,Romik,Bridges}; and strict partitions are a classical subject (see, e.g., \cite{Hua,Roccia,VFY}). However, a juxtaposition of the first and third  constraints is new, leading to some interesting results. The choice of strict partitions is less significant and is mostly motivated by virtue of making the analysis a little easier.\footnote{Non-strict partitions under constraints (i) and (iii) remain to be studied. An intermediate model with a finite bound on multiplicities is also of interest (cf.\ the \emph{Brillouin statistics} in statistical physics \cite{Lindsay}, \cite[p.\,43]{Loeve}).}

The random structure with which we endow the space $\check{\varLambda}^q$ is based on the Boltzmann distribution, with probability weights assigned to each partition $\lambda\in\check{\varLambda}^q$ proportional to $z_1^{N_\lambda}z_2^{M_\lambda}$ ($z_1,z_2\in(0,1)$), where $N_\lambda$ and $M_\lambda$ are the weight and length of $\lambda$, respectively. Crucially, under this measure the random weight and length are finite with probability $1$; but its little idiosyncrasy is that the ``empty'' partition $\lambda_\varnothing$, with zero weight and length, is assigned a positive probability.\footnote{This nuisance can be easily eliminated by conditioning on $N_\lambda>0$.} We calibrate the parameters $z_1,z_2$ through the moment conditions $\EE_{\bm{z}}(N_\lambda) =\braket{N}$, $\EE_{\bm{z}}(M_\lambda)=\braket{M}$, where $\langle N\rangle$ and $\langle M\rangle$ are the external hyper-parameters that are used to control the distribution of the random weight and length. The slow growth condition (ii) is specialized as $\kappa:=\langle M\rangle^{q+1}\myn/\langle N\rangle\to 0$. The calibrating equations can be solved asymptotically as $\braket{N}\to\infty$ yielding the leading terms $\log z_1\sim -\langle M\rangle\myn/(q\mypp\langle N\rangle)=o(1)$ and $z_2\sim \kappa^{1/q}/(q^{1/q}\,\Gamma(1+1/q))=o(1)$ (Theorem~\ref{th:cal}). Using the terminology of statistical physics, $(-\log z_1)$ is interpreted as the \emph{inverse temperature} $1/(k_B\myp T)$ (where $k_B$ is the Boltzmann constant and $T>0$ the absolute temperature), while $z_2$ has the meaning of \emph{fugacity} \cite{Comtet,Roccia,Huang,Greiner}. In particular, it follows that the particle assembly is in a high-temperature regime and at a low fugacity, which is explained by a constrained number of particles, insufficient to create a non-negligible pressure.

Under the Boltzmann distribution calibrated as indicated above, we prove a variety of limiting results describing the joint behavior of the partition weight $N_\lambda$ and length $M_\lambda$. Unsurprisingly, such results are different according as the  expected length $\braket{M}$ is either fixed or growing slowly. In the former case, $M_\lambda$ is asymptotically Poisson with parameter $\braket{M}$ (which can be anticipated by virtue of general Poisson approximation results), while $N_\lambda$, conditionally on $M_\lambda=m$ and scaled by $m/(q\myp\langle N\rangle)$, is asymptotically gamma-distributed with shape parameter $m/q$ (Theorem~\ref{thm1}).
This implies that the scaled parts of a random partition conditioned on length $m$ can be interpreted as the order statistics of an independent sample of size $m$ from gamma distribution with shape parameter $1/q$. It also follows that for $m$ strictly less than $q$ the limiting density of the scaled weight $N_\lambda$ has a power singularity of order $x^{m/q-1}$ at zero. In the language of statistical physics, this means that the Boltzmann assembly of $m<q$ fermions with $q$-power energy levels may have, with a sizable probability, an untypically low total energy as compared to the expected ``target'' $\langle N\rangle\to\infty$. It would be interesting to give a physical justification of such a ``small system condensation'' phenomenon. 

In the slow growth regime for the expected length 
$\langle M\rangle$, the random pair $(N_\lambda,M_\lambda)$ proves to be  asymptotically normal, with the limiting correlation coefficient $1/\sqrt{q+1}$ (Theorem~\ref{th:CLT})\,. 
In this case, we have found the limit shape of scaled Young diagrams (Theorem~\ref{th:LS}), determined (up to a constant) by the tail of the $\Gamma(1/q)$-integral, $x\mapsto \int_x^\infty\mynn u^{1/q-1}\mypp\rme^{-u}\,\rmd{u}$ ($x\ge 0$), and also showed that fluctuations around the limit shape are asymptotically normal (Theorem~\ref{th:Gauss}).

Asymptotics of extreme values in the partition spectrum (i.e., the smallest and largest parts) can also be analyzed. Specifically, if $\braket{M}$ is fixed then all parts appear to grow on the same linear scale of order $\langle N\rangle/\langle M\rangle$ (Theorem~\ref{th:minmax}). In the slow growth regime of $\braket{M}$, the extremal behavior is more interesting: the smallest part ``lives'' on the scale of order $\kappa^{-1}=\langle N\rangle/\myn\langle M\rangle^{q+1}$ and has a Weibull limit distribution, whereas the largest part scales roughly as $q\mypp\langle N\rangle/\langle M\rangle$ and has the Gumbel double-exponential limit (Theorem~\ref{th:maxM}). Note that the latter result conforms to the general pattern for the maximum observed earlier in many particular cases (see, e.g., \cite{Erdos,VY2,Comtet1}). 

It is easy to check that the Boltzmann distribution on $\check{\varLambda}^q$
conditioned on both weight $N_\lambda=n$ and length $M_\lambda=m$ reverts to the uniform distribution on the corresponding subspace $\check{\varLambda}^q(n,m)\subset \check{\varLambda}^q$. So it may be natural to attempt the conditional version of our Boltzmann-based limiting results in order to tailor them to the constrained spaces  $\check{\varLambda}^q(n,m)$.
As we explained above (cf.\ \cite{Fristedt,Yakubovich,Bogachev}), this approach requires a suitable local limit theorem about the asymptotics of the Boltzmann probability of the ``slicing'' condition that defines the subspace $\check{\varLambda}^q(n,m)\subset \check{\varLambda}^q$ (for large $n$ and $m$ either fixed or suitably large).
However, there is a problem: 
unlike the case $q=1$, where every number $n\in\NN$ is partitionable with a required number $m$ of (unequal) parts (as long as $1+\dots+m=m\mypp(m+1)/2\le n$), for $q\ge 2$ this is no longer guaranteed and the space $\check{\varLambda}^q(n,m)$ may appear to be empty, unless the pair $(n,m)$ is covered by a solution of the Waring problem with the $q$-th powers \cite{Waring}.

Because of such number-theoretic complications, we did not pursue this approach in the present paper, so for the most part (with two notable exceptions indicated below) we confined ourselves to the Boltzmann-based results, which are nonetheless quite interesting. 
One exception is that the conditioning device can be utilized in order to identify the growth rate of the \emph{cumulative cardinality} of the union  $\bigcup_{k\le x} \mynn\check{\varLambda}^q(k,m)$, with $m$ fixed 
(Theorem~\ref{th:ID}), which gives the leading term in the generalized $m$-dimensional Gauss circle problem under the $q$-norm in $\RR^m$. Also, we outlined a ``semi-local'' result (Theorem~\ref{thm1+}) in the slow growth regime to address the local type asymptotics for the length $M_\lambda$ complemented by the conditional limit for the weight $N_\lambda$.

Last but not least, our interest in the class of Boltzmann distributions on the partition spaces $\check{\varLambda}^q$ is also motivated by the sampling applications, wherein simple but efficient algorithms can be designed and implemented to sample uniformly distributed random instances from the subspaces of interest such as $\check{\varLambda}^q(n,m)$. Although such algorithms are intuitively appealing and straightforward thanks to the intrinsic independence of the random multiplicities of parts, there are familiar issues with ensuring the finiteness of the sampling loops. An agreed convention in computer science to resolve such issues is to deploy the so-called \emph{oracle} \cite{Duchon,Flajolet-Fusy-Pivoteau,BeFaRa}, which is a collective name for an external device that is capable of computing at request (exactly or approximately) the values of the corresponding generating function, serving as a normalizing denominator in probability expressions. 

In the present paper, we pursue a different approach by truncating the sampling loop on the basis of high statistical confidence, thus adopting the methodology of hypotheses testing in statistics. Similar ideas have been used before, for instance, in the well-known \emph{Miller--Rabin algorithm} for primality testing \cite{Rabin}. Selection of the proper truncation thresholds is guided by our limit theorems for the partition weight and length. This approach is especially useful in the case of partition spaces $\check{\varLambda}^q(n,m)$ that
 may suffer from being empty
 for some of the pairs $(n,m)$ (which may not be  known in advance). Moreover, we argue that the statistical approach to sampling may be beneficial as a practical tool to effectively explore the hypothetical partitionability of large integers.

\medskip
\noindent
\textit{Layout.} The rest of the paper is organized as follows.
Section \ref{sec:2} introduces the main elements that we use throughout this paper, including the definition of restricted classes of integer partitions and basic results about the Boltzmann distribution. In Section \ref{sec:3} we present important foundational results for our analysis of the class $\check{\varLambda}^q$ under the asymptotic regime specified by Assumption \ref{as0}. The main result of this section is Theorem \ref{th:cal} about the asymptotic calibration of the Boltzmann parameters. Section \ref{sec:4} focuses on partitions with fixed expected length, where the main  Theorem \ref{thm1} characterizes the limit distribution of the random length and weight. This result can be used for enumeration purposes, at least in the cumulative sense (Theorem~\ref{th:ID}). We also obtain the joint limit distribution of the largest and smallest parts (Theorem~\ref{th:minmax}).
Section \ref{sec:5} extends the analysis to partitions with a slowly growing expected length (under Assumption~\ref{as2}), where we obtain asymptotic results for the length and weight (Theorem \ref{th:CLT}) as well as for the extreme parts (Theorem~\ref{th:maxM}). In this regime it is also possible to derive the limit shape of properly scaled Young diagrams, which form a family of curves indexed by $q$ (Theorems \ref{th:LS} and~\ref{th:Gauss}).
Finally, Section \ref{sec:6} illustrates an application of the previously developed tools and results in the context of random sampling.
Designing issues and performance of our sampling  algorithms are discussed there in detail.

\medskip
\noindent
\textit{Some general notation:}
$\NN:=\{1,2,\dots\}$ is the set of natural numbers, $\NN_0:=\{0\}\cup \NN=\{0,1,2,\dots\}$, $\mathbb{N}^{q}:=\{j^q\colon j\in\mathbb{N}\}$. The cardinality of a (finite) set $\mathcal{A}$ (i.e., the number of elements) is denoted  $\vphantom{\int^y}\#\mathcal{A}$. Asymptotic comparisons: $a\sim b$ means that $a/b\to1$; $a=o(b)$ that $a/b\to0$; and $a=O(b)$ that $a/b$ is bounded.
Vectors are understood as rows, e.g.\ $\bm{z}=(z_1,z_2)$. 

\section{Preliminaries} \label{sec:2}
\subsection{Integer partitions}\label{sec:2.1}

For a given integer $n\in\NN$, a \emph{partition} of $n$ 
is a decomposition of $n$ 
into a sum of non-negative integers, disregarding the order of the terms; for example,
$35=10+7+5+5+4+3+1$ is an integer partition of $n=35$. 
To fix the notation, we adopt a convention of non-increasing ordering of terms; that is to say, a sequence of integers $\lambda_1\ge\lambda_2\ge\dots\ge 0$ 
with finitely many parts $\lambda_i>0$ is a partition of $n\in\NN$ 
if $n=\lambda_1+\lambda_2+\cdots$. 
This is expressed as $\lambda\vdash n$. We formally allow the case $n=0$ represented by the ``empty'' partition $\lambda_\varnothing=(0,0,\dots)$, with no parts. This is convenient when working with generating functions. 
The set of all partitions $\lambda\vdash n$ is denoted by $\varLambda(n)$, and
the set $\varLambda:=\bigcup_{n\in\NN_0}\varLambda(n)$ is the collection of \emph{all} integer partitions. The subset $\check{\varLambda}\subset \varLambda$ of \emph{strict} partitions is \strut{}defined by the property that all parts $(\lambda_i)$ are different from one another, $\lambda_1>\lambda_2>\cdots$. Accordingly, the set of all strict partitions $\lambda\vdash n$ is denoted $\check{\varLambda}(n)$.

For a partition $\lambda=(\lambda_i)\in\varLambda$, the sum $N_\lambda:=\lambda_1+\lambda_2+\cdots$ is referred to as its
\textit{weight} (i.e., $\lambda\vdash N_\lambda$), and the number of its parts
$M_\lambda:=\#\{\lambda_i\in\lambda\colon\lambda_i>0\}$ is called
the \textit{length} of~$\lambda$. Thus, for $\lambda\in\varLambda(n)$ we
have $N_\lambda=n$ and $M_\lambda\le n$. 
The largest and smallest parts of a partition $\lambda=(\lambda_i)$ are denoted $\lambda_{\rm max}=\lambda_1=\max_{1\le i\le M_\lambda} \lambda_i$ and $\lambda_{\rm min}=\lambda_{M_\lambda}=\min_{1\le i\le M_\lambda} \lambda_i$, respectively. We make a convention that for the empty partition $\lambda_{\varnothing}$, its largest and smallest parts are defined\footnote{The familiar paradox of such definitions, suggesting that the maximum is smaller than the minimum, is but a logical consequence of applying the operations $\sup$ and $\inf$ to the empty set $\varnothing$. Despite a counter-intuitive appearance, these definitions are perfectly consistent with our limiting results in Theorem~\ref{th:minmax}.} as $\lambda_{\rm max}=0$ and $\lambda_{\rm min}=\infty$.

The alternative notation $\lambda=(1^{\nu_1}
2^{\nu_2}\mynn\dots)$ refers to the \emph{multiplicities} 
of the parts involved,
$\nu_\ell:=\#\{\lambda_i\in\lambda\colon \lambda_i=\ell\}$
\,($\ell\in\NN$), with zero multiplicities usually omitted from the notation. Thus, the partition $\lambda\vdash 35$ in the example above can be written as $\lambda=1^13^14^15^2\myp7^110^1)$. 
The weight and length of a partition $\lambda\in\varLambda$ can be expressed through its  multiplicities $(\nu_\ell)$ as follows,
\begin{equation}\label{eq:NM}
N_\lambda= \sum_{\ell} \ell\mypp\nu_\ell\myp,\qquad M_\lambda=\sum_{\ell} \nu_\ell\myp.
\end{equation}
In terms of multiplicities $(\nu_\ell)$, the set of strict partitions $\check{\varLambda}$ is defined by the condition that any part $\ell$ can be used no more than once,
$$
\check{\varLambda}:=\{\lambda=(\ell^{\myp\nu_\ell})\in\varLambda\colon \nu_\ell\le 1\ \text{for all }\,\ell\}.
$$

\subsection{Young diagrams and limit shape}
\label{sec:2.2}
A partition $\lambda = (\lambda_1,\lambda_2,\dots)$ is succinctly visualized by its Young diagram $\varUpsilon_\lambda$ formed by
(left- and bottom-aligned) row blocks with $\lambda_1, \lambda_2,\dots$ unit square cells (see Figure \ref{fig:yddef}(a)). The upper boundary of $\varUpsilon_\lambda$ is a piecewise-constant, non-increasing function $Y_\lambda\colon [0,\infty) \to \NN_0$ defined by 
\begin{equation}\label{eq:Young}
Y_\lambda(x)=\sum_{\ell\ge x}\nu_\ell\qquad (x\ge0).
\end{equation}
In particular, 
$Y_\lambda(0)=M_\lambda$, while the area of the Young diagram $\varUpsilon_\lambda$ is 
$$
\int_0^\infty\! Y_\lambda(x)\,\rmd{x} =\sum_{\ell} \ell\mypp\nu_\ell = N_\lambda.
$$
According to the definition \eqref{eq:Young}, the step function $Y_\lambda(x)$ is left-continuous and has right limits (and is also right-continuous at the origin).

\begin{figure}[ht!]
\centering
\subfigure[]{\includegraphics[width=0.45\textwidth]{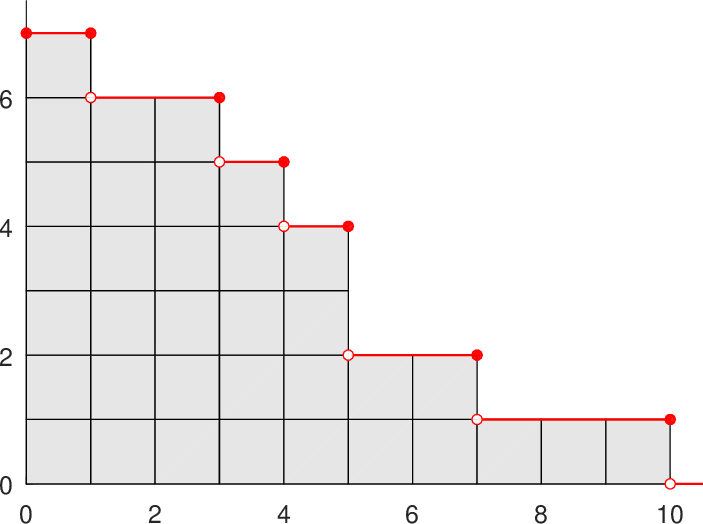} 
\put(-111,143.5){\line(-2,-1){40}}
\put(-109,143){\mbox{\small $y = Y_\lambda(x)$}}
\put(-70,95){\line(-2,-1){40}}
\put(-69,93){\mbox{\small $\varUpsilon_\lambda$}}}\hspace{1.5pc}
\subfigure[]{\includegraphics[width=0.45\textwidth]{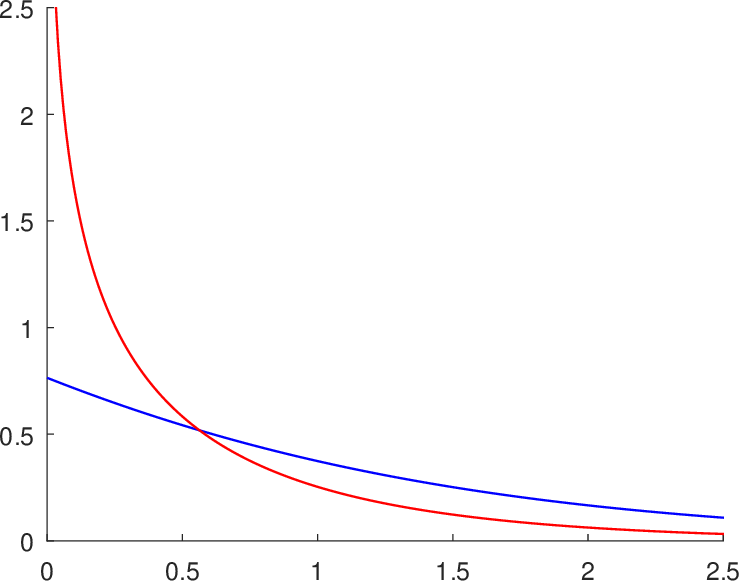}
\put(-132.5,95){\line(-2,-1){40}}\put(-131,93.5){\small$\varLambda(n)$}
\put(-60,51){\line(-2,-1){40}}
\put(-58.5,49.5){\small$\check{\varLambda}(n)$}
}
\caption{(a) The Young diagram $\varUpsilon_\lambda$ (shaded) of partition $\lambda = (10, 7, 5, 5, 4, 3, 1)$, with weight $N_\lambda = 35$ and length $M_\lambda = 7$. The graph of the step function $x\mapsto Y_\lambda(x)$ defined in \eqref{eq:Young} depicts the upper boundary of $\varUpsilon_\lambda$ (shown in red in the online version). (b) Two classical limit shapes,  
for unrestricted partitions  $\lambda\in\varLambda(n)$ (red) and strict partitions $\lambda\in\check{\varLambda}(n)$ (blue), determined by equations \eqref{eq:limit1} and \eqref{eq:limit2}, respectively.
}
\label{fig:yddef}
\end{figure}
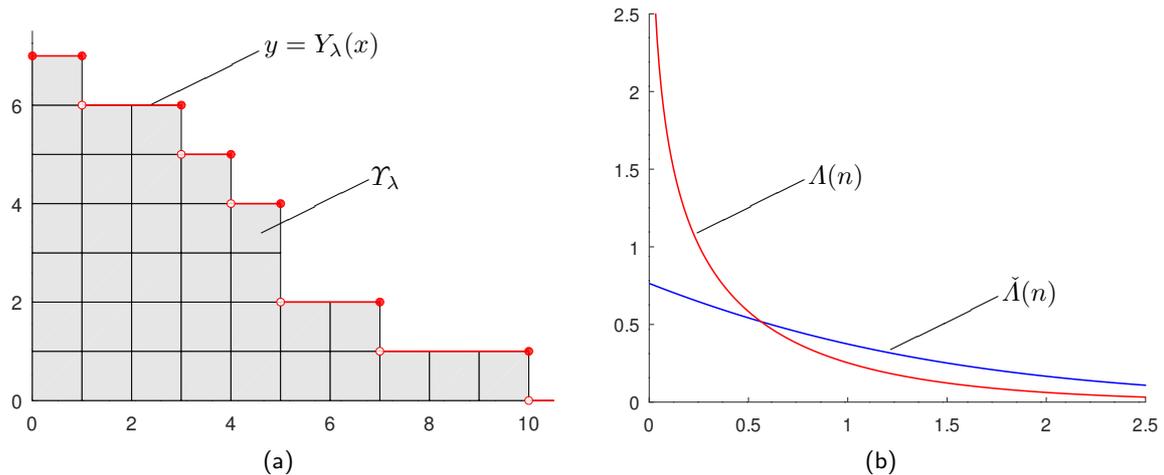

The geometric nature of Young diagrams makes it natural to pose a question of a possible typical behavior of their boundary as the ``size'' of partitions grows. This motivates the concept of \emph{limit shape}, which may be thought of as such a curve to which the bulk of the Young diagram boundaries are asymptotically close, of course upon a suitable scaling. More precisely, choosing some sequences $a_n\to\infty$, $b_n\to\infty$, define the scaled Young boundary
$$
\widetilde{Y}^{(n)}_\lambda(x):=b_n^{-1}\,Y_\lambda(a_n x) ,\qquad x\ge0.
$$
Then (the graph of) a function $y=\omega^*(x)$ is the \emph{limit shape} for partitions $\lambda\in\varLambda(n)$ as $n\to\infty$ if, for all $x>0$ and for any $\varepsilon>0$,
\begin{equation}\label{eq:LS-uniform}
\lim_{n\to\infty}\frac{\#\{\lambda\in\varLambda(n)\colon |\widetilde{Y}^{(n)}_\lambda(x)-\omega^*(x)|\le \varepsilon\}}{\#\varLambda(n)}=1.
\end{equation}
It is well known that the limit shape $\omega^*$ exists under the scaling $a_n=b_n=\sqrt{n}\myp$ and is determined by the equation
\begin{equation}\label{eq:limit1}
\rme^{-x\myp\pi/\sqrt{6}}+\rme^{-y\myp\pi/\sqrt{6}}=1.
\end{equation}
It was first identified by Temperley \cite{Temperley} in connection with the equilibrium shape of a growing
crystal, and obtained more rigorously by Vershik (see \cite[p.\:30]{VK2}) using some asymptotic estimates from \cite{ST}. A different proof in
its modern form was outlined by Vershik \cite{Vershik1} and
elaborated by Pittel~\cite{Pittel}, both using a suitable randomization and multiplicative measures (see also a more recent survey in  \cite{EG}). For completeness, let us mention an alternative approach to the limit shape based on bijections \cite{Pak2} and a powerful variational method \cite{V0,Vershik1}, which also yields a large deviation principle \cite{DVZ}.

Note from \eqref{eq:limit1} that $\omega^*(0)=\infty$, indicating that the number of parts, $M_\lambda$, grows faster than $\sqrt{n}$. Indeed, as was shown by Erd\H{o}s and Lehner \cite{Erdos}, $M_\lambda\sim (2\pi)^{-1}\mypp\sqrt{6\myp n}\,\log n$ (in the sense similar to 
\eqref{eq:LS-uniform}).

For strict partitions $\lambda\in\check{\varLambda}(n)$, the limit shape (under the same scaling $a_n=b_n=\sqrt{n}$ and in the sense of definition  \eqref{eq:LS-uniform} adapted to $\check{\varLambda}(n)$) is  specified by the equation (see Vershik \cite{Vershik1})
\begin{equation}\label{eq:limit2}
\rme^{\myp y\myp\pi/\sqrt{12}}=1+\rme^{- x\myp \pi/\sqrt{12}}.
\end{equation}
Note that here the value at the origin is finite, $\omega^*(0)= \pi^{-1}\sqrt{12}\,\log 2$,
which implies that the number of parts $M_\lambda$ in a
typical strict partition $\lambda\in\check{\varLambda}(n)$ grows like $\pi^{-1}\log 2\mypp\sqrt{12\myp n}$ \cite{Erdos}. The graphs of the limit shapes \eqref{eq:limit1} and \eqref{eq:limit2} are shown in Figure \ref{fig:yddef}(b).

The limiting formula \eqref{eq:LS-uniform} and its version for strict partitions, as mentioned above, may be interpreted as convergence in probability, $\widetilde{Y}^{(n)}_\lambda(x)\stackrel{\mathrm{p}}{\longrightarrow}\omega^*(x)$, under the uniform probability measure on the corresponding partition spaces $\varLambda(n)$ or $\check{\varLambda}(n)$ (whereby all member partitions are assumed to be equally likely). 
A general question about limit shapes under alternative measures was pioneered by Vershik \cite{V0,Vershik1}. In the present paper, we study the limit shape under the Boltzmann measure (see Sections \ref{sec:2.4} and \ref{sec:5.2}).

\subsection{Integer partitions with constraints}\label{sec:2.3}
The main focus of our paper is on strict partitions with parts confined to be \emph{$q$-th powers of integers} (for some $q\in\NN$), which means that $\nu_\ell=0$ unless $\ell\in\NN^{q}$; we denote the corresponding set of partitions by $\varLambda^q$. Moreover, it is of interest to combine these two constraints by considering only strict partitions in $\varLambda^q$, leading to the subset that we naturally denote by $\check{\varLambda}^q$. 

A general approach to introducing certain constraints in the partition spaces can be described as follows. Fix 
non-empty integer sets $\mathbb{A}\subseteq\mathbb{N}$ and $\mathbb{B}_\ell\subseteq\mathbb{N}_0$ ($\ell\in\mathbb{A}$), 
assuming that $0\in
\mathbb{B}_\ell$ but $\mathbb{B}_\ell\ne\{0\}$, for each $\ell\in\mathbb{A}$.
The set $\mathbb{A}$ specifies the source of permissible parts $\ell\in\NN$. For example, if $\mathbb{A}=2\mypp\NN_0+1$ then all parts must be odd integers, or if $\mathbb{A}=\mathbb{N}^{q}$ with some $q\in\mathbb{N}$ then only perfect $q$-th powers can be used; if $\mathbb{A}=\NN$ then any (positive) integer part is allowed. The sets $(\mathbb{B}_\ell)$ specify the allowed values of multiplicities $\nu_\ell\in\mathbb{B}_\ell$ for each part $\ell\in \mathbb{A}$.
For example, if $\mathbb{B}_\ell=\{0,1\}$ then 
part $\ell$ can be used no more than once; if $\mathbb{B}_\ell=\NN_0$ then $\nu_\ell$ is not constrained.

Given the sets $ \mathbb{A}$ and $(\mathbb{B}_\ell,\ell\in\mathbb{A})$, we use a generic notation $\tilde{\varLambda}$ to  denote the set of integer partitions satisfying the constraints imposed by $ \mathbb{A}$ and $(\mathbb{B}_\ell)$ as described above,
\begin{equation}\label{eq:tilde-Lambda}
\tilde{\varLambda}:=\{\lambda=(\ell^{\myp\nu_\ell})\in\varLambda\colon \,\ell\in  \mathbb{A},\,\nu_\ell\in   \mathbb{B}_\ell\}.
\end{equation}
In this loose notation we take the liberty to omit the explicit reference to $ \mathbb{A}$ and $(\mathbb{B}_\ell)$, which should cause no confusion. When the specific choice of $\mathbb{A}$ and $(\mathbb{B}_\ell)$ becomes important (in Section \ref{sec:3} below),  this will be clarified. For a partition $\lambda\in\tilde{\varLambda}$, we keep using the notation $N_\lambda$ and $M_\lambda$ for its weight and length, respectively, which are now given by (cf.\ \eqref{eq:NM})
\begin{equation}\label{eq:NM-constr}
N_\lambda= \sum_{\ell\in\mathbb{A}}\ell\mypp\nu_\ell\myp,\qquad M_\lambda=\sum_{\ell\in\mathbb{A}} \nu_\ell\myp.
\end{equation}   

\subsection{Boltzmann distributions}\label{sec:2.4}

The general idea of the \emph{Boltzmann distribution} as a probability measure on a decomposable combinatorial structure $\mathcal{C}=\{c\}$ (such as the set of all integer partitions $\varLambda$ or its constrained versions, e.g., the set of strict partitions  $\check{\varLambda}$) is that it is defined by picking some additive  structural features of the elements in $\mathcal{C}$ (such as weight and/or length of a partition) and making the probability of the element $c\in\mathcal{C}$ depend only on those features in a ``geometric'' fashion. Below, this idea is made precise for the class $\tilde{\varLambda}$ of integer partitions (see \eqref{eq:tilde-Lambda}) with constraints on the source of parts (via set $ \mathbb{A}$) and their multiplicities (via sets $(\mathbb{B}_\ell)$). 

\begin{definition}\label{def2'} Suppose that the constraining sets $ \mathbb{A}$ and $(\mathbb{B}_\ell)$ are fixed, and consider the corresponding partition space $\tilde{\varLambda}$ defined in \eqref{eq:tilde-Lambda}. Given a two-dimensional parameter $\bm{z}=(z_1,z_2)$, with $0<z_1<1$ and $0<z_2<1/z_1$,
the \emph{Boltzmann distribution} on $\tilde{\varLambda}$ is defined by  the formula
\begin{equation}\label{Boltzmann}
\QQ_{\bm{z}}(\lambda)= \frac{z_1^{N_\lambda} z_2^{M_\lambda}}{F(\bm{z})},\qquad \lambda\in \tilde{\varLambda},
\end{equation}
with the normalizing factor
\begin{equation}\label{eq:Fnorm}
F(\bm{z})=\sum_{\lambda\in \tilde{\varLambda}} z_1^{N_\lambda}z_2^{M_\lambda}.
\end{equation}
\end{definition}
Considering the constituent subspaces
\begin{equation}\label{eq:Lambda_nm}
\tilde{\varLambda}(n,m):=\{\lambda\in\tilde{\varLambda}\colon N_\lambda=n,M_\lambda=m\},
\end{equation}
the function $F(\bm{z})$
can be expressed as a double power series, 
\begin{equation}\label{eq:F}
F(\bm{z})=\sum_{n,m}  F_{n,m}\myp z_1^nz_2^m,
\end{equation}
where
$F_{n,m}=\#\tilde{\varLambda}(n,m)$.
In particular, the ``initial'' values with $n=0$ are reduced to
\begin{equation}\label{eq:F_nm}
F_{0,0}=1,\qquad F_{0,m}=0 \ \ \  (m\ge1).
\end{equation}
Of course, $F_{n,m}>0$ if and only if the condition $\{N_\lambda=n, M_\lambda=m\}$ is realizable, that is, if $\tilde{\varLambda}(n,m)\ne\varnothing$. If the focus is on the specific weight $N_\lambda=n$ or length $M_\lambda=m$ alone, this corresponds to the ``marginal'' subspaces
\begin{align}\label{eq:Lambda_n}
\tilde{\varLambda}(n,\myn\sbullet):={}&\bigcup_{m\le n} \tilde{\varLambda}(n,m)=\{\lambda\in\tilde{\varLambda}\colon N_\lambda=n\},\\
\label{eq:Lambda_m}\tilde{\varLambda}(\sbullet\myp,m):={}&\bigcup_{n\ge m} \tilde{\varLambda}(n,m)=\{\lambda\in\tilde{\varLambda}\colon M_\lambda=m\}.
\end{align}

The joint distribution of $N_\lambda$ and $M_\lambda$ under the Boltzmann measure \eqref{Boltzmann} is given by
\begin{equation}\label{eq:N=,M=}
\QQ_{\bm{z}}(N_\lambda=n,M_\lambda=m)=\QQ_{\bm{z}}(\tilde{\varLambda}(n,m))=\frac{F_{n,m}\myp z_1^n z_2^m}{F(\bm{z})},
\end{equation}
with the marginals 
\begin{align}\label{eq:NorMN}
\QQ_{\bm{z}}(N_\lambda=n)&=\QQ_{\bm{z}}(\tilde{\varLambda}(n,\myn\sbullet))=\frac{z_1^n}{F(\bm{z})}\sum_{m\le n} F_{n,m}\myp z_2^m,\\ \QQ_{\bm{z}}(M_\lambda=m)&=\QQ_{\bm{z}}(\tilde{\varLambda}(\sbullet\myp,m))=\frac{z_2^m}{F(\bm{z})}\sum_{n\ge m} F_{n,m}\myp z_1^n.
\label{eq:NorMM}
\end{align}

\begin{remark}\label{rm:=0}
For the empty partition $\lambda_{\varnothing}\vdash 0$ formally
associated with the null configuration $\nu_\ell\equiv0$, formula \eqref{Boltzmann}
yields $\QQ_{\bm{z}}(\lambda_{\varnothing})=1/F(\bm{z})>0$. On the other hand, $\QQ_{\bm{z}}(\lambda_{\varnothing})<1$, since
$F(\bm{z})>F(0,0)=1$ for $z_1>0,z_2>0$.
\end{remark}

The following result describes the Boltzmann distribution \eqref{Boltzmann} in terms of the joint distribution of the multiplicities $(\nu_\ell)$. As a
a by-product, it provides a multiplicative  representation of the generating function $F(\bm{z})$.
\begin{lemma}\label{pr:mult}
Under the Boltzmann measure $\QQ_{\bm{z}}$ on the generic partition space $\tilde{\varLambda}$ defined in (\ref{eq:tilde-Lambda}), the random multiplicities $(\nu_\ell,\,\ell\in  \mathbb{A})$ are mutually independent, with  marginal distributions 
\begin{equation}\label{eq:marginal}
\QQ_{\bm{z}}(\nu_\ell=k)=\frac{z_1^{\ell k}z_2^k}{F_\ell(\bm{z})},\qquad k\in  \mathbb{B}_\ell,
\end{equation}
where
\begin{equation}\label{eq:F-ell}
F_\ell(\bm{z})=\sum_{k\in  \mathbb{B}_\ell} z_1^{\ell k}z_2^k,\qquad \ell\in  \mathbb{A}.
\end{equation}
In particular, the generating function $F(\bm{z})$ admits the following product representation,
\begin{equation}\label{eq:F_product}
F(\bm{z})=\prod_{\ell\in  \mathbb{A}} F_\ell(\bm{z}).
\end{equation}
\end{lemma}
\begin{proof}
It suffices to verify that the product measure $\tilde{\QQ}_{\bm{z}}$ on $\tilde{\varLambda}$ with marginals \eqref{eq:marginal} is consistent with the definition \eqref{Boltzmann}. Let a partition $\lambda\in \tilde{\varLambda}$ be specified by the sequence of multiplicities $k_\ell\in\mathbb{B}_\ell$ ($\ell\in  \mathbb{A}$). 
Due to
independence of $(\nu_\ell)$ under $\tilde{\QQ}_{\bm{z}}$ and formula \eqref{eq:marginal}, we have on account of expressions \eqref{eq:NM-constr},
\begin{align}
\notag
\tilde{\QQ}_{\bm{z}}(\lambda)=\prod_{\ell\in  \mathbb{A}} \tilde{\QQ}_{\bm{z}}(\nu_\ell=k_\ell)&=\prod_{\ell\in  \mathbb{A}} \frac{z_1^{\ell k_\ell}z_2^{k_\ell}}{F_\ell(\bm{z})}   \\
&=\frac{z_1^{\sum_{\ell\in  \mathbb{A}} \ell k_\ell}z_2^{\sum_{\ell\in  \mathbb{A}} k_\ell}}{\prod_{\ell\in  \mathbb{A}} F_\ell(\bm{z})}=\frac{z_1^{N_\lambda} z_2^{M_\lambda}}{\widetilde{F}(\bm{z})},
\label{Boltzmann-prod}
\end{align}
where $\widetilde{F}(\bm{z}):=\prod_{\ell\in  \mathbb{A}} F_\ell(\bm{z})$. Since the probability distributions \eqref{Boltzmann} and \eqref{Boltzmann-prod} on the same space $\tilde{\varLambda}$ appear to be proportional to one another, it follows that the normalization factors $F(\bm{z})$ and $\widetilde{F}(\bm{z})$ coincide, which proves the product representation \eqref{eq:F_product}.
\end{proof}
\begin{remark}
Since $0\in\mathbb{B}_\ell$, we have a lower bound
\begin{equation}\label{eq:F>}
F_\ell(\bm{z})=\sum_{k\in \mathbb{B}_\ell} z_1^{\ell k}z_2^k\ge 1,\qquad \ell\in  \mathbb{A}.
\end{equation}
\end{remark}

The next lemma ensures that a random partition generated according to the Boltzmann distribution $\QQ_{\bm{z}}$ is a.s.\ finite, so that $\QQ_{\bm{z}} (M_\lambda<\infty)=1$.
\begin{lemma}\label{lm:lemma2} Under the probability measure $\QQ_{\bm{z}}$, the number of nonzero terms in the sequence of multiplicities $(\nu_\ell)$ is a.s.\ finite.
\end{lemma}
\begin{proof}
By the Borel--Cantelli lemma (see, e.g., \cite[Sec.\,II.10, p.\,255]{Shiryaev}), 
it suffices to check that 
$$
\sum_{\ell\in  \mathbb{A}}\QQ_{\bm{z}}(\nu_\ell>0)<\infty.
$$
Noting that $z_1^\ell z_2\le z_1z_2<1$ (see Definition~\ref{def2'}), we have
$$
F_\ell(\bm{z})=\sum_{k\in  \mathbb{B}_\ell} z_1^{\ell k}z_2^k\le \sum_{k=0}^\infty (z_1^{\ell}z_2)^k=\frac{1}{1-z_1^\ell z_2}.
$$
Hence, from the distribution formula \eqref{eq:marginal} we get 
\begin{align*}
\QQ_{\bm{z}}(\nu_\ell>0)&=1-\frac{1}{F_\ell(\bm{z})}\le 1-\bigl(1-z_1^\ell z_2\bigr)=z_1^\ell z_2,
\end{align*}
and therefore
$$
\sum_{\ell\in  \mathbb{A}}\QQ_{\bm{z}}(\nu_\ell>0)\le z_2\sum_{\ell\in  \mathbb{A}}z_1^\ell\le z_2\sum_{\ell=1}^\infty z_1^\ell=\frac{z_1z_2}{1-z_1}<\infty,
$$
as required.
\end{proof}

\subsection{Conditional Boltzmann distributions} \label{sec:2.5}

Recall that the Boltzmann distribution $\QQ_{\bm{z}}$ (see \eqref{Boltzmann}) is defined on the partition space $\tilde{\varLambda}$ subject to the tacit constraints determined by the sets $\mathbb{A}$ and $(\mathbb{B}_\ell,\ell\in \mathbb{A})$, as described in Section~\ref{sec:2.1}.  
It is sometimes useful to impose further constraints on permissible partitions and consider the arising conditional projections of the original measure $\QQ_{\bm{z}}$ onto the corresponding partition subspaces. To be specific, let $\tilde{\varLambda}^\dag\subseteq \tilde{\varLambda}$  and consider the   conditional measure supported on the space $\tilde{\varLambda}^\dag$,
\begin{equation}\label{eq:B-cond+}
\QQ_{\bm{z}}^\dag(\lambda):=\QQ_{\bm{z}}(\lambda\mypp|\myp\tilde{\varLambda}^\dag)=\frac{\QQ_{\bm{z}}(\lambda)}{\QQ_{\bm{z}}(\tilde{\varLambda}^\dag)},\qquad \lambda\in\tilde{\varLambda}^\dag.
\end{equation}
\begin{lemma}\label{lm:B-cutoff0}
The measure \eqref{eq:B-cond+} coincides with the Boltzmann measure on the partition space $\tilde{\varLambda}^\dag$.
\end{lemma}
\proof
By the definition \eqref{Boltzmann}, for any partition  $\lambda\in\tilde{\varLambda}^\dag$ we have
\begin{equation}\label{eq:F-trunc+}
\QQ_{\bm{z}}(\lambda\mypp|\myp\tilde{\varLambda}^\dag)=\frac{z_1^{N_\lambda}z_2^{M_\lambda}/F(\bm{z})}{F^\dag(\bm{z})/F(\bm{z})}=\frac{z_1^{N_\lambda}z_2^{M_\lambda}}{F^\dag(\bm{z})},\qquad \lambda\in\tilde{\varLambda}^\dag,
\end{equation}
where
\begin{equation*}
F^{\dag}(\bm{z})=
\sum_{\lambda\in \tilde{\varLambda}^\dag} z_1^{N_\lambda}z_2^{M_\lambda}.
\end{equation*}
Again referring to \eqref{Boltzmann} together with  \eqref{eq:Fnorm}, we see that formula \eqref{eq:F-trunc+} defines the Boltzmann distribution on $\tilde{\varLambda}^\dag$ with the same parameters $\bm{z}=(z_1,z_2)$. 
\endproof

In fact, the Boltzmann measure \eqref{Boltzmann} on the  partition space $\tilde{\varLambda}$ constrained by means of the sets $\mathbb{A}$ and $(\mathbb{B}_\ell)$ as described in Definition \ref{def2'}, can itself be identified with the conditional Boltzmann measure projected from the full partition space $\varLambda$ via conditioning on $\lambda\in\tilde{\varLambda}$. The following  particular cases of further constraints on the space $\tilde{\varLambda}$ are also of interest: 
\begin{alignat*}{3}
\tilde{\varLambda}_K:={}&\{\lambda\in\tilde{\varLambda}\colon \,\max \nu_\ell\le K\}&&\text{(multiplicities bounded by $K$)};\\
\tilde{\varLambda}_L:={}&\bigl\{\lambda\in\tilde{\varLambda}\colon \max\lambda_i\le L\bigr\} &&\text{(parts bounded by $L$)};\\
\tilde{\varLambda}_M:={}&\{\lambda\in\tilde{\varLambda}\colon \textstyle \sum\nolimits _{\ell} \nu_\ell=M\} &&\text{(fixed length $M_\lambda=M$)};\\ \tilde{\varLambda}_N:={}&\bigl\{\lambda\in\tilde{\varLambda}\colon \textstyle\sum\nolimits_{\ell} \ell\myp\nu_\ell=N\bigr\}&\quad& \text{(fixed weight $N_\lambda=N$)}.
\end{alignat*}

Note that the result of Lemma \ref{pr:mult} about mutual independence and marginal distributions of the multiplicities $(\nu_\ell)$ remains true for the first two examples, $\tilde{\varLambda}_K$ and $\tilde{\varLambda}_L$, simply because they follow our basic Definition \ref{def2'} with the constraining sets  $\mathbb{B}_\ell$ or $\mathbb{A}$ replaced by $\mathbb{B}_{\ell,K}=\{k\in\mathbb{\mathbb{B}_\ell}\colon k\le K\}$ or $\mathbb{A}_L=\{\ell\in\mathbb{A}\colon \ell\le L\}$, \strut{}respectively.
However, this result does not apply to the Boltzmann measures on the spaces $\tilde{\varLambda}_M=\tilde{\varLambda}(\sbullet\myp,\myn M)$ or $\tilde{\varLambda}_N=\tilde{\varLambda}(N,\myn\sbullet)$ (see \eqref{eq:Lambda_n} and \eqref{eq:Lambda_m}), which take a reduced form by ``burning out'' the parameters $z_1$ or $z_2$, respectively:
\begin{align}\label{eq:uniM}
\QQ_{\bm{z}}(\lambda\mypp|\mypp\tilde{\varLambda}_M)& =\frac{z_1^{N_\lambda}}{\sum_{n\ge M}F(n,M)\mypp z_1^n},\qquad \lambda\in\tilde{\varLambda}_M,\\
\QQ_{\bm{z}}(\lambda\mypp|\mypp\tilde{\varLambda}_N)& =\frac{z_2^{M_\lambda}}{\sum_{m\le N}F(N,m)\mypp z_2^m},\qquad \lambda\in\tilde{\varLambda}_N.
\label{eq:uniN}
\end{align}

Individual constraints such as listed above can be combined. An important example is considered in the next lemma, stating that the Boltzmann distribution conditioned on both $N_\lambda$ and $M_\lambda$ is reduced to a uniform measure on the corresponding partition subspace.
\begin{lemma}\label{uniformity}
Let $n,m\in\NN_0$ be such that the conditions $N_\lambda=n$ and $M_\lambda=m$ are compatible with the constraining sets $\mathbb{A}$ and $ (\mathbb{B}_\ell)$, that is, $\tilde{\varLambda}(n,m)\ne\varnothing$.
Then 
\begin{equation}\label{eq:uni}
\QQ_{\bm{z}}\bigl(\lambda\mypp|\myp\tilde{\varLambda}(n,m)\bigr) =\frac{1}{\#\tilde{\varLambda}(n,m)},\qquad \lambda\in\tilde{\varLambda}(n,m).
\end{equation}
\end{lemma}
\begin{proof}
By virtue of  Lemma \ref{lm:B-cutoff0}, it suffices to observe that the Boltzmann distribution on $\tilde{\varLambda}(n,m)$ is uniform, because $\QQ_{\bm{z}}(\lambda)\propto z_1^nz_2^m=\mathrm{const}$ for all $\lambda\in\tilde{\varLambda}(n,m)$.
\end{proof}

Thanks to Lemma \ref{uniformity}, the Boltzmann distribution can be used for enumeration purposes.
\begin{corollary}\label{cor:uni}
The following representation holds with any $\bm{z}=(z_1,z_2)$ such that $0<z_1<1$ and $z_2<1/z_1$,
\begin{equation}\label{eq:uni-space}
F_{n,m}\equiv\#\tilde{\varLambda}(n,m)= \frac{F(\bm{z})\,\QQ_{\bm{z}}(N_\lambda=n,M_\lambda=m)}{z_1^n z_2^m}.
\end{equation}
\end{corollary}

To use formula \eqref{eq:uni-space} in practice, the parameters $z_1,z_2$ are usually calibrated so as to make the events $\{N_\lambda=n,M_\lambda=m\}$ ``likely'' under the Boltzmann distribution $\QQ_{\bm{z}}$ when $n$ and $m$ are close to some target values $\langle N\rangle$ and $\langle M\rangle$, respectively (treated as \emph{hyper-parameters}). In this paper, we pursue the standard approach (see, e.g.,  \cite{Fristedt}, \cite{VY1},  \cite{VY2}, \cite{Romik},  \cite{Bridges}, \cite{Bogachev}) based on making the \emph{expected values} of $N_\lambda$ and $M_\lambda$ consistent with the prescribed values $\langle N\rangle$ and $\langle M\rangle$, respectively,
\begin{equation}\label{braket0}
  \EE_{\bm{z}}(N_\lambda)=\langle N\rangle, \qquad   \EE_{\bm{z}}(M_\lambda)=\langle M\rangle.
\end{equation} 
Using \eqref{eq:NM-constr} and \eqref{eq:marginal}, conditions \eqref{braket0} are rewritten more explicitly as a system of equations for the parameters $z_1$ and $z_2$,
\begin{equation}\label{eq:calib=NM}
\left\{
\begin{aligned}
\sum_{\ell\in\mathbb{A}} \ell\,\sum_{k\in\mathbb{B}_\ell}\frac{k\mypp z_1^{\ell k} z_2^{k}}{F_\ell(\bm{z})}&=\braket{N}\!,\\
\sum_{\ell\in\mathbb{A}} \sum_{k\in\mathbb{B}_\ell}\frac{k\mypp z_1^{\ell k} z_2^{k}}{F_\ell(\bm{z})}&=\braket{M}\!.
\end{aligned}
\right.
\end{equation}
Solving such a system exactly is usually beyond reach but an asymptotic analysis may be feasible when one or both of the hyper-parameters $\braket{N}$ and $\braket{M}$ are large (see Section \ref{sec:3.3} below).

Truncation of the source of parts, thus reducing the partition space to $\tilde{\varLambda}_L$ may be useful in the design of Boltzmann samplers with the aim to avoid indefinite computation (see Section~\ref{sec:6}). Clearly, if $L<\max\mathbb{A}$ then the truncation leads to a distortion of the original  Boltzmann distribution $\QQ_{\bm{z}}$; in particular, it will cause a negative bias between the target hyper-parameters $\braket{N}$ and $\braket{M}$ used for calibration of $\QQ_{\bm{z}}$ (see \eqref{eq:calib=NM}) and the truncated expected values,
\begin{align}
\label{eq:ExpN<}
\EE_{\bm{z}}^{L}(N_\lambda)&=\sum_{\ell\in\mathbb{A}_L} \ell\,\sum_{k\in\mathbb{B}_\ell}\frac{k\mypp z_1^{\ell k} z_2^{k}}{F_\ell(\bm{z})}<\sum_{\ell\in\mathbb{A}} \ell\,\sum_{k\in\mathbb{B}_\ell}\frac{k\mypp z_1^{\ell k} z_2^{k}}{F_\ell(\bm{z})}=\EE_{\bm{z}}(N_\lambda)=\braket{N}\!,\\
\EE_{\bm{z}}^{L}(M_\lambda)&=\sum_{\ell\in\mathbb{A}_L} \sum_{k\in\mathbb{B}_\ell}\frac{k\mypp z_1^{\ell k} z_2^{k}}{F_\ell(\bm{z})}<\sum_{\ell\in\mathbb{A}} \sum_{k\in\mathbb{B}_\ell}\frac{k\mypp z_1^{\ell k} z_2^{k}}{F_\ell(\bm{z})}=\EE_{\bm{z}}(M_\lambda)=\braket{M}\!.
\label{eq:ExpM<}
\end{align}
However, if the probability of the condition $\{\lambda\in\tilde{\varLambda}_L\}$ is large enough, then the two distributions are close to one another in total variation, which is justified by the following elementary estimate.
\begin{lemma}\label{lm:TV-appr}
Let $(\Omega,\mathcal{F},\PP)$ be a probability space, and let an event $A\in\mathcal{F}$ be such that\/ $\PP(A)\ge 1-\delta$ for some (small) $\delta>0$. Then for any event $B\in\mathcal{F}$,
$$
|\PP(B\mypp|\myp A)-\PP(B)|\le\frac{2\delta}{1-\delta}. 
$$
\end{lemma}
\begin{proof}
Denoting $A^c=\Omega\setminus A$, we have 
\begin{align*}
|\PP(B\mypp|\myp A)-\PP(B)|&=\frac{|\PP(B\cap A)-\PP(B)\mypp\PP(A)|}{\PP(A)} \\
&=\frac{|\PP(B\cap A)\mypp\PP(A^c)-\PP(B\cap A^c)\mypp\PP(A)|}{1-\PP(A^c)}\\
&\le \frac{2\,\PP(A^c)}{1-\PP(A^c)}\le \frac{2\myp\delta}{1-\delta},
\end{align*}
noting that $\PP(A^c)\le \delta$.
\end{proof}

A simple but useful version of the truncation idea adapted to the spaces $\tilde{\varLambda}(n,m)$ states that if the truncation threshold $L$ is high enough then the additional condition $\ell\le L$ does not affect the conditional uniformity stated in Lemma~\ref{uniformity}.

\begin{definition}\label{def:majorant}
Assuming that the partition space $\tilde{\varLambda}(n,m)$ is non-empty, denote by $L^*\mynn=L^*(n,m)$ any \emph{majorant} of parts involved in partitions $\lambda$ belonging to  this space, that is,
\begin{equation}\label{eq:ell-max}
L^*\!\ge \max\mynn\{\lambda_1\colon \lambda=(\lambda_i)\in\tilde{\varLambda}(n,m)\}.
\end{equation}
In terms of multiplicities $(\nu_\ell)$ encoding partitions $\lambda\in\tilde{\varLambda}(n,m)$, condition \eqref{eq:ell-max} is equivalent to 
saying that
$\nu_\ell\equiv 0$ for all $\ell>L^*\mynn$.
\end{definition}
For example, a loose majorant is provided by  $L^*\mynn=n$; this is actually sharp if $m=1$. In general, a sharp majorant is given by
\begin{equation}\label{eq:l-tight}
L^*\mynn=\max\mynn\{\ell\in\mathbb{A}\colon \tilde{\varLambda}(n-\ell,m-1)\ne\varnothing\}.
\end{equation}
Note that formula \eqref{eq:l-tight} holds in the boundary case $m=1$ due to our convention in Section \ref{sec:2.1}, effectively stating that the set $\tilde{\varLambda}(0,0)$ is not empty by containing a (single)  partition $\lambda_\varnothing = (0,0, \dots)$.

The following fact is self-evident by observing that $\tilde{\varLambda}(n,m)\cap \tilde{\varLambda}_{L^*}=\tilde{\varLambda}(n,m)$ and in view of Lemma~\ref{uniformity}.
\begin{lemma}\label{uniformity-gen}
Suppose that  $\tilde{\varLambda}(n,m)\ne\varnothing$ and let $L^*$ be a majorant as in Definition \ref{def:majorant}. 
Then 
\begin{equation}\label{eq:uni1}
\QQ_{\bm{z}}(\lambda\,|\, \tilde{\varLambda}(n,m)\cap \tilde{\varLambda}_{L^*})=\frac{1}{\#\tilde{\varLambda}(n,m)},\qquad \lambda\in\tilde{\varLambda}(n,m).
\end{equation}
\end{lemma}

This lemma can be utilized in random sampling of integer partitions based on the Boltzmann distribution. Indeed, choosing a suitable majorant $L^*$ and building a random partition in $\tilde{\varLambda}(n,m)$ by iteratively sampling the multiplicities $\nu_\ell$ with $\ell\in\mathbb{A}_{L^*}$ until the target conditions \eqref{eq:Lambda_nm} are satisfied, the resulting partition $\lambda=(\ell^{\myp\nu_\ell})$ will be  uniformly sampled from $\tilde{\varLambda}(n,m)$, according to \eqref{eq:uni1}. We will return to these issues in Section~\ref{sec:6}.

\subsection{Second-order moments of the partition weight and length}\label{sec:2.6}

Consider the covariance matrix of the vector $(N_\lambda,M_\lambda)$ under the Boltzmann measure $\QQ_{\bm{z}}$,
\begin{equation}\label{eq:K(z)}
\bm{K}(\bm{z}):=\begin{pmatrix}
\Var_{\bm{z}}(N_\lambda)&\Cov_{\bm{z}}(N_\lambda,M_\lambda)\\[.2pc]
\Cov_{\bm{z}}(N_\lambda,M_\lambda)&\Var_{\bm{z}}(M_\lambda)
\end{pmatrix}\!,\qquad \bm{z}=(z_1,z_2).
\end{equation}
As such, the matrix $\bm{K}(\bm{z})$ is automatically positive semi-definite; moreover, it is \emph{positive definite} provided that the set $\mathbb{A}$ of permissible parts contains at least two elements,  $\#\mathbb{A}\ge 2$. To this effect, since both $\Var_{\bm{z}}(N_\lambda)>0$ and $\Var_{\bm{z}}(M_\lambda)>0$, we only need to check that $\det \bm{K}(\bm{z})>0$, that is, the underlying Cauchy inequality is strict, $\bigl|\Cov_{\bm{z}}(N_\lambda,M_\lambda)\bigr|<\sqrt{\Var_{\bm{z}}(N_\lambda)\,\Var_{\bm{z}}(M_\lambda)}\myp$. Indeed, otherwise the random variables $N_\lambda$ and $M_\lambda$ would be linearly dependent, that is, with some deterministic constants $c_1,c_2\in\RR$ 
$$
c_1N_\lambda+c_2M_\lambda=\sum_{\ell\in\mathbb{A}} \left(c_1\ell+c_2\right)\nu_\ell\equiv\mathrm{const}\qquad \text{($\QQ_{\bm{z}}$-a.s.)}.
$$
But this is impossible if $\#\mathbb{A}\ge2$, since $(\nu_\ell)$ are mutually independent and $\nu_\ell\not\equiv \mathrm{const}$ ($\QQ_{\bm{z}}$-a.s.).

A single-parameter version of the next lemma is well known (see, e.g., 
\cite[Proposition 2.1]{Duchon} or \cite[formula (2.2), p.\mypp 110]{BeFaRa}). An extension to the general multi-parametric case is stated,
with a sketch proof, in \cite[Proposition 7, pp.\,772--773]{BBD}. For convenience, we give a direct proof in the general case of a constrained partition space $\tilde{\varLambda}$ described in Definition \ref{def2'}.
\begin{lemma}\label{lm:Phi}
For $\bm{s}=(s_1,s_2)$, denote\/ $\rme^{\bm{s}}\mynn:=(\rme^{s_1}\mynn,\myp \rme^{s_2})$.
Define the function
\begin{equation}\label{eq:Phi}
\varPhi(\bm{s}):=\log F(\rme^{\bm{s}}),\qquad s_1,s_2<0,
\end{equation}
where $F(\bm{z})$ is introduced in (\ref{eq:Fnorm}). Then
\begin{equation}\label{eq:partialPhi}
\frac{\partial\myp \varPhi}{\partial s_1}=\EE_{\bm{z}}(N_\lambda)\bigr|_{\bm{z}=\rme^{\bm{s}}},\qquad \frac{\partial\myp \varPhi}{\partial s_2}=\EE_{\bm{z}}(M_\lambda)\bigr|_{\bm{z}=\rme^{\bm{s}}},
\end{equation}
Moreover, the Hessian of\/ $\varPhi(\bm{s})$ is expressed as follows,
\begin{equation}\label{eq:K(exp)}
\left(\frac{\partial^2\varPhi}{\partial s_i\myp\partial s_j}\right)=\bm{K}(\bm{z})\bigr|_{\bm{z}=\rme^{\bm{s}}},
\end{equation}
where $\bm{K}(\bm{z})$ is the covariance matrix defined in (\ref{eq:K(z)}). 
\end{lemma}
\begin{proof}
Differentiating \eqref{eq:F} and using formula \eqref{eq:NorMN}, we get
\begin{align*}
z_1\myp\frac{\partial F}{\partial z_1}&=z_1\sum_{n=0}^\infty n\myp z_1^{n-1}\mynn\sum_{m=0}^n F_{n,m}\myp z_2^m\\
&=F(\bm{z})\sum_{n=0}^\infty n\, \QQ_{\bm{z}}(N_\lambda=n)=F(\bm{z})\:\EE_{\bm{z}}(N_\lambda).
\end{align*}
Similarly, using \eqref{eq:NorMM},
\begin{align*}
z_2\mypp\frac{\partial F}{\partial z_2}&=z_2\sum_{m=0}^\infty m\mypp z_2^{m-1}\mynn\sum_{n=m}^\infty F_{n,m}\myp z_1^n\\
&=F(\bm{z})\sum_{m=0}^\infty m\, \QQ_{\bm{z}}(M_\lambda=m)=F(\bm{z})\:\EE_{\bm{z}}(M_\lambda).
\end{align*}
Hence, the chain rule applied to  \eqref{eq:Phi} yields
\begin{align*}
\frac{\partial\myp \varPhi}{\partial s_1}=
\rme^{s_1}\!\left.\frac{\partial F /\partial z_1 }{F(\bm{z})}\right|_{\bm{z}=\rme^{\bm{s}}}\! = \EE_{\bm{z}}(N_\lambda)\bigr|_{\bm{z}=\rme^{\bm{s}}},\\[.3pc]
\frac{\partial\myp \varPhi}{\partial s_2}=
\rme^{s_2}\!\left.\frac{\partial F /\partial z_2 }{F(\bm{z})}\right|_{\bm{z}=\rme^{\bm{s}}}\! = \EE_{\bm{z}}(M_\lambda)\bigr|_{\bm{z}=\rme^{s}},
\end{align*}
which proves formulas \eqref{eq:partialPhi}.

Likewise, considering second-order partial derivatives, we obtain
\begin{align*}
z_1^2\myp\frac{\partial ^2 F}{\partial z_1^2}&=z_1^2\sum_{n=0}^\infty n\myp(n-1)\myp z_1^{n-2}\mynn\sum_{m=0}^n F_{n,m}\myp z_2^m\\
&=F(\bm{z})\sum_{n=0}^\infty n\myp(n-1)\, \QQ_{\bm{z}}(N_\lambda=n)\\
&=F(\bm{z})\:\EE_{\bm{z}}\bigl(N^2_\lambda-N_\lambda\bigr).
\end{align*}
Therefore,
\begin{align}
\notag
\frac{\partial^2 \varPhi}{\partial s_1^2}&=\frac{\partial}{\partial s_1}\!\left(\rme^{s_1}\!\left.\frac{\partial F /\partial z_1 }{F(\bm{z})}\right|_{\bm{z}=\rme^{\bm{s}}}\right)\\
\notag
&=\rme^{s_1}\!\left.\frac{\partial F /\partial z_1 }{F(\bm{z})}\right|_{\bm{z}=\rme^{\bm{s}}}\!+
\rme^{2s_1}\!\left.\frac{\partial^2 F /\partial z_1^2}{F(\bm{z})}\right|_{\bm{z}=\rme^{\bm{s}}}
\!-\rme^{2s_1}\!\left(\left.\frac{\partial F /\partial z_1 }{F(\bm{z})}\right|_{\bm{z}=\rme^{\bm{s}}}\right)^2\\
\notag
&=\Bigl(\EE_{\bm{z}}(N_\lambda)+\EE_{\bm{z}}\bigl(N_\lambda^2-N_\lambda\bigr)-\bigl(\EE_{\bm{z}}(N_\lambda)\bigr)^2\Bigr)_{\bm{z}=\rme^{\bm{s}}}\\
&=\Var_{\bm{z}}(N_\lambda)\bigr|_{\bm{z}=\rme^{\bm{s}}}.
\label{eq:VarN}
\end{align}
Similarly, one can show that 
\begin{equation}\label{eq:VarM}
\frac{\partial^2 \varPhi}{\partial s_2^2}=\Var_{\bm{z}}(M_\lambda)\bigr|_{\bm{z}=\rme^{\bm{s}}}.
\end{equation}
Furthermore, considering the mixed partial derivative, by means of formula \eqref{eq:N=,M=} we have
\begin{align*}
z_1z_2\myp\frac{\partial ^2 F}{\partial z_1\myp\partial z_2}&=z_1 z_2\sum_{n=0}^\infty \sum_{m=0}^n F_{n,m}\mypp n\myp z_1^{n-1}\myp  m\myp z_2^{m-1}\mypp \\[-.2pc]
&=F(\bm{z})\sum_{n=0}^\infty\sum_{m=0}^n n\myp m\, \QQ_{\bm{z}}(N_\lambda=n,M_\lambda=m)=F(\bm{z})\:\EE_{\bm{z}}\bigl(N_\lambda M_\lambda\bigr).
\end{align*}
Hence,
\begin{align}
\notag
\frac{\partial^2 \varPhi}{\partial s_1\myp\partial s_2}&=\frac{\partial}{\partial s_2}\!\left(\rme^{s_1}\!\left.\frac{\partial F /\partial z_1 }{F(\bm{z})}\right|_{\bm{z}=\rme^{\bm{s}}}\right)\\
\notag
&=
\rme^{s_1}\myp\rme^{s_2}\!\left.\frac{\partial^2 F /\partial z_1^2}{F(\bm{z})}\right|_{\bm{z}=\rme^{\bm{s}}}
\!-\rme^{s_1}\myp\rme^{s_2}\!\left(\left.\frac{\partial F /\partial z_1 }{F(\bm{z})}\cdot \frac{\partial F /\partial z_2 }{F(\bm{z})}\right)\right|_{\bm{z}=\rme^{\bm{s}}}\\
&=\Bigl(\EE_{\bm{z}}(N_\lambda M_\lambda)-\EE_{\bm{z}}(N_\lambda)\, \EE_{\bm{z}}(M_\lambda) \Bigr)_{\bm{z}=\rme^{\bm{s}}}=\Cov_{\bm{z}}(N_\lambda,M_\lambda)\bigr|_{\bm{z}=\rme^{\bm{s}}}.
\label{eq:CovNM}
\end{align}

Finally, it remains to notice that, collectively, formulas \eqref{eq:VarN}, \eqref{eq:VarM} and \eqref{eq:CovNM} prove the claim \eqref{eq:K(exp)}, which completes the proof of Lemma~\ref{lm:Phi}.
\end{proof}

\begin{corollary}\label{cor:unique}
Suppose that $\#\mathbb{A}\ge 2$ (see a comment after  definition (\ref{eq:K(z)})). Then the function $\varPhi(\bm{s})$ defined in (\ref{eq:Phi}) is strictly convex. 
Consequently, any system of equations in variables  $\bm{z}=(z_1,z_2)$ of the form (\ref{braket0}) has at most one solution. 
\end{corollary}
\begin{proof}
Convexity of $\varPhi(\bm{s})$ follows from the representation \eqref{eq:K(exp)} and the fact that the covariance matrix $\bm{K}(\bm{z})$ is positive definite.   
\end{proof}

\section{Strict power partitions}\label{sec:3}

\subsection{Basic formulas}
From now on, we consider the case $\mathbb{A}=\NN^{q}$, with a fixed $q\in\NN$, and $\mathbb{B}_\ell\equiv\{0,1\}$ {for all $\ell\in\NN^q$. This specification corresponds to strict integer partitions (i.e., with unequal parts) into perfect $q$-th powers. To highlight the choice of this model, in what follows we switch from the generic notation $\tilde{\varLambda}$ and  $\tilde{\varLambda}(n,m)$ to a more adapted notation $\check{\varLambda}^{q}$ and
$\check{\varLambda}^{q}(n,m)$, where the super-index $^{q}$ indicates that parts are $q$-th powers, while the ``check'' symbol $\,\check{}\,$ stands as a reminder that partitions are strict. 
{To conform with the conventional notation, for $q=1$  we will omit the super-index by writing $\check{\varLambda}$, 
$\check{\varLambda}(n,m)$, etc.} 

The next key lemma is but a specialization of the general Proposition \ref{pr:mult} to the case $\check{\varLambda}^{q}$.
\begin{lemma}\label{distnu}
Under the Boltzmann distribution on the space $\check{\varLambda}^{q}$, the random multiplicities  $(\nu_{\ell},\,\ell \in \mathbb{N}^q)$ are mutually independent and have a Bernoulli distribution with the corresponding parameter $z_1^{\ell}z_2\mypp(1+z_1^{\ell}z_2)^{-1}$,
\begin{equation}\label{eq:Bern}
\QQ_{\bm{z}}(\nu_{\ell}=0)= \frac{1}{1+z_1^{\ell}z_2},\qquad \QQ_{\bm{z}}(\nu_{\ell}=1)= \frac{z_1^{\ell}z_2}{1+z_1^{\ell}z_2}\qquad (\ell\in\NN^{q}).
\end{equation}
Furthermore,
the corresponding generating function is given by 
\begin{equation}\label{eq:F_q}
F(\bm{z}) = \prod_{\ell\in\NN^q} (1+z_1^{\ell}z_2)=\prod_{j=1}^{\infty}(1+z_1^{j^q}\! z_2).
\end{equation}
\end{lemma}
\begin{remark}
Here the Borel--Cantelli condition $\sum_\ell \QQ_{\bm{z}}(\nu_\ell>0)<\infty$ (see the proof of Lemma \ref{lm:lemma2}) specializes to
$$
\sum_{\ell\in\NN^q} \frac{z_1^{\ell}z_2}{1+z_1^{\ell}z_2}
\le z_2\sum_{\ell=1}^\infty z_1^{\ell}=\frac{z_1z_2}{1-z_1}.
$$
Thus, we need $0<z_1<1$ but no condition on $z_2>0$.
\end{remark}

\subsection{Sums asymptotics} 
We will be frequently using asymptotic formulas for certain sums over the integer set $\NN^q$. Of course, such results are well known for $q=1$. The analysis is greatly facilitated by the classical (first-order) \emph{Euler--Maclaurin summation formula} (see, e.g.,
\cite[\S12.2]{Cramer} or \cite[2.10.1]{NIST}),
conveniently written in the form\footnote{The Euler--Maclaurin formula is usually written using the (first-order) Bernoulli polynomial $B_1(x)=x-\frac12$ ($0\le x<1$).
Our notation yields simpler expressions. A similar remark applies to the second-order version \eqref{eq:EM2} below.}
\begin{equation}\label{EM0}
\sum_{j=1}^\infty f(j)=
\int_{0}^\infty\! f(x)\,\rmd{x} + \int_{0}^\infty\! \!
{B}^{\circ\vphantom{a_a}}_{1}\bigl(x-\lfloor x\rfloor\bigr)\mypp f^{\myp\prime}(x)\,\rmd{x},
\end{equation}
where ${B}^{\circ\vphantom{a_a}}_{1}(x):=x$ ($0\le x<1$), and both the series $\sum_{j=1}^\infty f(j)$ and the integral $\int_{0}^\infty \!f(x)\,\rmd{x}$ are assumed to converge. In particular, formula \eqref{EM0} gives a simple bound for the error arising from replacing the sum with the integral,
$$
\sum_{j=1}^\infty f(j)=
\int_{0}^\infty\!\! f(x)\,\rmd{x} + O(1)\!\int_{0}^\infty \!
|f^{\myp\prime}(x)|\,\rmd{x}.
$$
A more general ``indented'' version of the Euler--Maclaurin formula reads
\begin{equation}\label{EM1}
\sum_{j>j_*} f(j)=
\int_{j_*}^\infty\!\! f(x)\,\rmd{x} + \int_{j_*}^\infty \!\!
{B}^{\circ\vphantom{a_a}}_1\bigl(x-\lfloor x\rfloor\bigr)\mypp f^{\myp\prime}(x)\,\rmd{x},
\end{equation}
leading to the estimate
$$
\sum_{j>j_*}^\infty f(j)=
\int_{j_*}^\infty\!\! f(x)\,\rmd{x} + O(1)\!\int_{j_*}^\infty \!
|f^{\myp\prime}(x)|\,\rmd{x}.
$$

For most of our purposes, the first-order formulas will suffice. If needed, a refined estimate of the remainder term can be obtained using  a second-order Euler--Maclaurin formula (cf.\ \cite[2.10.1]{NIST}),
\begin{equation}\label{eq:EM2}
\sum_{j=1}^\infty f(j)=
\int_{0}^\infty\!\! f(x)\,\rmd{x}-\frac{f(0)}{2}+ \frac{1}{2}\myn\int_{0}^\infty \!
{B}^{\circ\vphantom{a_a}}_2\bigl(x-\lfloor x\rfloor\bigr)\mypp f^{\myp\prime\prime}(x)\,\rmd{x},
\end{equation}
where ${B}^{\circ\vphantom{a_a}}_2(x):=x-x^2$ ($0\le x<1$).
\begin{lemma}\label{lm:sum1}
Let $s\ge 0$ be fixed. Then, as $\gamma\to0^{+}$, 
\begin{equation}\label{eq:1st}
\sum_{\ell\in\NN^q} \ell^s\myp \rme^{-\gamma\ell}=\frac{\Gamma(s+1/q)}{q\,\gamma^{s+1/q}}\left(1+O\bigl(\gamma^{1/q}\bigr)\right)\!.
\end{equation}
\end{lemma}
\begin{proof}
Setting $f(x):=\ell^s\myp \rme^{-\gamma\ell}|_{\myp\ell=x^q}$, we have 
\begin{equation}\label{eq:f'}
f'(x)=\left.\frac{\rmd(\ell^s\myp \rme^{-\gamma\ell})}{\rmd{\ell}}\right|_{\myp\ell=x^q}\!\cdot (x^q)'=(s-\gamma\ell)\,\ell^{s-1}\myp\rme^{-\gamma\ell}|_{\myp\ell=x^q}\cdot q\,x^{q-1}.
\end{equation} 
Hence, the Euler--Maclaurin formula \eqref{EM0} yields
\begin{equation}\label{eq:N0}
\sum_{\ell\in\NN^q}\ell^s\myp \rme^{-\gamma\ell}= \sum_{j=1}^\infty j^{qs}\myp \rme^{-\gamma j^{q}}=\int_0^\infty \!x^{qs}\mypp \rme^{-\gamma\myp x^q}\rmd{x}+\Delta_s(\gamma),
\end{equation}
where the error term is given by
\begin{equation}\label{eq:Delta_s}
\Delta_s(\gamma)=q\mynn\int_0^\infty \!\mypp {B}^{\circ\vphantom{a_a}}_1(x)\,(s-\gamma\myp x^q)\,x^{qs-1}\mypp \rme^{-\gamma\myp x^q}\myp\rmd{x}.
\end{equation}
The integral in \eqref{eq:N0} is easily computed by the substitution $u=\gamma\myp x^q$,  
\begin{equation}
\label{eq:N1}
\int_0^\infty \!x^{qs}\mypp \rme^{-\gamma\myp x^q}\rmd{x}=\frac{1}{q\mypp \gamma^{s+1/q}}\int_0^\infty\!u^{s+1/q-1}\mypp\rme^{-u}\,\rmd{u}=\frac{\Gamma(s+1/q)}{q\mypp \gamma^{s+1/q}}.
\end{equation}

To estimate the error term $\Delta_s(\gamma)$, we have to consider the cases $s=0$ and $s>0$ separately. Using that $0\le {B}^{\circ\vphantom{a_a}}_1(x)\le 1$, from \eqref{eq:Delta_s} we obtain, via the same substitution $u=\gamma\myp x^q$, 
\begin{equation}\label{eq:Delta0}
0<-\Delta_0(\gamma)< q\myp\gamma\myn\int_0^\infty \!x^{q-1}\myp \rme^{-\gamma x^q}\myp\rmd{x}=\int_0^\infty \!\rme^{-u}\myp\rmd{u}=1.
\end{equation}
If $s>0$ then expression \eqref{eq:Delta_s} implies a two-sided inequality,  
\begin{equation}\label{eq:Delta1}
-q\myp \gamma\mynn\int_0^\infty \!x^{q\myp(s+1)-1}\myp \rme^{-\gamma x^q}\myp\rmd{x}< \Delta_s(\gamma)< q\myp s\mynn\int_0^\infty \!x^{qs-1}\myp \rme^{-\gamma x^q}\myp\rmd{x}.
\end{equation}
Computing the integrals as before, the bounds \eqref{eq:Delta1} specialize as follows,
\begin{equation}\label{eq:Delta2}
-\frac{\Gamma(s+1)}{\gamma^s}<\Delta_s(\gamma)< \frac{s\mypp\Gamma(s)}{\gamma^s}
= \frac{\Gamma(s+1)}{\gamma^s}.
\end{equation}
In fact, the estimate \eqref{eq:Delta0} for $s=0$ can be included in \eqref{eq:Delta2}. 

Finally, the claim \eqref{EM1} follows from \eqref{eq:N0}, \eqref{eq:N1} and \eqref{eq:Delta2}.
\end{proof}

Refined asymptotics of sums of the form \eqref{eq:1st} can be obtained with the help of the second-order Euler--Maclaurin formula \eqref{eq:EM2}. 
Keeping the same notation $f(x)=\ell^s\myp \rme^{-\gamma\ell}|_{\myp\ell=x^q}$ and using that ${B}^{\circ\vphantom{a_a}}_2(x)$ is bounded, we obtain
\begin{equation}\label{eq:EM-ref}
\sum_{\ell\in\NN^q}\ell^s\myp \rme^{-\gamma\ell}= \int_0^\infty \!f(x)\,\rmd{x}-\frac{f(0)}{2}+O(1)\!\int_0^\infty\!|f''(x)|\,\rmd{x},
\end{equation}
where $f(0)=0$ for $s>0$ and $f(0)=1$ for $s=0$. Similarly as in \eqref{eq:f'}, we have
\begin{equation}\label{eq:f''ell}
f''(x)=\left.\frac{\rmd^2(\ell^s\myp \rme^{-\gamma\ell})}{\rmd{\ell}^2}\right|_{\myp \ell=x^q}\!\cdot (x^q)^{\prime\mypp 2}+\left.\frac{\rmd(\ell^s\myp \rme^{-\gamma\ell})}{\rmd{\ell}}\right|_{\myp\ell=x^q}\!\cdot (x^q)''.
\end{equation}
Recalling from \eqref{eq:f'} that $\frac{\rmd}{\rmd{\ell}}(\ell^s\myp \rme^{-\gamma\ell})=(s-\gamma\ell)\,\ell^{s-1}\myp\rme^{-\gamma\ell}$, we further compute
$$
\frac{\rmd^2(\ell^s\myp \rme^{-\gamma\ell})}{\rmd{\ell}^2}
=\frac{\rmd}{\rmd{\ell}}\!\left((s-\gamma\ell)\,\ell^{s-1}\myp\rme^{-\gamma\ell}\right)=\left((s-1)\myp s-2s\gamma\myp\ell+\gamma^2\ell^2\right)\ell^{s-2}\mypp\rme^{-\gamma\ell}. 
$$
Returning to \eqref{eq:f''ell} and \eqref{eq:EM-ref}, one can check  that each of the arising integrals in the remainder term is estimated by $O(\gamma^{-s+1/q})$. As a result, this leads to an improved asymptotic formula (cf.~\eqref{eq:1st})
\begin{equation}\label{eq:2nd+}
\sum_{\ell\in\NN^q} \ell^s\myp \rme^{-\gamma\ell}=\begin{cases}
\displaystyle \frac{\Gamma(1/q)}{q\,\gamma^{1/q}}-\frac12+O\bigl(\gamma^{1/q}\bigr),&s=0,\\[.8pc] \displaystyle \frac{\Gamma(s+1/q)}{q\,\gamma^{s+1/q}}\left(1+O\bigl(\gamma^{2/q}\bigr)\right),&s>0.
\end{cases}
\end{equation}

The refined estimate \eqref{eq:2nd+} with $s=1$ will be useful in the discussion of a numerical implementation of the Boltzmann sampler in Section \ref{sec:6.1}.
\begin{lemma}\label{lm:sum2} As $\eta\to0$ and $\gamma\to0^+$, 
\begin{equation}\label{eq:2nd}
 \sum_{\ell\in\NN^q}\log\bigl(1+ \eta\,\rme^{-\gamma\ell}\bigr)= \frac{\eta\:\Gamma(1/q)}{q\,\gamma^{1/q}}\left(1+O(\eta)+O\bigl(\gamma^{1/q}\bigr)\right)\!.
\end{equation}
\end{lemma}
\begin{proof}
By the elementary inequalities
$$
x-\tfrac{1}{2}\mypp x^2\le \log\myn(1+x)\le x\qquad (0<x<1),
$$
applied to each term in the sum, we obtain
\begin{align*}
 \sum_{\ell\in\NN^q}\log\bigl(1+ \eta\,\rme^{-\gamma\ell}\bigr)&=\eta\sum_{\ell\in\NN^q} \rme^{-\gamma\ell}+O(\eta^2) \sum_{\ell\in\NN^q} \rme^{-2\myp\gamma\ell}\\
 &=\eta \sum_{\ell\in\NN^q} \rme^{-\gamma\ell}\left(1+O(\eta)\,\frac{\sum_{\ell\in\NN^q} \rme^{-2\myp\gamma\ell}}{\sum_{\ell\in\NN^q} \rme^{-\myp\gamma\ell}}\right)\!,
\end{align*}
and the claim follows due to Lemma \ref{lm:sum1} (with $s=0$).
\end{proof}

We will also need a more general version of Lemmas \ref{lm:sum1} and \ref{lm:sum2}, which can be proved in a similar manner  using the indented Euler--Maclaurin formula \eqref{EM1}. In what follows, the summation range $\{\ell\ge\ell_*\}$ is a shorthand for $\{\ell\in\NN^q\colon \ell\ge\ell_*\}$. We also use the notation \begin{equation}\label{eq:inc-gamma}
\Gamma(a,x):=\int_x^\infty \!u^{a-1}\mypp\rme^{-u}\,\rmd{u}\qquad (a>0, \,x\ge0)
\end{equation}
for the \emph{(upper) incomplete gamma function} \cite[8.2.2]{NIST}.
\begin{lemma}\label{lm:sum1+}
As $\gamma\to0^+$,
\begin{align}\label{eq:1st'}
\sum_{\ell\ge\ell_*} \ell^s\myp \rme^{-\gamma\ell} = \frac{\Gamma(s+1/q,\gamma\myp\ell_*)}{q\,\gamma^{s+1/q}}
\left(1+O\bigl(\gamma^{1/q}\bigr)\right)\!.
\end{align}
\end{lemma}
\begin{lemma}\label{lm:sum2+} As $\eta\to0$ and  $\gamma\to0^+$, 
\begin{equation}\label{eq:2nd'}
 \sum_{\ell\ge\ell_*}\log\bigl(1+ \eta\,\rme^{-\gamma\ell}\bigr)=\frac{\eta\:\Gamma(1/q,\gamma\myp\ell_*)}{{q\,\gamma^{1/q}}}
\end{equation}
\end{lemma}

For ease of future reference, we state here a well-known criterion for uniform convergence of monotone functions adapted to the half-line domain (see, e.g., \cite[Sec.\,0.1]
{Resnick}).
\begin{lemma}\label{lm:uniform}
Let a sequence of monotone functions on $(0,\infty)$, uniformly bounded 
on $[\myp\delta,\infty)$ for any $\delta>0$, converge pointwise to a continuous (monotone) function. Then this convergence is uniform on $[\myp\delta,\infty)$ for any $\delta>0$. 
\end{lemma}

\subsection{Calibration of the parameters}\label{sec:3.3}
To analyze the Boltzmann distribution under various limit regimes, it is convenient to re-parameterize it via the \emph{hyper-parameters}   \begin{equation}\label{braket}
  \langle N\rangle=\EE_{\bm{z}}(N_\lambda), \qquad   \langle M\rangle=\EE_{\bm{z}}(M_\lambda).
\end{equation} 
Throughout the rest of the paper, we will work under the following growth condition.
\begin{assumption}\label{as0}
It is assumed that  $\langle N \rangle \to \infty$ and $\langle M\rangle^{-1}=O(1)$ (that is, $\langle M\rangle$ is bounded away from zero), and furthermore,
\begin{equation}\label{eq:kappa}
\kappa:=\frac{\langle M\rangle^{q+1}}{\langle N\rangle\ \ }\to 0.
\end{equation}
\end{assumption}

\begin{remark}
The meaning of Assumption \ref{as0} is elucidated by a  comparison with the case of \emph{all} strict partitions $\check{\varLambda}^{q}$, that is,  \emph{without controlling the number of parts}. Here, the parameter $z_2$ would become obsolete (we can formally set $z_2=1$ in the Boltzmann distribution formula \eqref{Boltzmann}), while the parameter $z_1\in(0,1)$, calibrated from the weight condition in \eqref{braket}, can be shown (using the Euler--Maclaurin sum formula \eqref{eq:EM}, similarly as in the proof of Theorem \ref{th:cal} below) to satisfy the asymptotics $-\log z_1\sim c_q\mypp\langle N\rangle^{-q/(q+1)}$, with 
\begin{equation}\label{eq:cq}
c_q^{1+1/q}=\frac{1}{q}\int_0^\infty\! u^{1/q}\myp\frac{\rme^{-u}}{1+\rme^{-u}}\,\rmd{u}. 
\end{equation}
In turn, the expected length has the asymptotics 
$\EE_{\bm{z}}(M_\lambda)\sim C_q\mypp\langle N\rangle^{1/(q+1)}$,
where
\begin{equation}\label{eq:Cq}
C_q=\frac{1}{q\,c_q^{1/q}}\int_0^\infty\! u^{1/q-1}\myp\frac{\rme^{-u}}{1+\rme^{-u}}\,\rmd{u}.
\end{equation}
For example, for $q=1$ the 
integrals in \eqref{eq:cq} and \eqref{eq:Cq} can be evaluated to yield  $c_1=\pi/\sqrt{12}$ and  $C_1=\bigl(\sqrt{12}\,\log 2\bigr)/\pi$ (cf.\ \cite{Erdos}).
Thus, the restriction that we put on the growth of $\langle M\rangle$ in Assumption \ref{as0} means that the number of parts is asymptotically smaller than what is expected from typical partitions in $\check{\varLambda}^q$ of large expected weight $\langle N\rangle$.
\end{remark}

The conditions \eqref{braket} can be viewed as a set of equations on the parameters $z_1$ and $z_2$ (cf.\ \eqref{braket0}). According to Corollary \ref{cor:unique}, a solution to \eqref{braket} is unique, if it exists. The following theorem gives an asymptotic representation of the roots $z_1$ and $z_2$ in terms of $\langle N\rangle$ and $\langle M\rangle$. 

\begin{theorem}\label{th:cal}
Under Assumption \ref{as0}, the roots $z_1$ and $z_2$ of the equations (\ref{braket}) are asymptotically given by
\begin{align}
\label{eq:z1}
z_1 &=\exp\left(-\frac{\langle M\rangle}{q\mypp\langle N\rangle}\left(1+O\bigl(\kappa^{1/q}\bigr) \right)\right)\!,\\
z_2 &= \frac{\kappa^{1/q}}{q^{1/q}\,\Gamma(1+1/q)}\left(1+O\bigl(\kappa^{1/q}\bigr)\right)\!.
\label{eq:z2}
\end{align}
\end{theorem}

\begin{proof}
Denote for short
$   \gamma:=-\log z_1$. In view of Lemma \ref{distnu} (see \eqref{eq:Bern}), we have  
\begin{equation}\label{eq:<M>}
\EE_{\bm{z}}(M_\lambda)=\sum_{\ell\in\NN^q}\EE_{\bm{z}}(\nu_\ell)= \sum_{\ell\in\NN^q}\frac{z_1^{\ell} z_2}{1+z_1^{\ell} z_2}=z_2\sum_{\ell\in\NN^q}z_1^{\ell} -R_1(\bm{z}),
\end{equation}
where
\begin{equation}\label{eqR1}
0<R_1(\bm{z})=\sum_{\ell\in\NN^q}\frac{z_1^{2\ell}z_2^2}{1+z_1^{\ell} z_2}<  z_2^2\sum_{\ell\in\NN^q} z_1^{2\ell}.
\end{equation}
Then Lemma \ref{lm:sum1} (with $s=0$)  applied to the sums on the right-hand side of \eqref{eq:<M>} and \eqref{eqR1} yields 
\begin{equation}\label{eq:EM}
\EE_{\bm{z}}(M_\lambda)=
\frac{z_2\,\Gamma(1/q)}{q\,\gamma^{1/q}}\left(1+ O\bigl(\gamma^{1/q}\bigr)+O(z_2)\right)\!.
\end{equation}

Similarly, 
\begin{equation}\label{eq:<N>}
\EE_{\bm{z}}(N_\lambda)=\sum_{\ell\in\NN^q}\ell\:\EE_{\bm{z}}(\nu_{\ell})=\sum_{\ell\in\NN^q}\frac{\ell\myp z_1^{\ell} z_2}{1+z_1^{\ell} z_2}=z_2\sum_{\ell\in\NN^q}\ell\myp z_1^{\ell} -R_2(\bm{z}),
\end{equation}
where
\begin{equation}\label{eqR2}
0<R_2(\bm{z})=\sum_{\ell\in\NN^q}\frac{\ell\myp z_1^{2\ell}\myn z_2^2 }{1+z_1^{\ell} z_2}< z_2^2\sum_{\ell\in\NN^q} \ell\myp z_1^{2\ell}.
\end{equation}
Again applying Lemma \ref{lm:sum1} (now with $s=1$), we get
\begin{equation}\label{eq:EN}
\EE_{\bm{z}}(N_\lambda)=
\frac{z_2\,\Gamma(1+1/q)}{q\,\gamma^{1+1/q}}\left(1+ O\bigl(\gamma^{1/q}\bigr)+O(z_2)\right)\!.
\end{equation}

Returning to the calibrating conditions \eqref{braket} and substituting the asymptotic expressions \eqref{eq:EM} and \eqref{eq:EN}, we obtain the following system of asymptotic equations,
\begin{equation}\label{eq:<MN>1}
\left\{
\begin{aligned}\langle M\rangle &= \frac{z_2\,\Gamma(1/q)}{q\,\gamma^{1/q}}\left(1+O\bigl(\gamma^{1/q}\bigr)+O(z_2)\right)\!,\\
\langle N\rangle &= \frac{z_2\, \Gamma(1+1/q)}{q\, \gamma^{1+1/q}}\left(1+O\bigl(\gamma^{1/q}\bigr)+O(z_2) \right)\!.
\end{aligned}
\right.\end{equation}
Since $\langle M\rangle$ is bounded away from zero, the first of these equations implies that $\gamma^{1/q}=O(z_2)$, so that the error terms $O(\gamma^{1/q})$ in \eqref{eq:<MN>1} are superfluous. 

A further simple analysis of the system \eqref{eq:<MN>1} shows that $z_2$ is of order of $\kappa^{1/q}$; specifically, using that $\Gamma(1+1/q)=(1/q)\,\Gamma(1/q)$, we find   
\begin{align*}
\gamma&=\frac{\langle M\rangle}{q\mypp\langle N\rangle} \left(1+O\bigl(\kappa^{1/q}\bigr)\right)\!,\\
z_2&=\frac{\langle M\rangle \mypp\gamma^{1/q}}{\Gamma(1+1/q)}\left(1+O\bigl(\kappa^{1/q}\bigr)\right) =\frac{\kappa^{1/q}}{q^{1/q}\,\Gamma(1+1/q)}\left(1+O\bigl(\kappa^{1/q}\bigr)\right)\!,
\end{align*}
in line with formulas \eqref{eq:z1} and
\eqref{eq:z2}.
\end{proof}

\begin{assumption}
\label{as1}
Throughout the rest of the paper, we assume that the parameters $z_1$ and $z_2$ are chosen according to formulas \eqref{eq:z1} and \eqref{eq:z2}, respectively. In particular, the Boltzmann measure $\QQ_{\bm{z}}$ becomes dependent on the hyper-parameters $\braket{N}$ and $\braket{M}$, as well as the
$\QQ_{\bm{z}}$-probabilities and the corresponding expected values.
\end{assumption}

For the sake of future reference and to fix the notation already used in the proof of Theorem \ref{th:cal}, 
the asymptotic formula \eqref{eq:z1} can be written as 
\begin{equation}\label{eq:gamma}
\gamma:=-\log z_1=  \frac{\braket{M}}{q\myn\braket{N}} \left(1+O\bigl(\kappa^{1/q}\bigr)\right)\!.
\end{equation}
It is also useful to record a simple consequence of the relations \eqref{eq:z1} and \eqref{eq:z2} (on account of the notation \eqref{eq:kappa} and \eqref{eq:gamma}), 
\begin{equation}\label{eq:z2/gamma}
\frac{z_2}{\gamma^{1/q}}=\frac{\braket{M}}{\Gamma(1+1/q)} \left(1+O\bigl(\kappa^{1/q}\bigr)\right)\!. 
\end{equation}

\begin{lemma}\label{lm:=0}
Under Assumptions \ref{as0} and \ref{as1},
\begin{equation}\label{eq:F0}
\log F(\bm{z})\sim \braket{M}\mynn.
\end{equation}
\end{lemma}
\begin{proof}
Using formula \eqref{eq:F_q} and applying Lemma \ref{lm:sum2} with $\gamma=-\log z_1$ and $\eta=z_2$, we have
$$
\log F(\bm{z})=\sum_{\ell\in\NN^q} \log\bigl(1+z_1^\ell z_2\bigr)\sim\frac{z_2\,\Gamma(1+1/q)}{\gamma^{1/q}}\sim \langle M\rangle,
$$
according to formula~\eqref{eq:z2/gamma}.
\end{proof}
\begin{remark}\label{rm:==0}
The result \eqref{eq:F0} provides the asymptotics of the probability of the empty partition $\lambda_\varnothing$ (cf.\ Remark~\ref{rm:=0}); indeed, by formula \eqref{Boltzmann}
\begin{equation}\label{eq:l=empty}
\QQ_{\bm{z}}(\lambda_\varnothing)=\frac{1}{F(\bm{z})}=\rme^{-\langle M\rangle\mypp(1+o(1))}.
\end{equation}
\end{remark}

\subsection{Asymptotics of the covariance matrix}\label{sec:3.4}
\begin{theorem}\label{th:VarCov}
Under Assumptions \ref{as0} and \ref{as1}, we have
\begin{equation}\label{eq:VarCov}
\Var_{\bm{z}}(M_\lambda)\sim\langle M\rangle,\qquad \Var_{\bm{z}}(N_\lambda)\sim\frac{(q+1)\mypp\langle N\rangle^2}{\langle M\rangle},\qquad \Cov_{\bm{z}}(M_\lambda,N_\lambda)\sim \langle N\rangle.
\end{equation}
\end{theorem}
\begin{proof}
Using formulas \eqref{eq:NM}, mutual independence of the multiplicities $(\nu_\ell)$ and the Bernoulli marginals \eqref{eq:Bern}, we have
\begin{equation}\label{eq:VarM2}
\Var_{\bm{z}}(M_\lambda)=\sum_{\ell\in\NN^q}\Var_{\bm{z}}(\nu_\ell)=\sum_{\ell\in\NN^q}\frac{z_1^\ell z_2}{(1+z_1^\ell z_2)^2}=z_2\sum_{\ell\in\NN^q} z_1^\ell -R_3(\bm{z}),
\end{equation}
where (cf.\ the proof of Theorem \ref{th:cal})
\begin{equation}\label{eq:R3}
0< R_3(\bm{z})= \sum_{\ell\in\NN^q}\frac{z_1^{2\ell}z_2^2\left(2+z_1^{\ell}z_2\right)}{(1+z_1^\ell z_2)^2}< 2 z_2^2\sum_{\ell\in\NN^q}\! z_1^{2\ell}+ z_2^3\sum_{\ell\in\NN^q} \!z_1^{3\ell}.
\end{equation}
By Lemma \ref{lm:sum1} we obtain, for any $r>0$,
\begin{equation}\label{eq:sum_q}
z_2\sum_{\ell\in\NN^q}\! z_1^{r\ell}\sim \frac{z_2\mypp\Gamma(1/q)}{q\left(r\gamma\right)^{1/q}}\sim \frac{\langle M\rangle}{r^{1/q}},
\end{equation}
according to \eqref{eq:z2/gamma}. Recalling that $z_2=o(1)$ (see \eqref{eq:z2}), it follows from \eqref{eq:R3} and \eqref{eq:sum_q} that
\begin{equation}\label{eq:R3'}
R_3(\bm{z})=O\mypp(z_2\myp\langle M\rangle)=o\mypp(\langle M\rangle).
\end{equation}
Hence, returning to \eqref{eq:VarM2} and using \eqref{eq:sum_q}, \eqref{eq:R3'} and formulas \eqref{eq:z1} and \eqref{eq:z2}, we get
\begin{equation}\label{eq:Sum-22}
\Var_{\bm{z}}(M_\lambda)\sim z_2\sum_{\ell\in\NN^q} z_1^\ell \sim  \langle M\rangle,
\end{equation}
in accord with the first formula in \eqref{eq:VarCov}.

Similarly (omitting technical details, which can be easily worked out), we obtain
\begin{align}
\notag
\Var_{\bm{z}}(N_\lambda)=\sum_{\ell\in\NN^q}\ell^2\, \Var_{\bm{z}}(\nu_\ell)&=\sum_{\ell\in\NN^q}\frac{\ell^2 z_1^\ell z_2}{(1+z_1^\ell z_2)^2}\\
&\sim z_2\sum_{\ell\in\NN^q} \ell^2 z_1^\ell 
\sim \frac{z_2\,\Gamma(2+1/q)}{q\mypp\gamma^{2+1/q}}
\sim\frac{(q+1)\mypp\langle N\rangle^2}{\langle M\rangle},
\label{eq:Sum-11}
\end{align}
and
\begin{align}
\notag
\Cov_{\bm{z}}(M_\lambda,N_\lambda)=\sum_{\ell\in\NN^q}\ell\; \Var_{\bm{z}}(\nu_\ell)
&=\sum_{\ell\in\NN^q}\frac{\ell\myp z_1^\ell z_2}{(1+z_1^\ell z_2)^2}\\
&\sim z_2\sum_{\ell\in\NN^q} \ell\myp z_1^\ell
\sim \frac{z_2\,\Gamma(1+1/q)}{q\mypp\gamma^{1+1/q}}
\sim \langle N\rangle,
\label{eq:Sum-12}
\end{align}
as claimed in \eqref{eq:VarCov}.
\end{proof}
\begin{corollary}\label{cor:5.2} 
The correlation coefficient between $M_\lambda$ and $N_\lambda$ is asymptotically given by
$$
\varrho(M_\lambda,N_\lambda)\sim\frac{1}{\sqrt{q+1}}.
$$
\end{corollary}
\begin{remark}
Corollary \ref{cor:5.2} shows that the dependence between $M_\lambda$ and $N_\lambda$ does not vanish, and also that the amount of this dependence is decreasing with the growth of the power index $q$.
\end{remark}

\section{Fixed expected length} \label{sec:4}

\subsection{Limit theorems for the partition length and weight}\label{sec:4.1}

In what follows, we use the notation
\begin{equation}\label{eq:gamma-cdf}
G_\alpha(x):=\frac{1}{\Gamma(\alpha)}\int_0^x\! u^{\alpha-1}\mypp\rme^{-u}\,\rmd{u},\qquad x\ge0,
\end{equation}
for the distribution function of the gamma distribution $\mathrm{Gamma}\mypp(\alpha)$ with shape parameter $\alpha>0$ (and unit scale parameter). We are now in a position to obtain our main result in this section.
\begin{theorem}\label{thm1}
Under Assumptions \ref{as0} and \ref{as1}, consider the regime where $\EE_{\bm{z}}(M_\lambda)=\langle M\rangle>0$ is fixed, while $\EE_{\bm{z}}(N_\lambda)=\langle N\rangle\to\infty$. Then the following distributional asymptotics hold under the Boltzmann distribution $\QQ_{\bm{z}}$ on the space $\check{\varLambda}^q$. 
\begin{itemize}
\item[\rm (a)]
The distribution of the length $M_\lambda$ converges to a Poisson distribution with parameter $\langle M\rangle$,
\begin{equation}\label{cvM}
\QQ_{\bm{z}}(M_\lambda=m)\to\pi_m:=\frac{\langle M\rangle^m\, \rme^{-\langle M\rangle}}{m!},\qquad m\in\NN_0.
\end{equation}
\item[\rm (b)]
The conditional distribution of the weight $N_\lambda$ 
given $M_\lambda=m\ge 1$ 
converges to the gamma distribution with shape parameter $\alpha_m=m/q$, 
\begin{equation}\label{cvN}
\QQ_{\bm{z}}\!\left(\gamma\myp N_\lambda\le x \,
\bigr|\myp M_\lambda=m \right) \to 
G_{m/q}(x)=\frac{1}{\Gamma(m/q)}
\int_0^x\! u^{m/q-1}\mypp\rme^{-u}\,\rmd{u},\qquad x\ge0,
\end{equation}
where $\gamma$ is defined in (\ref{eq:gamma}),
Moreover, convergence (\ref{cvN}) is uniform in $x\ge0$.
\item [\rm (c)] The marginal (unconditional) distribution function $G(x)=\lim_{\langle N\rangle\to\infty} \QQ_{\bm{z}}(\gamma\myp N_\lambda\le x)$, 
with atom $G(0)=\pi_0 =\rme^{-\braket{M}}$ at zero, is determined by its Laplace transform
\begin{equation}\label{eq:N-marginal}
\phi(s)=
\exp\left\{-\langle M\rangle \! \left(1-(1+s)^{-1/q}\right)\right\}\!,\qquad s\ge0.
\end{equation}
Furthermore, conditioned on 
$M_\lambda>0$, the Laplace transform becomes
\begin{equation}\label{eq:N-marginal-cond}
\tilde{\phi}(s)=
\frac{\rme^{-\braket{M}}}{1-\rme^{-\braket{M}}}\left(\exp\left\{\frac{\braket{M}}{(1+s)^{1/q}}\right\}-1\right)\!,\qquad s\ge0.
\end{equation}
\end{itemize}
\end{theorem}

\begin{remark}\label{rm:Poisson}
Part (a) is a particular case of a well-known Poisson approximation for the distribution of the total number of successes in a sequence of independent  Bernoulli trials with success probabilities $(p_i)$, which is valid as long as $\sum_i p_i\to\mathrm{const}>0$ and $\sum_i p_i^2\to0$ \cite{Barbour,Novak}.
Indeed, here we deal with the Bernoulli sequence $(\nu_\ell)$, where  $M_\lambda=\sum_\ell \nu_\ell$ and $p_\ell=\EE_{\bm{z}}(\nu_{\ell})=z_1^\ell z_2/(1+z_1^\ell z_2)$. According to \eqref{braket}, $\sum_\ell p_\ell=\braket{M}$. Furthermore, noting that $(p_\ell)$ is monotone decreasing, we obtain
\begin{equation}\label{eq:sqsum}
\sum_{\ell\in\NN^q} p_\ell^2\le p_1\sum_{\ell\in\NN^q} p_\ell=\frac{z_1 z_2}{1+z_1z_2}\braket{M}\le z_2\braket{M}=O\bigl(\braket{N}^{-1/q}\bigr)=o(1),
\end{equation}
on account of formula \eqref{eq:z2}.
Alternatively, the asymptotic estimate \eqref{eq:sqsum} follows with the help of formula \eqref{eq:sum_q} (with $r=2$). 
\end{remark}

\begin{remark}\label{rm:4.2}
The normalizing constant $\gamma$ in 
parts (b) and (c) can be replaced by 
its asymptotic equivalent $\gamma_0=\langle M\rangle/(q\myn\braket{N})$ (see~\eqref{eq:gamma}).  
\end{remark}
\begin{remark}\label{rm:m=0}
The case $m=0$ excluded in Theorem \ref{thm1}\myp(b) corresponds to the empty partition $\lambda_\varnothing$,
$$
\QQ_{\bm{z}}\!\left(\gamma\myp N_\lambda\le x \,
\bigr|\myp M_\lambda=0\right)=\QQ_{\bm{z}}(N_\lambda=0 \,
|\myp M_\lambda=0)=1,\qquad x\ge0.
$$
This is consistent with the gamma distribution \eqref{cvN} weakly converging to $0$ as the parameter  $\alpha_m=m/q$ is formally sent to zero. Indeed, using the Laplace transform, for any $s>0$ we have
$$
\frac{1}{\Gamma(\alpha)}\int_{0}^\infty \! x^{\alpha-1}\myp\rme^{-sx-x}\,\rmd{x}=\frac{1}{(s+1)^{\alpha}}\to 1\qquad (\alpha\to0^+).
$$
\end{remark}

\begin{remark}\label{rm:m-conv} Noting that  $\mathrm{Gamma}\mypp(m/q)$ is the convolution of $m$ copies of $\mathrm{Gamma}\mypp(1/q)$, the result of Theorem \ref{thm1}\myp(b) may be interpreted by saying that, on the scale $\gamma^{-1}\sim q\myp\langle N\rangle/\langle M\rangle$ and conditional on $M_\lambda=m$, the partition parts $\{\lambda_1,\dots,\lambda_m\}$ (considered without ordering) behave asymptotically as $m$ independent random variables, each with distribution  $\mathrm{Gamma}\mypp(1/q)$.
\end{remark} 
\begin{proof}[Proof of Theorem \ref{thm1}]
Consider the Laplace transform of the pair $(N_\lambda,M_\lambda)$,
\begin{equation}\label{eq:phi}
\phi_{\bm{z}}(\bm{s}):=\EE_{\bm{z}}\bigl[\exp\myn(-s_1 N_\lambda-s_2\myp M_\lambda)\bigr],\qquad s_1,s_2\ge0.
\end{equation}
Using formulas \eqref{eq:NM}, mutual independence of the multiplicities $(\nu_\ell)$ and the Bernoulli marginals \eqref{eq:Bern}, the definition \eqref{eq:phi} is rewritten as
\begin{align*}
\phi_{\bm{z}}(\bm{s})&=\EE_{\bm{z}}\!\left[\exp\left(-\sum_{\ell\in\NN^q}\left(s_1\ell+s_2\right)\nu_\ell\right)\right]\\
&=\prod_{\ell\in\NN^q}\EE_{\bm{z}}\bigl[\rme^{-(s_1\ell+s_2)\myp\nu_\ell}\bigr]=\prod_{\ell\in\NN^q}\frac{1+z_1^\ell z_2\,\rme^{-(s_1\ell+s_2)}}{1+z_1^\ell z_2}.
\end{align*}
The normalization $\gamma\myp  N_\lambda$ corresponds to replacing the argument $s_1$ in \eqref{eq:phi} by $\gamma\myp s_1$. 
Hence,
\begin{align}\notag
\log\phi_{\bm{z}}(\gamma\myp s_1,s_2) &= \log \EE_{\bm{z}}\bigl[\exp\myn(-\gamma\myp s_1 N_\lambda-s_2M_\lambda)\bigr]\\
\notag
&= \sum_{\ell\in\NN^q} \log\frac{1+z_1^{\ell}z_2\,\rme^{-(\gamma\myp s_1\ell+s_2)}}{1+z_1^{\ell}z_2}\\
&=\sum_{\ell\in\NN^q} \log\bigl( 1+z_1^{\ell}z_2\,\rme^{-(\gamma\myp s_1\ell+s_2)}\bigr)-\sum_{\ell\in\NN^q}\log\bigl( 1+z_1^{\ell}z_2\bigr).
\label{eq:phi1}
\end{align}
Starting with the last sum in \eqref{eq:phi1}, Lemma \ref{lm:=0} immediately gives 
\begin{align}
\sum_{\ell\in\NN^q} \log\bigl( 1+z_1^{\ell}z_2\bigr)=\log F(\bm{z})\to\langle M\rangle.
\label{eq:F-limit1}
\end{align}
Similarly, applying Lemma \ref{lm:sum2}  we obtain
\begin{align}
\sum_{\ell\in\NN^q} \log\bigl( 1+z_1^{\ell}z_2\,\rme^{-(\gamma\myp s_1\ell+s_2)}\bigr)
&\sim\frac{z_2\,\rme^{-s_2}\,\Gamma(1/q)}{
q\,\gamma^{1/q}(1+s_1)^{1/q}}\sim \frac{\rme^{-s_2}\langle M\rangle}{(1+s_1)^{1/q}},
\label{eq:F-limit2}
\end{align}
according to the asymptotic relation  \eqref{eq:z2/gamma}.

As a result, combining \eqref{eq:F-limit1} and \eqref{eq:F-limit2} yields
\begin{equation}\label{eq:phi-lim}
\phi_{\bm{z}}(\gamma\myp s_1,s_2)\to\exp\left\{-\langle M\rangle\! \left(1-\frac{\rme^{-s_2}}{\displaystyle(1+s_1)^{1/q}} \right)\right\}\!.
\end{equation}
In particular, setting $s_1=0$ we get the limiting Laplace transform of $M_\lambda$,
$$
\phi_{\bm{z}}(0,s_2)\to\exp\bigl(-\langle M\rangle (1-\rme^{-s_2})\bigr)=\sum_{m=0}^\infty \pi_m\,\rme^{-s_2\myp m},
$$
which corresponds to the Poisson distribution $(\pi_m)$ (see \eqref{cvM}),
thus proving the claim of part~(a). 

Furthermore, by Taylor expanding the exponential in the formula \eqref{eq:phi-lim}, we obtain
\begin{align}
\rme^{-\langle M\rangle}\exp \left(\frac{\langle M\rangle\, \rme^{-s_2}}{\displaystyle(1+s_1)^{1/q}} \right)
&= \sum_{m=0}^\infty 
\frac{\pi_m\,\rme^{-m\myp s_2}}{\displaystyle (1+s_1)^{m/q}}.
\label{eq:phi0}
\end{align}
This can be interpreted as follows: by the total expectation formula, we have 
\begin{align}
\notag
\phi_{\bm{z}}(\gamma\myp s_1,s_2)&=\EE_{\bm{z}}\!\left[\EE_{\bm{z}}\bigl(\rme^{-\gamma\myp s_1N_\lambda- s_2 M_\lambda}\bigr|\,M_\lambda=m\bigr)\right]\\
&=
\sum_{m=0}^\infty \QQ_{\bm{z}}(M_\lambda=m)\,\phi_{\bm{z}}(\gamma\myp s_1\myp|\mypp m)\,\rme^{-m\myp s_2},
\label{eq:phi2}
\end{align}
where $\phi_{\bm{z}}(s\mypp|\mypp m):=\EE_{\bm{z}}\bigl(\rme^{- s\myp N_\lambda}\myp|\mypp M_\lambda=m\bigr)$.
Comparing \eqref{eq:phi0} and \eqref{eq:phi2}, we conclude that
\begin{equation}\label{eq:con_s1}
\phi_{\bm{z}}(\gamma\myp s_1\myp|\mypp m)\to \frac{1}{\displaystyle (1+s_1)^{m/q}}.
\end{equation}
To be more precise, by the continuity theorem for Laplace transforms \cite[Sec.\,XIII.1, Theorem~2, p.\,431]{Feller} applied to the measure with atoms $a_{\bm{z}}(m;s_1):=\QQ_{\bm{z}}(M_\lambda=m)\,\phi_{\bm{z}}(\gamma\myp s_1\myp|\mypp m)$ (with $s_1\ge0$ fixed), it follows from \eqref{eq:phi0} and \eqref{eq:phi2} that
$$
a_{\bm{z}}(m;s_1)\to
\frac{\pi_m}{\displaystyle (1+s_1)^{m/q}},\qquad m\in\NN_0.
$$
But since the convergence \eqref{cvM} has already been established, this implies \eqref{eq:con_s1}, and it remains to observe that the right-hand side
is the Laplace transform of $\mathrm{Gamma}\mypp(m/q)$, 
as claimed in part~(b). Finally, the uniform convergence in \eqref{cvN} readily follows by application of Lemma~\ref{lm:uniform}. 

As for part (c), the first claim follows immediately
by setting $s_2=0$ in the limit \eqref{eq:phi-lim}. The atom at zero is identified as  \strut{}$\lim_{s\to\infty} \phi(s)=\rme^{-\braket{M}}=\pi_0$, and the conditional Laplace transform is expressed as $\tilde{\phi}(s)=(\phi(s)-\rme^{-\braket{M}})/(1-\rme^{-\braket{M}})$.
This completes the proof of Theorem~\ref{thm1}. 
\end{proof}

The limiting marginal distribution defined in \eqref{eq:N-marginal} is a mixture of a discrete family of gamma distributions indexed by the shape parameter $\alpha_m=m/q$ ($m\in\NN_0$), subject to a Poisson mixing distribution with parameter~$\braket{M}$,
\begin{equation}\label{eq:mix}
G(x)=\pi_0+\sum_{m=1}^\infty \pi_m\mypp G_{m/q}(x),\qquad x\ge 0.
\end{equation}
Formula \eqref{eq:mix} defines the compound Poisson-Gamma distribution of a random variable
\begin{equation}\label{eq:compound}
Y=Z_1+\dots+Z_{M},
\end{equation}
where $(Z_i)$ are independent random variables with gamma distribution $\mathrm{Gamma}\mypp(1/q)$ and $M$ is an independent random variable with a Poisson distribution $(\pi_m)$. In line with Remark \ref{rm:m=0}, the case $m=0$ is represented in \eqref{eq:mix} by a point mass $\pi_0=\rme^{-\braket{M}}$ at zero.  The absolutely continuous part of this distribution has density 
\begin{equation}\label{eq:g-density}
g(x)=\sum_{m=1}^\infty \pi_m\mypp G^{\myp\prime}_{m/q}(x)=\sum_{m=1}^\infty \pi_m\mypp\frac{x^{m/q-1}\mypp\rme^{-x}}{\Gamma(m/q)}=\frac{\rme^{-\braket{M}-x}}{x}
\,W_{1/q}\bigl(\langle M\rangle\mypp x^{1/q}\bigr),
\end{equation}
where $\displaystyle W_{\myn\varrho}(x)=\sum\nolimits_{m=1}^\infty \frac{x^m}{m!\:\Gamma(m\varrho)}$ is a special case of the \emph{Wright function} \cite{Wright-f}. Noting that $\mathrm{Gamma}\mypp(\alpha)$ has mean \strut{}$\alpha$, the expected value of the distribution \eqref{eq:mix} is
given by 
$$
\sum_{m=1}^\infty \pi_m\!\int_0^\infty\! x\,\rmd{\myp G_{m/q}(x)}=\sum_{m=1}^\infty \pi_m\mypp\frac{m}{q}
=\frac{\braket{M}}{q},
$$
which is consistent with the calibration $\EE_{\bm{z}}(N_\lambda)=\braket{N}$ (see \eqref{braket}) in view of the asymptotic formula $\gamma\sim\gamma_0=\langle M\rangle/(q\myn\braket{N})$ (see~\eqref{eq:gamma}). Of course, the same result can be obtained by differentiating the Laplace transform \eqref{eq:N-marginal} at $s=0$. 

The principal term in the asymptotics of the density $g(x)$ as $x\to+\infty$ can be recovered from the known asymptotic expansion of the Wright function \cite[Theorem 2, p.\,258]{Wright-f}, yielding
$$
g(x)=\frac{\rme^{-\braket{M}}}{2\pi \mypp(q+1)}\left(\frac{q}{x}\right)^{\mynn\frac{q+2}{2\myp(q+1)}}\exp\!\left((q+1)\left(\frac{x}{q}\right)^{\frac{q}{q+1}}\!-x\right)\left(1+O\bigl(x^{-\frac{q}{q+1}}\bigr)\right)\qquad (x\to+\infty).
$$
Turning to the behavior of $g(x)$ near zero, from the expansion \eqref{eq:g-density} it is clear that 
this is determined by the lowest-order terms with $m\le q$, that is, 
\begin{equation}\label{eq:pdf-as}
g(x) =\sum_{m=1}^{q} \pi_m\,\frac{x^{m/q-1}}{\Gamma(m/q)}+O(x^{1/q}) \qquad (x\to0^+).
\end{equation} 
Observe that if $q>1$ then the density of the absolutely continuous part of $G(x)$ has singularity at the origin, thus causing an ``excess'' of partitions with an anomalously small weight on the scale $\gamma^{-1}\sim \gamma_0^{-1}=q\,\langle N\rangle/\langle M\rangle$. On the other hand, the contribution of this singularity is exponentially vanishing as $\braket{M}\to\infty$. These effects will be verified empirically using  the output of the Boltzmann sampler considered below in Section \ref{sec:6.1.3} (see Figure \ref{fig4}). The exact distribution \eqref{eq:mix} will also be contrasted there with a crude approximation via replacing the Poisson mixing parameter by its expected value $\braket{M}$, yielding the gamma distribution with shape parameter $\langle M\rangle/q$. Of course, such an approximation cannot capture the aforementioned singularity at zero, but it works reasonably well for larger values of $\braket{M}$, whereby singularity becomes immaterial.

For the benefit of the Boltzmann sampling considered below in Section \ref{sec:6}, we conclude this section by stating a ``truncated'' version of Theorem \ref{thm1} with a reduced source of parts. We write $G_{\alpha}(x\,|\,a):=G_{\alpha}(x)/G_\alpha(a)$ \myp($0\le x\le a$) for the distribution function of the $\mathrm{Gamma}\mypp(\alpha)$-distribution truncated by threshold $a>0$. The symbol $\star$ stands for the convolution of probability distributions.  
\begin{theorem}\label{thm1<L}
Under the hypotheses of Theorem \ref{thm1}, let $L\sim\theta\mypp\langle N\rangle$ with $\theta>0$, and denote $a_\theta:=\theta\myp\langle M\rangle\myn/q$. Then the following distributional asymptotics hold subject to the constraint $\lambda_{\rm max}\le L$. 
\begin{itemize}
\item[\rm (a)]
The conditional distribution of the length $M_\lambda$ converges to a Poisson law,
\begin{equation}\label{eq:M<L}
\QQ_{\bm{z}}(M_\lambda=m\,|\, \lambda_{\rm max}\le L)\to\pi^\theta_m:=\frac{\mu_\theta^m\, \rme^{-\mu_\theta}}{m!},\qquad m\in\NN_0,
\end{equation}
with mean \begin{equation}\label{eq:mu-theta}
\mu_\theta:=\langle M\rangle\,G_{1/q}(a_\theta)=\frac{\braket{M}}{\Gamma(1/q)}\int_0^{a_\theta}\!u^{1/q-1}\mypp\rme^{-u}\,\rmd{u}.
\end{equation}
\item[\rm (b)]
The conditional distribution of the weight $N_\lambda$ 
given $M_\lambda=m\ge 1$ 
converges to the convolution of $m$ copies of an $a_\theta$-truncated gamma distribution with shape $1/q$,
\begin{equation}\label{cvN<L}
\QQ_{\bm{z}}(\gamma\myp N_\lambda\le x \,
\bigr|\myp M_\lambda=m, \lambda_{\rm max}\le L) \to 
\bigl(\underbrace{G_{1/q}(\cdot\,|\,a_\theta)\star \dots \star G_{1/q}(\cdot\,|\,a_\theta)}_{m}\bigr)(x),\qquad 0<x\le m\myp a_\theta.
\end{equation}
Like in (\ref{cvN}), convergence (\ref{cvN<L}) is uniform in\/ $0\le x\le m\myp a_\theta$.
\item [\rm (c)] The marginal distribution function $G(x;\theta):=\lim_{\langle N\rangle\to\infty} \QQ_{\bm{z}}(\gamma\myp N_\lambda\le x\mypp|\mypp\lambda_{\rm max}\le L)$, 
with atom $G(0;\theta)=\pi^\theta_0=\exp\left\{-\langle M\rangle\mypp G_{1/q}(a_\theta)\right\}$ at zero, is determined by its Laplace transform
\begin{equation}\label{eq:N-marginal<L}
\phi(s;\theta)
=\exp\left\{-\langle M\rangle \! \left(G_{1/q}(a_\theta)-\frac{G_{1/q}\bigl(a_\theta\myp(1+s)\bigr)}{(1+s)^{1/q}}\right)\right\}\!,\qquad s\ge0.
\end{equation}
Furthermore, conditioned on 
$M_\lambda>0$, the Laplace transform becomes
\begin{equation}\label{eq:N-marginal-cond-L}
\tilde{\phi}(s;\theta)=
\frac{\rme^{-\mu_\theta}}{1-\rme^{-\mu_\theta}}\left(\exp\left\{\frac{\braket{M}G_{1/q}\bigl(a_\theta\myp(1+s)\bigr)}{(1+s)^{1/q}}\right\}-1\right)\!,\qquad s\ge0.
\end{equation}
\end{itemize}
\end{theorem}

This theorem can be proved by adapting the proof of Theorem \ref{thm1}, whereby the (finite) sums of logarithmic expressions are analyzed with the help of Lemma \ref{lm:sum2+} in place of Lemma \ref{lm:sum2}.

\begin{remark}
Comparing Theorems \ref{thm1} and \ref{thm1<L}, a reduction of the Poisson parameter (mean) in part (a) from $\braket{M}$ to $\mu_\theta=\braket{M} G_{1/q}(a_\theta)<\braket{M}$ (see \eqref{eq:mu-theta}), as well as the replacement of the gamma distribution $G_{1/q}(x)$ in part (b) with a truncated version $G_{1/q}(x\mypp|\mypp a_\theta)$ (see \eqref{cvN<L}) is clearly due to a reduced source of parts, $\ell\le L\sim \theta\myp\langle N\rangle$.
\end{remark}
\begin{remark}
As a sanity check of formula \eqref{eq:N-marginal<L}, the expected value of $\gamma\myp N_\lambda$ conditional on $\lambda_{\rm max}\le L$ is asymptotically evaluated (using Lemma \ref{lm:sum1+} and formula \eqref{eq:z2/gamma}) as   
\begin{equation}\label{eq:ExpN1}
\EE_{\bm{z}}(\gamma\myp N_\lambda\,|\,\lambda_{\rm max}\le L)=\gamma \sum_{\ell\le L} \frac{\ell \mypp z_1^\ell z_2}{1+z_1^\ell z_2}\sim  \frac{\braket{M}}{\Gamma(1/q)}\int_0^{a_\theta}\!u^{1/q}\mypp\rme^{-u}\,\rmd{u}.
\end{equation}
On the other hand, by differentiating the Laplace transform \eqref{eq:N-marginal<L} at $s=0$ we obtain
\begin{equation*}
\EE_{\bm{z}}(\gamma\myp N_\lambda\,|\,\lambda_{\rm max}\le L)\sim \langle M\rangle \mynn \left(\frac{1}{q}\,G_{1/q}(a_\theta)-\frac{a_\theta^{1/q}\mypp\rme^{-a_\theta}}{\Gamma(1/q)}\right)\!.
\end{equation*}
These two expressions are reconciled by integration of parts in \eqref{eq:ExpN1} and in view of notation \eqref{eq:gamma-cdf}.
\end{remark}

\subsection{Cumulative cardinality of strict power partitions}\label{sec:4.2}
For fixed $q,m\in\NN$ and for any $x>0$, consider a sub-level partition set 
\begin{equation}\label{eq:sub-level}
\check{\varLambda}^q_m(x):=\bigcup_{n\le x}\check{\varLambda}^q(n,m),
\end{equation}
and denote its cardinality \begin{equation}\label{eq:Lambda-card}
L^q_{m}(x):=\# \check{\varLambda}^q_m(x)=\sum_{n\le x}\#\check{\varLambda}^q(n,m).
\end{equation}
That is to say, $L^q_{m}(x)$ denotes the number of integral solutions to the inequality $j_1^q+\dots+j_m^q\le x$ such that  $j_1>\dots>j_m>0$.

\begin{theorem}\label{th:ID}
The following asymptotics hold as $x\to\infty$,
\begin{equation}\label{eq:L}
L^q_{m}(x)\sim \frac{q\,\bigl(\Gamma(1+1/q)\bigr)^m \mypp x^{m/q}}{m!\,m\:\Gamma(m/q)}.
\end{equation}
In particular, for $q=1$ and $q=2$
\begin{equation}\label{eq:L=12}
L_{m}(x)\sim \frac{x^{m}}{m!\,m!}, \qquad L^2_{m}(x)\sim \frac{\pi^{m/2} \mypp x^{m/2}}{2^{m-1}\myp m!\,m\:\Gamma(m/2)}.
\end{equation}
\end{theorem}

\begin{figure}[h]
\centering
\hfill\subfigure[Sub-level set $\check{\varLambda}^q_m(x)$ with $q=1$,  $m=2$, $x=20$.]{\includegraphics[width=0.45\textwidth]{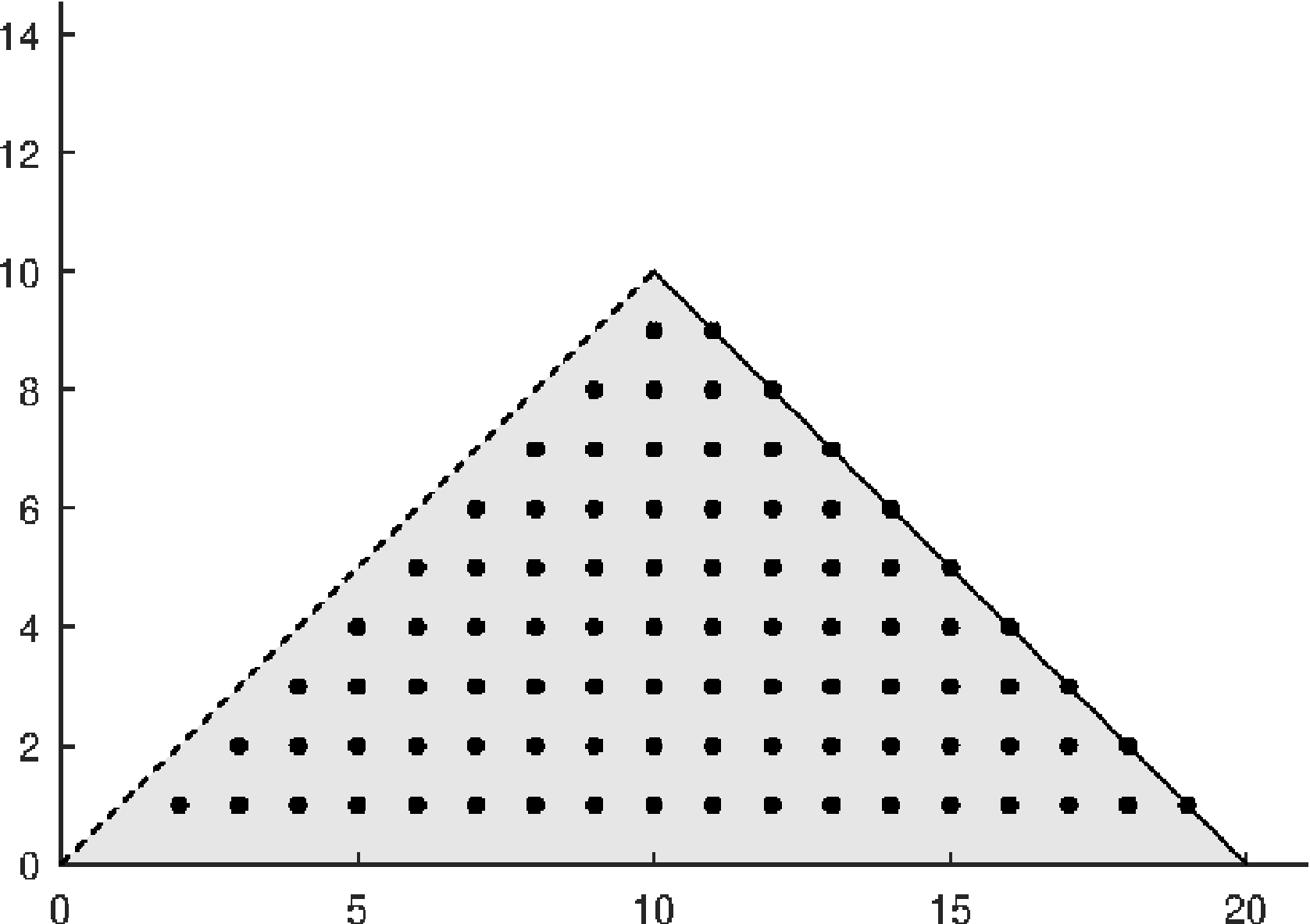}}
\hspace{1pc}\subfigure[Sub-level set $\check{\varLambda}^q_m(x)$ with $q=2$,  $m=2$, $x=400$.]{\includegraphics[width=0.45\textwidth]{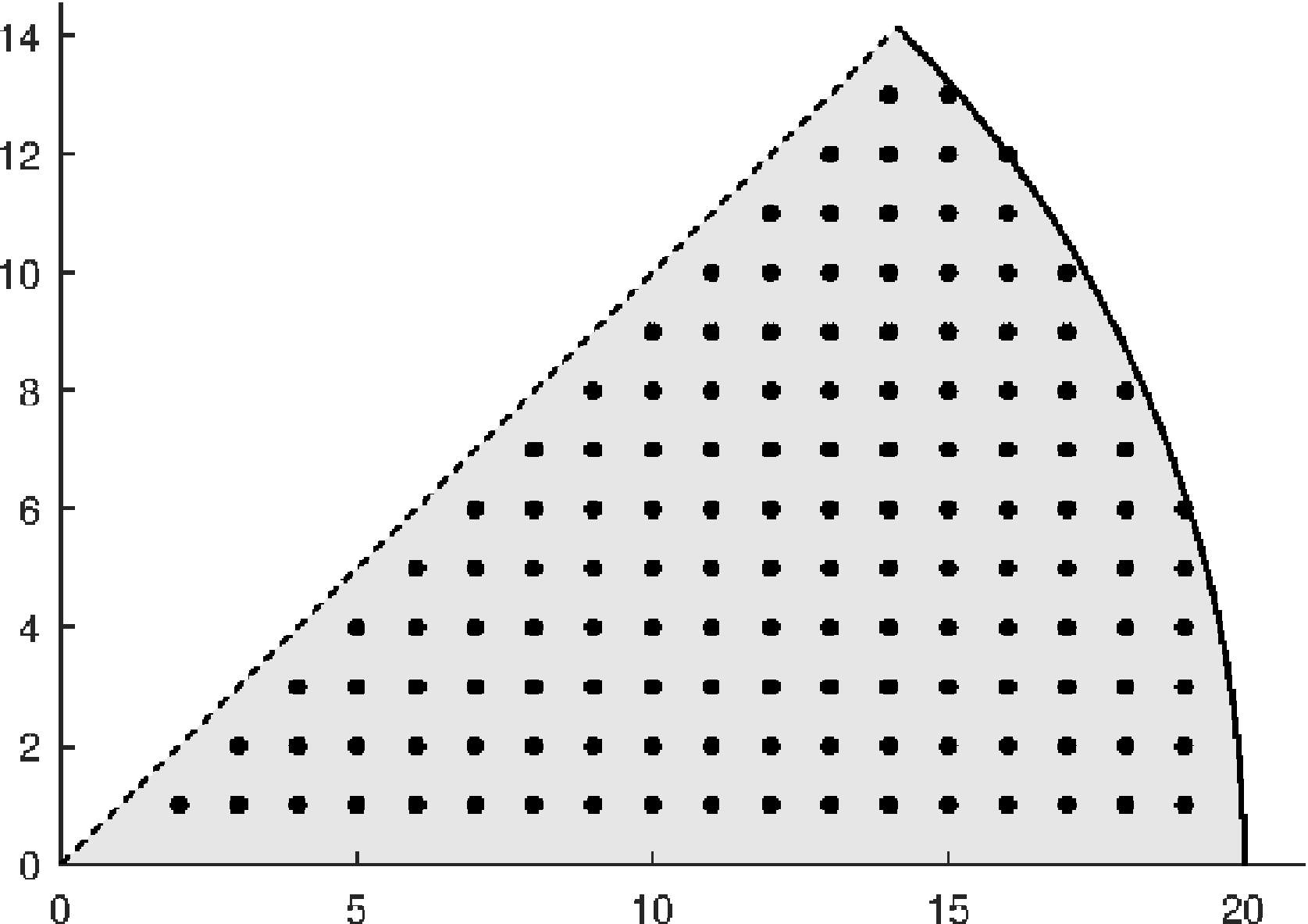}
}\hfill
\caption{Geometric illustration of the sub-level partition sets $\check{\varLambda}^q_m(x)$ (defined  in \eqref{eq:sub-level})
with $m=2$ and (a) $q=1$ or (b) $q=2$,
represented as the sets of integer points $(j_1,j_2) \in \ZZ^2$ such that $0<j_2<j_1$ and $j_1+j_2\le x$ or $j_1^2+j_2^2\le x$, respectively. In line with Theorem~\ref{th:ID}, 
their cardinalities have the asymptotics $L_{2}(x)\sim \frac14\myp x^2$ and 
$L^2_{2}(x)\sim \frac18\myp\pi x
$, 
corresponding to the area of the shaded domains. 
}
\label{fig:ball}
\end{figure}

\begin{proof}
Without loss of generality,
we may and will assume that $x$ is an integer. Set the hyper-parameters in the Boltzmann distribution $\QQ_{\bm{z}}$ to be $\langle M\rangle=m$ and $\langle N\rangle=x$.
By Corollary \ref{uniformity}, 
\begin{equation}\label{eq:nm}
\#\check{\varLambda}^{q}(n,m) = \frac{F(\bm{z})\,\QQ_{\bm{z}}(M_\lambda=m)\,\QQ_{\bm{z}}(N_\lambda=n\,|\mypp M_\lambda=m)}{z_1^n z_2^m}.
\end{equation} 
Hence, according to \eqref{eq:Lambda-card},
\begin{equation}\label{eq:Lq}
L^q_{m}(x)=\frac{F(\bm{z})\,\QQ_{\bm{z}}(M_\lambda=m)}{z_2^{m}}\sum_{n \le x}\frac{\QQ_{\bm{z}}\left(N_\lambda=n\,|\mypp M_\lambda=m  \right)}{z_1^{n}}.
\end{equation}

First of all, by Lemma \ref{lm:=0}, Theorem \ref{thm1}(a) and formula \eqref{eq:z2}, we get
\begin{align}
\notag
\frac{F(\bm{z})\,\QQ_{\bm{z}}(M_\lambda=m)}{z_2^{m}}&\sim \rme^{m}\, \frac{m^m\mypp\rme^{-m}}{m!}\left(\frac{m^{1+1/q}}{\displaystyle q^{1/q}\,\Gamma(1+1/q)}\,x^{-1/q}\right)^{\mynn-m}\\
&=\frac{q^{m/q}\myp\bigl(\Gamma(1+1/q)\bigr)^m\mypp x^{m/q}}{m!\, m^{m/q}}.
\label{eq:Q-as}
\end{align}
To handle the sum in \eqref{eq:Lq}, define the cumulative probabilities 
\begin{equation}\label{eq:T}
T_n:=\QQ_{\bm{z}}(N_\lambda\le n\,|\mypp M_\lambda=m),\qquad T_0:=0.
\end{equation}
Then by Abel's summation-by-parts formula (see, e.g., \cite[p.\,390]{Shiryaev}) we can write
\begin{align}
\notag
\sum_{n \le x}\frac{\QQ_{\bm{z}}\left(N_\lambda=n\,|\mypp M_\lambda=m  \right)}{z_1^{n}}
&=\sum_{n\le x} \bigl(T_n-T_{n-1}\bigr)\,\frac{1}{z_1^n}\\
\notag
&=\frac{T_{x}}{z_1^{x+1}}-\frac{T_0}{z_1}-\sum_{n\le x} T_n\left(\frac{1}{z_1^{n+1}}-\frac{1}{z_1^{n}}\right)\\
&=\frac{T_{x}}{z_1^{x-1}}- \frac{1-z_1}{z_1}\sum_{n\le x}\frac{T_n}{z_1^{n}}.\label{eq:Q2}
\end{align}
To shorten the notation, denote $\alpha:=m/q$. By the asymptotic formula \eqref{eq:z1}, we have  $1-z_1\sim \alpha/x$;
moreover, uniformly in $n\le x$
\begin{equation}\label{eq:z1uni}
z_1^{-n}=\rme^{\myp\alpha\myp n/x}\left(1+O\bigl(x^{-1/q}\bigr)\right)\!,
\end{equation}
since 
$\kappa=m^{q+1}\myn/x=O(x^{-1})$ (see \eqref{eq:kappa}).
Furthermore, using Theorem \ref{thm1}(b) and Remark \ref{rm:4.2} we get for the first term in \eqref{eq:Q2},
\begin{equation}\label{eq:appr1}
\frac{T_{x}}{z_1^{x-1}}\to \rme^{\alpha} \,G_\alpha(\alpha)=\frac{\rme^{\alpha}}{\Gamma(\alpha)}\int_0^\alpha \myn u^{\alpha-1}\mypp\rme^{-u}\,\rmd{u}.
\end{equation}
Generalizing, observe that $T_n\approx G_\alpha(\alpha\myp n/x)$ for $n\le x$.
More precisely, taking advantage of the uniform convergence in \eqref{cvN}, for any $\varepsilon>0$ and all large enough $\langle N\rangle$ we have, uniformly in $n\le x$,
$$
\left|T_n-G_\alpha(\alpha\myp n/x)\right|<\varepsilon.
$$
Hence, the total approximation error arising from the sum in \eqref{eq:Q2} is estimated as follows, 
\begin{align}\notag
(1-z_1)\sum_{n\le x} \frac{|T_n-G_\alpha(\alpha\myp n/x)|}{z_1^{n}}&
\le\varepsilon\,(1-z_1)\sum_{n=0}^{x} z_1^{-n}\\
&=\varepsilon\left(z_1^{-x}-z_1\right)
=O(\varepsilon).
\label{eq:eps}
\end{align}

Therefore, replacing $T_n$ with $G_\alpha(\alpha\myp n/x)$ in the sum \eqref{eq:appr1} and also using \eqref{eq:z1uni}, we obtain 
\begin{align}
\frac{1-z_1}{z_1}\sum_{n\le x} z_1^{-n}\mypp G_\alpha(\alpha\myp n/x)&\sim  \frac{\alpha}{x}\sum_{n=1}^{x}  \rme^{\alpha\myp n/x}\myp G_\alpha(\alpha\myp n/x)\to\int_{0}^{\alpha} \rme^{u}\mypp G_\alpha(u)\,\rmd{u}.
\label{eq:iint1}
\end{align}
Furthermore, the integral in \eqref{eq:iint1} is evaluated by integration by parts,
\begin{align}
\int_{0}^{\alpha}\! G_\alpha(u)\,\rmd\myp(\rme^{u})&=\rme^{\alpha}\myp G_\alpha(\alpha)
-\frac{1}{\Gamma(\alpha)}\int_{0}^{\alpha}\! u^{\alpha-1}\,\rmd{u}=\rme^{\alpha}\, G_\alpha(\alpha)-\frac{\alpha^{\alpha-1}}{\Gamma(\alpha)}.
\label{eq:iint2}
\end{align}
Thus, collecting \eqref{eq:appr1}, \eqref{eq:eps}, \eqref{eq:iint1} and \eqref{eq:iint2}, from \eqref{eq:Q2} we obtain 
\begin{equation}\label{eq:eps2}
\lim_{x\to\infty} \sum_{n\le x}\frac{\QQ_{\bm{z}}(N_\lambda=n\,|\mypp M_\lambda=m)}{z_1^{n}}= \frac{\alpha^{\alpha-1}}{\Gamma(\alpha)}+O(\varepsilon)\equiv \frac{\alpha^{\alpha-1}}{\Gamma(\alpha)},
\end{equation}
since $\varepsilon>0$ is arbitrary. 

Finally, returning to \eqref{eq:Lq} and substituting the limits \eqref{eq:Q-as}
and \eqref{eq:eps2} (the latter with $\alpha=m/q$), we obtain the asymptotic formula~\eqref{eq:L}.
\end{proof}

Continuing an illustration of Theorem \ref{th:ID} for $q=1$ and $q=2$ started in  Figure \ref{fig:ball} with $m=2$, it is well known \cite[Theorem 4.1, p.\,341]{Erdos}\footnote{To be precise, Theorem 4.1 of \cite{Erdos} is stated for partitions with at most $m$ parts, but due to Corollary 4.3 \cite[p.\,343]{Erdos} the same result also holds for partitions into exactly $m$ distinct parts.} (see also \cite[Theorem 4.3, p.\,56]{Andrews}) that, with $m\ge2$ fixed (or even $m=o(n^{1/3})$),
\begin{equation}\label{eq:q=1-size}
\#\check{\varLambda}(n,m)\sim \frac{1}{m!}\binom{n-1}{m-1}\sim
\frac{n^{m-1}}{m!\,(m-1)!}\qquad (n\to\infty).
\end{equation}
This is consistent with our cumulative formula \eqref{eq:L=12} for $q=1$, noting that
$$
\sum_{n\le x} n^{m-1}\sim\frac{x^m}{m}\qquad (x\to\infty). 
$$

With regards to the case $q=2$, it is interesting that in their famous paper of 1918 on an ``exact'' enumeration of 
integer partitions, Hardy and Ramanujan \cite[Eq.\,(7.21), p.\mypp 110]{Ramanujan} stated without proof an asymptotic formula for the number of representations of a large $n\in\NN$ as the sum of $m>4$ squares,
\begin{equation}\label{eq:HR-squares}
r_m(n):=\#\{(j_1,\dots,j_m)\in \ZZ^m\colon j_1^2+\dots+j_m^2=n\}=\frac{C_m\mypp\pi^{m/2} \mypp n^{m/2-1}}{\Gamma(m/2)}+O(n^{m/4})\qquad(n\to\infty),
\end{equation}
where the  constant $C_m>0$ is defined through a series 
$\sum_{k\ge 1} c_k\myp k^{-m/2}$ with computable coefficients $c_k=O(k)$. Using a geometric embedding of such representations into the space $\ZZ^m$, it is easy to see that their enumeration is asymptotically reduced to strict partitions with $m$ positive parts,
$$
r_m(n)\sim 2^m\myp m!\,\#\check{\varLambda}^2(n,m),
$$
and in view of \eqref{eq:HR-squares} it follows 
\begin{equation}\label{eq:Lq=2}
\#\check{\varLambda}^2(n,m)\sim \frac{C_m\mypp\pi^{m/2} \mypp n^{m/2-1}}{2^m\myp m!\:\Gamma(m/2)}.
\end{equation}
By the order of growth, this formula matches our cumulative asymptotic result \eqref{eq:L=12} (for $q=2$), and moreover, the constant is explicitly recovered from this comparison,
$C_m= 1$.

\subsection{A joint limit theorem for the 
extreme parts (fixed expected length)}\label{sec:4.3}
In this section, we address the limiting distribution of the largest and smallest parts of a random partition, $\lambda_{\rm max}$ and $\lambda_{\rm min}$. Recall the notation $\gamma=-\log z_1\sim \gamma_0=\langle M\rangle/(q\myp\langle N\rangle)$ (see~\eqref{eq:gamma}). 

\begin{theorem}\label{th:minmax}
Assume that $\braket{M}>0$ is fixed. 
Then\/ $\lambda_{\rm max}$ and\/ $\lambda_{\rm min}$ are asymptotically independent as  $\braket{N}\to\infty$, with the marginal limiting laws given by (for any $x\ge0$) 
\begin{align}\label{eq:max1}
\QQ_{\bm{z}}(\gamma\myp\lambda_{\rm max}\le x)&\to G_{\rm max}(x):=
\rme^{-\braket{M}\left(1-G_{1/q}(x)\right)},\\[.3pc]
\label{eq:min1}
\QQ_{\bm{z}}(\gamma\myp\lambda_{\rm min}>x)&\to G^{\myp c}_{\rm min}(x):=
\rme^{-\braket{M}\mypp G_{1/q}(x)}.
\end{align}
Moreover, conditionally on non-empty partition, 
\begin{align}\label{eq:max2}
\QQ_{\bm{z}}(\gamma\myp\lambda_{\rm max}\le x\,|\,\lambda_{\rm max}>0)&\to\widetilde{G}_{\rm max}(x):=\frac{G_{\rm max}(x)-\rme^{-\langle M\rangle}}{1-\rme^{-\langle M\rangle}},\\
\QQ_{\bm{z}}(\gamma\myp\lambda_{\rm min}>x\,|\,\lambda_{\rm min}<\infty)&\to\widetilde{G}^{\myp c}_{\rm min}(x):=\frac{G_{\rm min}^{\myp c}(x)-\rme^{-\langle M\rangle}}{1-\rme^{-\langle M\rangle}}.
\label{eq:min2}
\end{align}
\end{theorem}

\begin{remark}
Clearly, the normalization $\gamma$ in Theorem \ref{th:minmax} can be replaced by its asymptotic equivalent, $\gamma_0=\langle M\rangle/(q\myp\langle N\rangle )
$ (see \eqref{eq:gamma}).
\end{remark}

\begin{remark}\label{rm:G(0)}
Note that the distribution function $G_{\rm max}(x)$ has a jump at zero with mass $G_{\rm max}(0)=\rme^{-\langle M\rangle}$. This is in line with our convention for  $\lambda_{\rm max}=0$, which corresponds to the empty partition $\lambda_\varnothing$ (cf.\ Remark~\ref{rm:==0} and formula \eqref{eq:l=empty}).
On the other hand,
$\widetilde{G}_{\rm max}(0)=0$, and so the distribution function $\widetilde{G}_{\rm max}(x)$ is continuous at zero. Likewise, the tail distribution function $G^{\myp c}_{\rm min}(x)$ is improper, with defect mass $G^{\myp c}_{\rm min}(\infty)=\rme^{-\langle M\rangle}$, which again matches our convention for  $\lambda_{\rm min}=\infty$.
But $\widetilde{G}_{\rm min}^{\myp c}(\infty)=0$, so  $\widetilde{G}^{\myp c}_{\rm min}(x)$ defines a proper distribution. 
\end{remark}

\begin{proof}[Proof of Theorem \ref{th:minmax}]
For $x\ge0$, set 
$\ell_*(x):=\min\{\ell\in\NN^q\colon \ell> \gamma^{-1}x\}$, and note that $\gamma\myp\ell_*(x)\to x$.
By virtue of Lemma \ref{distnu},
\begin{align}
\notag
\QQ_{\bm{z}}(\gamma\myp\lambda_{\rm min}>x_1, \gamma\myp\lambda_{\rm max}\le x_2) &= \QQ_{\bm{z}}\bigl(\nu_{\ell}\equiv 0\ \text{for all}\ \ell<\ell_*(x_1)\ \text{and}\ \ell\ge\ell_*(x_2)\bigr)\\
\notag
&= \prod_{\ell<\ell_*(x_1),\,\ell\ge\ell_*(x_2)} \frac{1}{1+z_1^{\ell}z_2}\\
&= \exp\Biggl\{-\Biggl(\myp\sum_{\ell\in\NN^q}^\infty -\sum_{\ell\ge\ell_*(x_1)}^\infty +\sum_{\ell\ge\ell_*(x_2)}^\infty\Biggr) \log\mynn\!\left(1+z_1^{\ell}z_2 \right) \Biggr\}.
\label{eq:ell>ell*}
\end{align}
Hence, applying Lemma \ref{lm:sum2+} (with $\gamma=-\log z_1$ and $\eta=z_2$) and using the asymptotic relation \eqref{eq:z2/gamma}, we get
\begin{align*}
-\log \QQ_{\bm{z}}(\gamma\myp\lambda_{\rm min}> x_1,\gamma\myp\lambda_{\rm max}\le x_2)&\sim \frac{z_2}{ q\,\gamma^{1/q}}\left(\int_{0}^\infty\!-\int_{\gamma\myp\ell_*(x_1)}^\infty+\int_{\gamma\myp\ell_*(x_2)}^\infty \right) u^{1/q-1}\mypp\rme^{-u}\,\rmd{u}\\
&\sim\frac{\langle M\rangle}{\Gamma(1/q)} \bigl(\Gamma(1/q)-\Gamma(1/q,x_1)+\Gamma(1/q,x_2)\bigr)\\
&=-\log G^c_{\rm min}(x_1)-\log G_{\rm max}(x_2),
\end{align*}
which proves asymptotic independence and the marginal laws \eqref{eq:max1} and \eqref{eq:min1}.

The conditional versions \eqref{eq:max2} and \eqref{eq:min2} easily follow using that $G_{\rm max}(0)=G^c_{\rm min}(\infty)=\rme^{-\langle M\rangle}$ (see Remark~\ref{rm:G(0)}),
\begin{align*}
\QQ_{\bm{z}}(\gamma\myp\lambda_{\rm max}\le x\,|\,\lambda_{\rm max}>0) 
&=\frac{\QQ_{\bm{z}}(\gamma\myp\lambda_{\rm max}\le x)-\QQ_{\bm{z}}(\lambda_{\rm max}=0)}{1-\QQ_{\bm{z}}(\lambda_{\rm max}=0)}\\
&\to \frac{G_{\rm max}(x)-\rme^{-\langle M\rangle}}{1-\rme^{-\langle M\rangle}}=\widetilde{G}_{\rm max}(x),
\end{align*}
and similarly 
\begin{align*}
\QQ_{\bm{z}}(\gamma\myp\lambda_{\rm min}>x \,|\,\lambda_{\rm min}<\infty) 
&=\frac{\QQ_{\bm{z}}(\gamma\myp\lambda_{\rm min}>x) -\QQ_{\bm{z}}(\lambda_{\rm min}=\infty)}{1-\QQ_{\bm{z}}(\lambda_{\rm min}=\infty)}\\
&\to \frac{G^{\myp c}_{\rm min}(x)-\rme^{-\langle M\rangle}}{1-\rme^{-\langle M\rangle}}=\widetilde{G}_{\rm min}^{\myp c}(x),
\end{align*}
as claimed.
\end{proof}

\begin{remark}\label{rm:max-fixed}
The result of Theorem \ref{th:minmax} indicates that all parts $(\lambda_i)$ of a $\QQ_{\bm{z}}$-typical partition $\lambda\in\check{\varLambda}^q$ `live'' on the universal scale $A=\gamma^{-1}\sim q\mypp\langle N\rangle/\langle M\rangle$. Clearly, this is a manifestation of keeping the expected number of parts $\braket{M}$ fixed. The situation is different when the parameter $\braket{M}$ is allowed to grow with $\braket{N}$, as will be shown below in Theorem~\ref{th:maxM}.
\end{remark}
\begin{remark}\label{rm:maxmin-heuristic}
The interpretation of the limiting distribution of parts given in Remark \ref{rm:m-conv} can be used for a heuristic derivation of Theorem \ref{th:minmax}. Indeed, using 
independence and the $\mathrm{Gamma}\mypp(1/q)$-distribution of the independent limiting parts $(Z_i)$ (see~\eqref{eq:compound}), the  distribution function of $Z_{\rm max}=\max\{Z_1,\dots,Z_M\}$ is given by
\begin{align*}
G_{\rm max}(x)=\sum_{m=0}^\infty \pi_m\myp \bigl(G_{1/q}(x)\bigr)^m=\exp\left\{-\langle M\rangle\mynn\left(1-G_{1/q}(x)\right)\right\}=\exp\mynn\left\{-\frac{\langle M\rangle \myp\Gamma(1/q,x)}{\Gamma(1/q)}\right\}\!,
\end{align*}
which conforms with claim \eqref{eq:max1}. Derivation of \eqref{eq:min1} is similar.
\end{remark}

\section{Slow growth of the 
expected length}\label{sec:5}
Throughout this section, we impose the following condition on the growth of the hyper-parameters 
$\langle N\rangle$ and $\langle M\rangle$ (cf.\ Theorem \ref{th:cal}). 

\begin{assumption}\label{as2}
In addition to Assumption \ref{as0} stating that $\kappa=\langle M\rangle^{q+1}/\langle N\rangle\to0$ as $\langle N\rangle\to\infty$, it is assumed that
$\langle M\rangle\rightarrow \infty$.
\end{assumption}
Recall that the vector parameter $\bm{z}=(z_1,z_2)$ of the Boltzmann measure $\QQ_{\bm{z}}$ is calibrated according to Assumption~\ref{as1}. We use our standard notation $\gamma=-\log z_1\sim \langle M\rangle/(q\myn\braket{N})$ (see~\eqref{eq:gamma}).

\subsection{A joint limit theorem for the weight and length}\label{sec:5.1}
\begin{theorem}\label{th:CLT}
Under Assumptions \ref{as1} and \ref{as2}, 
define the normalized versions of $N_\lambda$ and $M_\lambda$, 
\begin{equation}\label{eq:MN-norm}
N_\lambda^*:=\frac{\sqrt{\langle M\rangle}}{\sqrt{q+1}}\left(\frac{N_\lambda-\langle N\rangle}{\langle N\rangle}\right)\!,\qquad M_\lambda^*:=\frac{M_\lambda-\langle M\rangle}{\sqrt{\langle M\rangle}} .
\end{equation}
Then both $N_\lambda^*$ and $M_\lambda^*$ are asymptotically standard normal,
\begin{equation*}
N_\lambda^* \stackrel{\mathrm{d}}{\longrightarrow} \mathcal{N}(0,1),\qquad M_\lambda^* \stackrel{\mathrm{d}}{\longrightarrow} \mathcal{N}(0,1).
\end{equation*}
Moreover, the joint limiting distribution of $N_\lambda^*$ and $M_\lambda^*$ is bivariate normal with zero mean and covariance matrix
\begin{equation}\label{eq:K}
\bm{K}_q=\begin{pmatrix}
1&\displaystyle \frac{1}{\sqrt{q+1}}\\[.5pc]
\displaystyle \frac{1}{\sqrt{q+1}}&1\end{pmatrix}\!.
\end{equation}
\end{theorem}

\begin{proof}
Consider the characteristic function of the pair $(N^*_\lambda,M^*_\lambda)$,
\begin{equation}\label{eq:phi3}
\varphi(t_1,t_2):=\EE_{\bm{z}}\bigl[\exp\myn\bigl(\rmi\mypp t_1 N^*_\lambda+\rmi\mypp t_2\myp M^*_\lambda\bigr)\bigr],\qquad t_1,t_2\in\RR.
\end{equation}
Substituting \eqref{eq:MN-norm}, this is transformed as
\begin{equation}\label{eq:phi3+}
\varphi(t_1,t_2)=\exp\myn\bigl(-\rmi\mypp \tilde{t}_1 \langle N\rangle-\rmi\mypp \tilde{t}_2\myp \langle M\rangle\bigr)
\;\EE_{\bm{z}}\bigl[\exp\myn(\rmi\mypp \tilde{t}_1 N_\lambda+\rmi\mypp \tilde{t}_2\myp M_\lambda)\bigr],
\end{equation}
where
\begin{equation}\label{eq:AB}
\tilde{t}_1= \frac{t_1\myp\sqrt{\langle M\rangle}}{\sqrt{q+1}\,\langle N\rangle}\myp, \qquad \tilde{t}_2= \frac{t_2}{\sqrt{\langle M\rangle}}\myp.
\end{equation}
Furthermore, using formulas \eqref{eq:NM}, mutual independence of the multiplicities $(\nu_\ell)$ and the Bernoulli marginals \eqref{eq:Bern}, the expectation in \eqref{eq:phi3} is rewritten as
\begin{align}
\notag
\EE_{\bm{z}}\!\left[\exp\left(\sum_{\ell\in\NN^q}\rmi\left(\tilde{t}_1\ell+\tilde{t}_2\right)\nu_\ell\right)\right]
&=\prod_{\ell\in\NN^q}\EE_{\bm{z}}\bigl[\rme^{\mypp\rmi\myp(\tilde{t}_1\ell+\tilde{t}_2)\mypp\nu_\ell}\bigr]\\
&=\prod_{\ell\in\NN^q}\frac{1+z_1^\ell z_2\,\rme^{\mypp\rmi\myp(\tilde{t}_1\ell+\tilde{t}_2)}}{1+z_1^\ell z_2}.
\label{eq:cf1}
\end{align}
Choosing the principal branch of the logarithm function $\CC\setminus\{0\}\ni \xi\mapsto \log\xi\in\CC$ (i.e., such that $\log 1=0$), we can rewrite \eqref{eq:phi3+} and \eqref{eq:cf1} 
as
\begin{equation}
\label{eq:phi4}\log\varphi(t_1,t_2) =-\rmi\mypp \tilde{t}_1 \langle N\rangle-\rmi\mypp \tilde{t}_2\myp \langle M\rangle+\sum_{\ell\in\NN^q} \log\left(1+w_\ell\right)\!,
\end{equation}
where
\begin{equation}\label{eq:w}
w_\ell\equiv w_\ell(\tilde{t}_1,\tilde{t}_2):=\frac{z_1^\ell z_2}{1+z_1^\ell z_2}\,\bigl(\rme^{\mypp\rmi\myp(\tilde{t}_1\ell+\tilde{t}_2)}-1\bigr).
\end{equation}
Remembering that $0<z_1<1$ and $z_2\to0$ (see \eqref{eq:z1} and \eqref{eq:z2}), note that, uniformly
for $\ell\in\NN^q$,
$$
|w_\ell|\le \frac{2\myp z_1^\ell z_2}{1+z_1^\ell z_2}\le \frac{2\myp  z_2}{1+ z_2}\le \frac12=:R_0,
$$
provided that  $\kappa=\langle M\rangle^{q+1}\myn/\langle N\rangle$ is small enough (cf.\  Assumption \ref{as2}). 

By Taylor's formula for complex-analytic functions (see, e.g., \cite[Sec.\,5.4, p.\,93]{WW}) applied to the function $g(w)=\log\left(1+w\right)$ with $|w|\le R_0<1$, we have
\begin{equation}\label{eq:Taylor}
g(w)=g(0)+g'(0)\mypp w
+\frac{w^2}{2\pi\rmi}\oint_{\varGamma_{\myn R}}
\frac{g(\xi)}{\xi^2\mypp(\xi-w)}\,\rmd{\xi},
\end{equation}
where $\varGamma_{\myn R}$ is the circle of radius $R\in(R_0,1)$ about the
origin, positively oriented (i.e., anti-clockwise). Noting that $|\xi-w|\ge R-R_0$ for all $\xi\in\varGamma_{\myn R}$ and $|w|\le R_0$, the remainder term in \eqref{eq:Taylor} is bounded in modulus as follows,
$$
\left|\frac{w^2}{2\pi\rmi}\oint_{\varGamma_{\myn R}}
\frac{g(\xi)}{\xi^2\mypp(\xi-w)}\,\rmd{\xi}\right|\le\frac{2\pi R}{2\pi}\left(\frac{|w|}{R}\right)^2\frac{C_{R}}{R-R_0}=\frac{|w|^2\mypp C_R}{R\,(R-R_0)},
$$
where 
$$
C_{R}:=\max\mynn\bigl\{|g(\xi)|,\,\xi\in \varGamma_{\myn R}\bigr\}<\infty.
$$
Thus, the expansion \eqref{eq:Taylor} specializes to
\begin{equation}\label{eq:Taylor-log}
\log\left(1+w\right)=w +O(|w|^2),
\end{equation}
where the $O$-term is uniform in the disk $|w|\le R_0=\frac12$.

We will also  use Taylor's expansion for the complex exponent in \eqref{eq:w}  
with a uniform bound
on the error term for any $k\in\NN$ and all $t\in\RR$  (see, e.g.,
\cite[Sec.\,XV.4, Lemma~1, p.\,512]{Feller}),
\begin{equation}\label{eq:exp-Taylor}
\Biggl|\myp\rme^{\rmi t}-\sum_{j=0}^{k-1}\frac{(\rmi\myp
t)^j}{j!}\Biggr|\le\frac{\,|t|^k}{k!\,}.
\end{equation}

Now, combining the uniform expansions  \eqref{eq:Taylor-log} and \eqref{eq:exp-Taylor} (the latter with $k=2$ or $k=1$ as appropriate) and returning to \eqref{eq:phi4}, we obtain
\begin{align}
\notag
\sum_{\ell\in\NN^q}\log\left(1+w_\ell\right)&=\sum_{\ell\in\NN^q} \frac{z_1^\ell z_2}{1+z_1^\ell z_2}\left(\rmi\mypp(\tilde{t}_1\ell+\tilde{t}_2)-\frac12\left(\tilde{t}_1\ell+\tilde{t}_2\right)^2+O(1)\left(\tilde{t}_1\ell+\tilde{t}_2\right)^3\right)\\
\notag
&\quad +O(1)\sum_{\ell\in\NN^q} \frac{z_1^{2\ell} z_2^2}{(1+z_1^\ell z_2)^2}\left(\tilde{t}_1\ell+\tilde{t}_2\right)^2\\
\notag
&=\rmi\mypp \tilde{t}_1\sum_{\ell\in\NN^q} \frac{\ell\mypp z_1^\ell z_2}{1+z_1^\ell z_2}+\rmi\mypp \tilde{t}_2\sum_{\ell\in\NN^q} \frac{z_1^\ell z_2}{1+z_1^\ell z_2} -\frac12\sum_{\ell\in\NN^q}\frac{z_1^\ell z_2}{1+z_1^\ell z_2}\left(\tilde{t}_1\ell+\tilde{t}_2\right)^2\\
\notag&\quad +O(1)\sum_{\ell\in\NN^q}z_1^\ell z_2\left(\tilde{t}_1\ell+\tilde{t}_2\right)^3 +O(1)\sum_{\ell\in\NN^q} z_1^{2\ell} z_2^2\left(\tilde{t}_1\ell+\tilde{t}_2\right)^2\\
&=:\rmi\myp\tilde{t}_1\myp\varSigma_1+\rmi\myp\tilde{t}_2\mypp\varSigma_2-\tfrac12\myp\varSigma_3+O(1)\myp\varSigma_4+O(1)\myp\varSigma_5.
\label{eq:log-expansion}
\end{align}

According to the calibration equations (see \eqref{eq:<M>} and  \eqref{eq:<N>}), the first two sums in \eqref{eq:log-expansion} are known exactly, 
\begin{equation}\label{eq:phi-sum12}
\varSigma_1=\langle N\rangle,\qquad \varSigma_2=\langle M\rangle.
\end{equation}
Next, the error sums $\varSigma_4$ and $\varSigma_5$ can be shown to asymptotically vanish. Indeed, using the elementary inequality $(a+b)^r\le 2^{r-1}\left(a^r+b^r\right)$ ($r\ge 1$) and  
combining Lemma \ref{lm:sum1} with formulas \eqref{eq:z1}, \eqref{eq:z2} and \eqref{eq:z2/gamma} gives upon simple calculations the estimate
\begin{align}\notag
0\le \varSigma_4&\le4\mypp\tilde{t}_1^{\mypp3}\sum_{\ell\in\NN^q}\ell^3z_1^\ell z_2+4\mypp\tilde{t}_2^{\mypp3}\sum_{\ell\in\NN^q}z_1^\ell z_2\\
&=\frac{O(1)\mypp\langle M\rangle^{3/2}\myp z_2}{\langle N\rangle^3\mypp\gamma^{3+1/q}}+\frac{O(1)\mypp z_2}{\langle M\rangle^{3/2}\mypp\gamma^{1/q}}=\frac{O(1)}{\langle M\rangle^{1/2}}=o(1).
\label{eq:phi-sum4}
\end{align}
Similarly,
\begin{align}\notag
0\le \varSigma_5&\le2\mypp\tilde{t}_1^{\mypp2}\sum_{\ell\in\NN^q}\ell^2z_1^{2\ell} z_2^2+2\mypp\tilde{t}_2^{\mypp2}\sum_{\ell\in\NN^q}z_1^{2\ell} z_2^2\\
&=\frac{O(1)\mypp\langle M\rangle\myp z_2^2}{\langle N\rangle^2\mypp\gamma^{2+1/q}}+\frac{O(1)\mypp z_2^2}{\langle M\rangle\mypp\gamma^{1/q}}=O(\kappa^{1/q})=o(1).
\label{eq:phi-sum5}
\end{align}

Finally, consider the sum $\varSigma_3$ in \eqref{eq:log-expansion} which, as we will see, provides the main contribution to \eqref{eq:phi4}.
To this end, observe (cf.\ \eqref{eq:<M>}) that  
\begin{align}
\notag
0\le \sum_{\ell\in\NN^q} z_1^\ell z_2\left(\tilde{t}_1\ell+\tilde{t}_2\right)^2-\varSigma_3&=\sum_{\ell\in\NN^q}\frac{z_1^{2\ell} z_2^2}{1+z_1^\ell z_2}\left(\tilde{t}_1\ell+\tilde{t}_2\right)^2\\
&\le \sum_{\ell\in\NN^q}z_1^{2\ell} z_2^2\left(\tilde{t}_1\ell+\tilde{t}_2\right)^2=\varSigma_5=o(1),
\label{eq:phi-sum31}
\end{align}
as shown in \eqref{eq:phi-sum5}. In turn, using the asymptotic results \eqref{eq:Sum-22}, \eqref{eq:Sum-11} and \eqref{eq:Sum-12}, and recalling the rescaling 
expressions \eqref{eq:AB}, we obtain the limit
\begin{align}\notag
\sum_{\ell\in\NN^q} z_1^\ell z_2\left(\tilde{t}_1\ell+\tilde{t}_2\right)^2&=\tilde{t}_1^{\mypp2}\sum_{\ell\in\NN^q} \ell^2z_1^\ell z_2+2\mypp\tilde{t}_1\tilde{t}_2\sum_{\ell\in\NN^q} \ell\myp z_1^\ell z_2+\tilde{t}_2^{\mypp2}\sum_{\ell\in\NN^q} z_1^\ell z_2\\
&\to t_1^2+\frac{2\mypp t_1t_2}{\sqrt{q+1}}+t_2^2\mypp,
\label{eq:phi-sum32}
\end{align}
which is a quadratic form with matrix \eqref{eq:K}.

Thus, substituting the estimates \eqref{eq:phi-sum12}, \eqref{eq:phi-sum4},
\eqref{eq:phi-sum5}, \eqref{eq:phi-sum31} and \eqref{eq:phi-sum32} into \eqref{eq:log-expansion} and returning to \eqref{eq:phi4}, we obtain
$$
\varphi(t_1,t_2)\to \exp\left\{-\frac12\left( t_1^2+\frac{2\mypp t_1t_2}{\sqrt{q+1}}+t_2^2\right)\right\}\!,
$$
and the proof of Theorem \ref{th:CLT} is complete.
\end{proof}

\begin{corollary}\label{cor:5.4}
Under the hypotheses of Theorem \ref{th:CLT}, the following laws of large numbers hold under the Boltzmann measure $\QQ_{\bm{z}}$,
\begin{equation}\label{eq:LLN-MN}
\frac{M_\lambda}{\langle M\rangle}\stackrel{\mathrm{p}}{\longrightarrow} 1, \qquad \frac{N_\lambda}{\langle N\rangle}\stackrel{\mathrm{p}}{\longrightarrow} 1.
\end{equation}
\end{corollary}

\begin{remark}\label{rm:link1}
As a curiosity, we observe that the limiting distribution in Theorem \ref{thm1} formally conforms to Theorem~\ref{th:CLT} under the additional limit as $\langle M\rangle \to\infty$. Indeed, start with the intermediate limit \eqref{eq:phi-lim} (with $\langle M\rangle$ fixed) and switch from Laplace transform to  characteristic function by formally changing $(s_1,s_2)$ ($s_i\ge0$) to $-\rmi\mypp (t_1,t_2)$ ($t_i\in\RR$). Then, bearing in mind the normalization \eqref{eq:MN-norm}, we obtain
\begin{align}
-\langle M\rangle\! \left(1-\rme^{\myp\rmi\myp t_2/\sqrt{\langle M\rangle}}\left(1-\frac{\rmi\mypp t_1 q}{\sqrt{\langle M\rangle}}\right)^{-1/q}\right)+\frac{\rmi\mypp t_1 \sqrt{\langle M\rangle}}{\sqrt{q+1}}&\to-\frac12\left( t_1^2+\frac{2\mypp t_1t_2}{\sqrt{q+1}}+t_2^2\right)\!,
\label{eq:lim2}
\end{align}
by Taylor expanding the left-hand side of \eqref{eq:lim2} up to the second order in parameter $\langle M\rangle^{-1/2}=o(1)$. 
\end{remark}

We finish this section by stating a theorem concerning the marginal local-type asymptotics for the length $M_\lambda$ and the corresponding conditional asymptotics for the weight $N_\lambda$. 
\begin{theorem}\label{thm1+}
Under Assumptions \ref{as1} and \ref{as2}, the following distributional asymptotics hold under the Boltzmann distribution $\QQ_{\bm{z}}$ on the space $\check{\varLambda}^q$. 
\begin{itemize}
\item[\rm (a)] \textup{(Local limit theorem for $M_\lambda$)}
For any $m$ such that $m-\langle M\rangle= O\bigl(\sqrt{\langle M\rangle}\myp\bigr)$, 
\begin{equation}\label{cvM+}
Q_{\bm{z}}(M_\lambda=m)\sim f_{\langle M\rangle}(m),
\end{equation}
where $f_{\langle M\rangle}(x)$ is the normal density with mean $\langle M\rangle$ and standard deviation  $\sqrt{\langle M\rangle}$,
$$
f_{\langle M\rangle}(x)=\frac{1}{\sqrt{2\pi\langle M\rangle}}\,\exp\left\{-\frac{(x-\langle M\rangle)^2}{2\myp\langle M\rangle}\right\}\!,\qquad x\in\RR.
$$
\item[\rm (b)]
\textup{(Conditional limit theorem for $N_\lambda$)} Conditionally on $M_\lambda=m$ with $m-\langle M\rangle = O\bigl(\sqrt{\langle M\rangle}\myp\bigr)$, the partition weight $N_\lambda$ is asymptotically normal with mean $m\myp\langle N\rangle/\langle M\rangle$ and variance $q\mypp\langle N\rangle^2/\langle M\rangle$, that is,
\begin{equation}\label{cvN2}
\frac{\sqrt{\langle M\rangle}}{\sqrt{q}}\left(\frac{N_\lambda}{\langle N\rangle}-\frac{m}{\langle M\rangle}\right) \stackrel{\mathrm{d}}{\longrightarrow}_{M_\lambda=m} \mathcal{N}(0,1). 
\end{equation}
\end{itemize}
\end{theorem}

Part (a) of this theorem can be anticipated from the marginal asymptotic normality of $M_\lambda$ proven in Theorem~\ref{th:CLT}. The result in part (b) can be formally  derived from the joint asymptotic normality of the pair $(N_\lambda,M_\lambda)$ by calculating the conditional density of $N_\lambda$ given $M_\lambda$. Alternatively, one can refer to the ``normal correlation'' result (see, e.g., \cite[Sec.\,II.13, Theorem 2, pp.\,303--304]{Shiryaev}). A rigorous proof of Theorem~\ref{thm1+} can be obtained by adapting the standard proof of a ``one-dimensional" local limit theorem for $N_\lambda$ (see, e.g., \cite{Bogachev}, \cite{VFY}). We will return to this issue in a separate publication.

\subsection{Limit shape of Young diagrams}\label{sec:5.2}

In this section we show that, under the slow growth condition on $\langle M\rangle$, properly scaled Young diagrams of random partitions $\lambda\in\check{\varLambda}^q$ have a limit shape given by the tail of the gamma integral, \begin{equation}\label{eq:omega}
\omega_q^*(x):=\frac{1}{\Gamma(1/q)}\int_x^\infty \!u^{1/q-1}\mypp\rme^{-u}\,\rmd{u}=1-G_{1/q}(x),
\qquad x\ge0,
\end{equation}
where $G_{1/q}(x)$ is the distribution function of $\mathrm{Gamma}\mypp(1/q)$ (see \eqref{eq:gamma-cdf}). In particular, for $q=1$ the definition \eqref{eq:omega} is reduced to 
$$
\omega_1^*(x)=\rme^{-x},\qquad x\ge0.
$$
Specifically, set 
\begin{equation}\label{eq:ab}
A:=q\mypp\langle N\rangle/\langle M\rangle,\qquad B=\langle M\rangle,
\end{equation}
and consider a scaled Young diagram with upper boundary
\begin{equation}\label{eq:Y-tilde}
\widetilde{Y}_\lambda(x) = B^{-1}\mypp Y_\lambda(Ax), \qquad x\ge0,
\end{equation}
where (see \eqref{eq:Young})
\begin{equation}\label{eq:Young1}
Y_\lambda(x)=\sum_{\ell\ge x}\nu_\ell,
\qquad x\ge0.
\end{equation}
\begin{remark}
The area under the scaled Young diagram is given by
$$
\int_0^\infty \mynn\widetilde{Y}_\lambda(x)\,\rmd{x}=B^{-1}\int_0^\infty \mynn Y_\lambda(Ax)\,\rmd{x}=\frac{\langle N\rangle}{AB}=\frac{1}{q}.
$$
Naturally, this condition is preserved by the limit shape; indeed, integrating by parts we get
$$
\int_0^\infty\!\omega_q^*(x)\,\rmd{x}=\frac{1}{\Gamma(1/q)}\int_0^\infty x^{1/q}\mypp\rme^{-x}\,\rmd{x}=\frac{\Gamma(1+1/q)}{\Gamma(1/q)}=\frac{1}{q}.
$$
\end{remark}

First, we obtain the \emph{expected limit shape} result. 
\begin{theorem}\label{th:exp-LS}
Under Assumptions \ref{as1} and \ref{as2}, uniformly in $x\ge0$ 
\begin{equation}\label{sh}
\EE_{\bm{z}}\myn\bigl(\widetilde{Y}_\lambda(x)\bigr)\to \omega_q^*(x),
\end{equation}where the limit shape $x\mapsto \omega_q^*(x)$ is defined in (\ref{eq:omega}).
\end{theorem}

\begin{proof}
We first show that for each $x\ge0$, the \strut{}convergence \eqref{sh} holds. 
By Lemma \ref{distnu},
\begin{align}
\EE_{\bm{z}}\bigl(Y_\lambda(Ax)\bigr)=\sum_{\ell\ge Ax}\EE_{\bm{z}}(\nu_\ell)&=\sum_{\ell\ge Ax}\frac{z_1^{\ell} z_2}{1+z_1^{\ell} z_2}=\sum_{\ell\ge Ax}z_1^{\ell} z_2- \widetilde{R}_1(\bm{z}),
\label{eq:E(Y)}
\end{align}
where (cf.\  \eqref{eq:<M>} and \eqref{eqR1})
\begin{equation}
\label{eq:R1}
0\le \widetilde{R}_1(\bm{z})=\sum_{\ell\ge Ax}\frac{z_1^{2\ell}\myn z_2^2 }{1+z_1^{\ell} z_2}\le \sum_{\ell\in\NN^q} z_1^{2\ell}\myn z_2^2=\braket{M}O(\kappa^{1/q}),
\end{equation}
by virtue of \eqref{eq:sum_q} (with $r=2$). Furthermore, applying Lemma \ref{lm:sum1+} (with $\gamma=-\log z_1$ and $s=0$) and noting that $\gamma A\to1$, we obtain
\begin{equation}\label{eq:sum1}
\sum_{\ell\ge Ax}z_1^{\ell} z_2\sim \frac{z_2}{q\,\gamma^{1/q}}\int_{\gamma Ax}^\infty\mynn  u^{1/q-1}\mypp \rme^{-u}\,\rmd{u} \sim\frac{\langle M\rangle}{\Gamma(1/q)}\int_{x}^\infty \!u^{1/q-1}\mypp\rme^{-u}\,\rmd{u},
\end{equation}
using the asymptotic formula \eqref{eq:z2/gamma}, Thus, substituting \eqref{eq:R1} and \eqref{eq:sum1} into \eqref{eq:E(Y)} gives
\begin{align*}
\EE_{\bm{z}}\bigl(\widetilde{Y}_\lambda(x)\bigr)=\frac{1}{\langle M\rangle}\,\EE_{\bm{z}}\bigl(Y_\lambda(Ax)\bigr)
\sim\frac{1}{\Gamma(1/q)}\int_{x}^\infty \!u^{1/q-1}\mypp\rme^{-u}\,\rmd{u}=\omega_q^*(x),
\end{align*}
as claimed. Finally, the uniform convergence in formula \eqref{sh} follows by Lemma \ref{lm:uniform}, noting that the function $x\mapsto\omega_q^*(x)$ is continuous, bounded, and decreasing on $[0,\infty)$.
 \end{proof}

Now, we are ready to state and prove the main result of this section.
\begin{theorem}\label{th:LS}
Under Assumptions \ref{as1} and \ref{as2}, the rescaled Young diagrams converge to the limit shape $y=\omega_{q}^*(x)$ in $\QQ_{\bm{z}}$-probability uniformly for $x\ge0$, that is, 
\begin{equation*}
\QQ_{\bm{z}}\!\left(\lambda\in\check{\varLambda}^q\colon
\sup_{x\ge0}\,\bigl|\myp \widetilde{Y}_\lambda(x)-\omega_q^*(x)\bigr|>\varepsilon\right)\to 0,
\end{equation*}
where $\widetilde{Y}_\lambda(x)$ is defined in (\ref{eq:Y-tilde}) with the aid of (\ref{eq:ab}).
\end{theorem}

\begin{proof}
By virtue of Theorem \ref{th:exp-LS}, letting
$Y^0_\lambda(x):=Y_\lambda(x)-\EE_{\bm{z}}\bigl(Y_\lambda(x)\bigr)$ it suffices to
check that
\begin{equation}\label{eq:Q->0}
\QQ_{\bm{z}}\!\left( \sup_{x\ge0}
\,\bigl|Y^{0}_\lambda(Ax)\bigr|>
B\myp\varepsilon\right)\to0.
\end{equation}
Put $Z_\lambda(x):=Y_\lambda(x^{-1})$ for $0\le x\le \infty$; in particular,  $Z_\lambda(0)=Y_\lambda(\infty)=0$,  $Z_\lambda(\infty)=Y_\lambda(0)=M_\lambda$. 
By the definition
\eqref{eq:Young}, for any $0<s<t\le\infty$ we have
$$
Z_\lambda(t)-Z_\lambda(s)=Y_\lambda(t^{-1})-Y_\lambda(s^{-1})=\sum_{t^{-1}\le
\myp\ell\myp<s^{-1}}\!\nu_\ell\mypp,
$$
which implies that the random process $Z_\lambda(x)$ ($x\ge0$) has
independent increments. Hence,
$Z^0_\lambda(x):=Z_\lambda(x)-\EE_{\bm{z}}\bigl(Z_\lambda(x)\bigr)$ is a martingale with respect to the filtration ${\mathcal
F}_x=\sigma\{\nu_\ell\myp,\, \ell\ge x^{-1}\}$. From
\eqref{eq:Young} it is also evident that $Z^0_\lambda(x)$ is
\emph{c\`adl\`ag} (i.e., its paths are everywhere right-continuous
and have left limits, cf.\ Figure \ref{fig:yd}). Therefore, by the Doob--Kolmogorov submartingale inequality (see, e.g.,
\cite[Theorem~6.16, p.\mypp 101]{Yeh})
we obtain
\begin{align}
\notag
 \QQ_{\bm{z}}\!\left( \sup_{x\ge0}
\,\bigl|Y^{0}_\lambda(Ax)\bigr|>B\myp\varepsilon\right)
&\equiv \QQ_{\bm{z}}\!\left( \sup_{y\le\infty}
\mynn|Z^{0}_\lambda(y\myp
A^{-1})|>B\myp\varepsilon\right)\\[.2pc]
&\le \frac{\Var_{\bm{z}}\bigl(Z_\lambda(\infty)\bigr)}{B^2\myp\varepsilon^2
}= \frac{\Var_{\bm{z}}\bigl(Y_\lambda(0)\bigr)}{B^2\myp\varepsilon^2}.
\label{eq:Var+}
\end{align}
Recalling that  $Y_\lambda(0)=M_\lambda$ and using Theorem~\ref{th:VarCov},
the right-hand side of \eqref{eq:Var+} is estimated by 
$O(\langle M\rangle^{-1})$. Thus, the claim \eqref{eq:Q->0}
follows and the proof of Theorem~\ref{th:LS} is complete.
\end{proof}

Convergence of normalized Young diagrams to their limits shape is illustrated in Figure \ref{fig:yd} for $q=1$ and $q=2$. Random partitions were simulated using a suitable Boltzmann sampler implemented as Algorithm \ref{sampler} (see Section \ref{sec:6.1}).

\begin{figure}[ht!]
\centering
\includegraphics[width=0.6\textwidth]{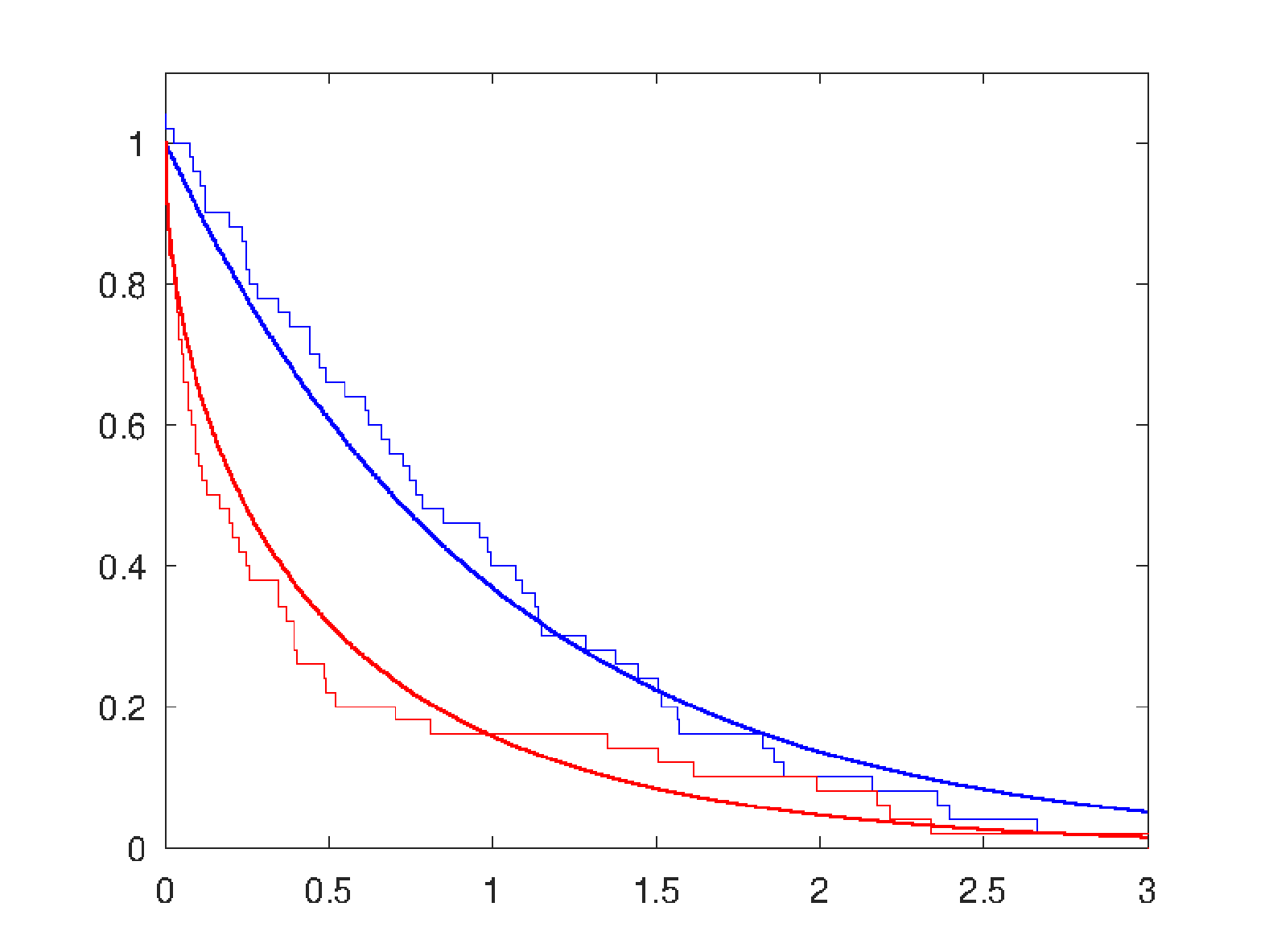}
\put(-220,40){\small$q=2$}
\put(-157,88){\small$q=1$}
\put(-130,160){\small\fbox{\!\!\begin{tabular}{l}
$\displaystyle\omega_1^*(x)=\rme^{-x}$\\[.2pc]
$\omega^*_2(x)=
1-G_{1/2}(x)$
\end{tabular}\!\!\!}}
\caption{Illustration of convergence to the limit shape for $q=1$ and $q=2$ (in the online version shown in blue and red, respectively). The step plots depict the 
upper boundary of the scaled Young diagrams (see~\eqref{eq:Y-tilde}), 
while the smooth lines represent the limit shape $\omega_q^*(x)=1-G_{1/q}(x)$
(see~\eqref{eq:omega}).
The corresponding partitions $\lambda\in\check{\varLambda}^q$ were sampled using Algorithm~\ref{sampler} with hyper-parameters $\braket{M}=50$ and $\braket{N}=2.5\cdot 10^5$ ($q=1$) or $\braket{N}=1.25\cdot 10^7$ ($q=2)$; in both cases, $\kappa=0.01$ (cf.\ Assumption~\ref{as1}). The respective sample weight and length are $N_\lambda=236{,}\myp369$, $M_\lambda=52$ ($q=1$)  and $N_\lambda=12{,}\myp 733{,}\myp 323$, $M_\lambda=45$ ($q=2$).}\label{fig:yd}
\end{figure}

Finally, we can analyze asymptotic fluctuations of scaled Young diagrams. Recall that $\widetilde{Y}_\lambda(x)$ is defined in (\ref{eq:Y-tilde}) with the aid of \eqref{eq:ab}.
\begin{theorem}\label{th:Gauss}
Under Assumptions \ref{as1} and \ref{as2}, for any $x>0$ the random value $\widetilde{Y}_\lambda(x)$ is asymptotically normal with 
variance $\omega_{q}^*(x)/\langle M\rangle$, that is, 
\begin{equation}\label{eq:Y*(x)}
\widetilde{Y}^*_\lambda(x):=\sqrt{\braket{M}}\left(\widetilde{Y}_\lambda(x)-\EE_{\bm{z}}\bigl(\widetilde{Y}_\lambda(x)\bigr)\right)\stackrel{\mathrm d}{\longrightarrow} \mathcal{N}\bigl(0,\omega_{\smash{q}}^*(x)\bigr).
\end{equation}
\end{theorem}
\begin{proof}
Consider the characteristic function of $\widetilde{Y}^*_\lambda(x)$,
\begin{equation}\label{eq:phi3Y}
\varphi_{\bm{z}}(t;x):=\EE_{\bm{z}}\bigl[\exp\myn\bigl(\rmi\mypp t\mypp \widetilde{Y}^*_\lambda(x)\bigr)\bigr],\qquad t\in\RR.
\end{equation}
Substituting the definition \eqref{eq:Y*(x)} and using  \eqref{eq:ab}, \eqref{eq:Y-tilde} and \eqref{eq:Young1}, this is transformed as
\begin{equation}\label{eq:phi3'}
\varphi_{\bm{z}}(t;x)=\exp\myn\bigl\{-\rmi\mypp \tilde{t}\, \EE_{\bm{z}}\bigl(Y_\lambda(Ax)\bigr)\mypp\bigr\}
\,\EE_{\bm{z}}\bigl[\exp\myn\bigl(\rmi\mypp 
\tilde{t}\, Y_\lambda(Ax)\bigr)\bigr],\qquad \tilde{t}:= \frac{t}{\sqrt{\braket{M}}}\myp,
\end{equation}
and furthermore (see \eqref{eq:E(Y)}) 
\begin{equation}\label{eq:E(Y)1}
\EE_{\bm{z}}\bigl(Y_\lambda(Ax)\bigr)=\sum_{\ell\ge Ax}\frac{z_1^{\ell} z_2}{1+z_1^{\ell} z_2}.
\end{equation}

Next, similarly to \eqref{eq:cf1} and \eqref{eq:phi4}  the last expectation in \eqref{eq:phi3'} is expressed as
\begin{equation}
\label{eq:cf1'}
\EE_{\bm{z}}\!\left[\exp\left(\rmi\mypp \tilde{t}\sum_{\ell\ge Ax}\nu_\ell\right)\right]
=\prod_{\ell\ge Ax}\frac{1+z_1^\ell z_2\,\rme^{\mypp\rmi\myp\tilde{t}}}{1+z_1^\ell z_2}=\prod_{\ell\ge Ax}\bigl(1+w_\ell(\tilde{t}\myp)\bigr),
\end{equation}
where
\begin{equation}\label{eq:w-}
w_\ell(t):=\frac{z_1^\ell z_2}{1+z_1^\ell z_2}\,(\rme^{\mypp\rmi\myp t}-1).
\end{equation}
Choosing the principal branch of the logarithm and using \eqref{eq:E(Y)1} and \eqref{eq:cf1'}, from \eqref{eq:phi3'} we get
\begin{equation}
\label{eq:phi4'} \log\varphi_{\bm{z}}(t;x) =-\rmi\mypp \tilde{t}\sum_{\ell\ge Ax}\frac{z_1^{\ell} z_2}{1+z_1^{\ell} z_2}+\sum_{\ell\ge Ax} \log\bigl(1+w_\ell(\tilde{t}\myp)\bigr).
\end{equation}
In turn, similarly to \eqref{eq:log-expansion} we obtain
\begin{equation}
\label{eq:log-expansion'}
 \sum_{\ell\ge Ax}\log\bigl(1+w_\ell(\tilde{t}\myp)\bigr)=\Bigl(\rmi\mypp\tilde{t}-\tfrac12\myp \tilde{t}^2+O(\tilde{t}^3)\Bigr)\sum_{\ell\ge Ax} \frac{z_1^\ell z_2}{1+z_1^\ell z_2} +O(\tilde{t}^2)\sum_{\ell\ge Ax} \frac{z_1^{2\ell} z_2^2}{(1+z_1^\ell z_2)^2}.
\end{equation}
As was shown in the proof of Theorem \ref{th:exp-LS} (see \eqref{eq:E(Y)}, \eqref{eq:R1}) and \eqref{eq:sum1}),
\begin{equation}\label{eq:R4}
\sum_{\ell\ge Ax} \frac{z_1^\ell z_2}{1+z_1^\ell z_2}\sim \langle M\rangle\,\omega_q^*(x),\qquad \sum_{\ell\ge Ax} \frac{z_1^{2\ell} z_2^2}{(1+z_1^\ell z_2)^2}
=\braket{M}O(\kappa^{1/q})
=\braket{M}o(1).
\end{equation}
Using \eqref{eq:R4} and recalling the notation of $\tilde{t}$ in \eqref{eq:phi3'}, from \eqref{eq:log-expansion'} we get
$$
\sum_{\ell\ge Ax}\log\bigl(1+w_\ell(\tilde{t}\myp)\bigr)=\rmi\mypp\tilde{t}\sum_{\ell\ge Ax} \frac{z_1^\ell z_2}{1+z_1^\ell z_2} -\tfrac12\myp t^2\omega_q^*(x)+o(1).
$$
Finally, returning to 
\eqref{eq:phi4'}, we see that $\log \varphi_{\bm{z}}(t;x)\to -\tfrac12\myp t^2\omega_q^*(x)$, 
which proves the theorem.
\end{proof}

\begin{remark}
The reason for using in \eqref{eq:Y*(x)} the intrinsic centering $\EE_{\bm{z}}\bigl(\widetilde{Y}_\lambda(x)\bigr)$ 
rather than the limit shape value $\omega_q^*(x)$ is that the error terms in the asymptotic estimates \eqref{eq:R4} are of order $\braket{M}\kappa^{1/q}$, where $\kappa=\langle M\rangle^{q+1}/\langle N\rangle=o(1)$ (see Assumption~\ref{as2}). Combined with the factor $\tilde{t}=O\bigl(\braket{M}^{-1/2}\bigr)$, this produces the error bound of order $\langle M\rangle^{1/2}\mypp\kappa^{1/q}$, which is not guaranteed to be small. Thus, a stronger assumption to this end is $\langle M\rangle^{1/2}\myp\kappa^{1/q}=o(1)$, that is, $\langle M\rangle^{1+3\myp q/2}/\langle N\rangle=o(1)$. On the other hand, lifting any control over the length may restore the limit-shape centering; for example, for $q=1$ (ordinary strict partitions), a central theorem of that kind was proved in~\cite{VFY}.   
\end{remark}

Theorem \ref{th:Gauss} can be extended to incorporate convergence of finite-dimensional distributions. 

\begin{theorem}\label{th:Gauss-fidi}
Under Assumptions \ref{as1} and \ref{as2}, the random process $\bigl(\widetilde{Y}^*_\lambda(x),\,x>0\bigr)$ defined in (\ref{eq:Y*(x)}) converges, in the sense of convergence of finite-dimensional distributions, to a Gaussian random process $(\varXi(x),\,x > 0)$ with zero mean and the covariance
function
\begin{equation}
\label{eq:K1}
K(x,x'):=\Cov\bigl(\varXi(x),\varXi(x')\bigr)=\min\{\omega_{\smash{q}}^*(x), \omega_{\smash{q}}^*(x')\},\qquad 
x,x'>0.
\end{equation}
In turn, the process $(\varXi(x))$ has the same  distribution as the Brownian motion $(B_{\omega_q^*(x)})$, where  $(B_t)$ is a standard Brownian motion, i.e.\ with mean zero and covariance function $\Cov(B_t,B_{t'})=\min\myn\{t,t'\}$.
\end{theorem}

\begin{proof}
Like in the proof of Theorem \ref{th:LS}, observe that for $0<x_1< x_2<\dots < x_m$ the increments
$$
\widetilde{Y}^*_\lambda(x_1)-\widetilde{Y}^*_\lambda(x_2),\ \dots,\ \ \widetilde{Y}^*_\lambda(x_{m-1})-\widetilde{Y}^*_\lambda(x_{m}),\ \ \widetilde{Y}^*_\lambda(x_m)
$$
are mutually independent due to independence of the multiplicities $(\nu_\ell)$ under the Boltzmann measure $\PP_{\bm{z}}$ (see \eqref{eq:Young1}). In particular, noting that for each $k\ge 1$ 
$$
\bigl(\widetilde{Y}^*_\lambda(x_{k}) -\widetilde{Y}^*_\lambda(x_{k+1})\bigr)+\widetilde{Y}^*_\lambda(x_{k+1}) =\widetilde{Y}^*_\lambda(x_{k}),
$$
independence implies that the characteristic function of $\widetilde{Y}^*_\lambda(x_{k}) -\widetilde{Y}^*_\lambda(x_{k+1})$ is given by the ratio $\varphi_{\bm{z}}(t;x_{k})/\varphi_{\bm{z}}(t;x_{k+1})$ (see the notation \eqref{eq:phi3Y}), and  by Theorem \ref{th:Gauss} it follows that  
$$
\widetilde{Y}^*_\lambda(x_{k}) -\widetilde{Y}^*_\lambda(x_{k+1})
\stackrel{\mathrm d}{\longrightarrow} \mathcal{N}\bigl(0,\,\omega_{\smash{q}}^*(x_{k})-\omega_{\smash{q}}^*(x_{k+1})\bigr)\qquad (k\ge1).
$$
Applying a suitable linear transformation,
we conclude that the vector 
$\bigl(\widetilde{Y}^*_\lambda(x_{1}),\dots,\allowbreak \widetilde{Y}^*_\lambda(x_{m})\bigr)$ converges in distribution to a normal vector $\bigl(\varXi(x_1),\dots,\allowbreak \varXi(x_m)\bigr)$, such that $\EE\bigl(\varXi(x_k)\bigr)=0$ ($k=1,\dots,m$) and, for $1\le j\le k\le m$,  
$$
\Cov\bigl(\varXi(x_j),\varXi(x_k)\bigr)=\Var\bigl(\varXi(x_k)\bigr)=\omega_{\smash{q}}^*(x_k)=\min\{\omega_{\smash{q}}^*(x_j), \omega_{\smash{q}}^*(x_k)\}
$$
(cf.\ \eqref{eq:K1}), because $x\mapsto\omega_{\smash{q}}^*(x)$ is a decreasing function. 

Finally,  notice that the process $\bigl(B_{\omega_q^*(x)}\bigr)$ has zero mean and the same covariance function,
$$
\Cov\bigl(B_{\omega_q^*(x)}, B_{\omega_q^*(x')}\bigr)=\min\{\omega_q^*(x),\omega_q^*(x')\},\qquad x,x'>0,
$$
which completes the proof. 
\end{proof}

\subsection{A joint limit theorem for the extreme parts (growing expected length)}\label{sec:5.4}
We use the notation $\lambda_{\rm min}$ and $\lambda_{\rm max}$ (cf.\ Section \ref{sec:4.3}). 
\begin{theorem}\label{th:maxM}
Under Assumptions \ref{as1} and \ref{as2}, set
\begin{equation}\label{eq:ab+}
b_q:=\left(\frac{q\myn\braket{M}}{\displaystyle\Gamma(1/q)}\right)^{\myn q}\!,\qquad B_q:=\log\myn \langle M\rangle-\bigl(1-\tfrac{1}{q}\bigr)\log \log\myn\langle M\rangle-\log\Gamma(1/q),
\end{equation}
and consider the normalized versions of $\lambda_{\rm min}$ and $\lambda_{\rm max}$ defined as follows,
\begin{equation}\label{eq:lambda1*}
\lambda_{\rm min}^*:=\gamma\mypp b_q\mypp \lambda_{\rm min},\qquad \lambda_{\rm max}^*:=\gamma\myp \lambda_{\rm max}-B_q\mypp.
\end{equation} 
Then $\lambda^*_{\rm min}$ and $\lambda^*_{\rm max}$ are asymptotically independent under the measure $\QQ_{\bm{z}}$ as $\braket{N}\to\infty$ and their marginal limiting laws are given, respectively, by a Weibull distribution with shape parameter $1/q$ and the standard double-exponential (Gumbel) distribution, 
\begin{alignat}{3}
\label{eq:min-lim}
\QQ_{\bm{z}}(\lambda_{\rm min}^*> x_1)&\to \exp\myn\bigl(-x_1^{1/q}\bigr),\qquad& x_1\ge0,&\\[.3pc]   
\label{eq:max-lim}
\QQ_{\bm{z}}(\lambda_{\rm max}^*\le x_2)&\to \exp\myn\bigl(-\rme^{-x_2}\bigr),&x_2\in\RR.&
\end{alignat}
\end{theorem}
\begin{proof}
For $x_1\ge0$ and $x_2\in\RR$, set 
$$
\ell^*_{1}(x_1):=\min\{\ell\in\NN^q\colon \ell > x_1/(\gamma\mypp b_q)\},\qquad \ell^*_{2}(x_2):=\min\{\ell\in\NN^q\colon \ell > (B_q+x_2)\mypp \gamma^{-1}\}.
$$ 
Recalling the asymptotic relations \eqref{eq:z1} and \eqref{eq:kappa}, observe that  
\begin{align}
\label{eq:ell*gamma1}
\gamma\mypp\ell^*_{1}(x_1)&=\frac{x_1}{b_q}+O(\gamma)\sim x_1\left(\frac{\Gamma(1/q)}{\displaystyle q\braket{M}}\right)^{\myn q}\!,\\
\label{eq:ell*gamma2}
\gamma\mypp \ell^*_{2}(x_2)&=B_q+x_2+O(\gamma)\sim \log\braket{M}\mynn.
\end{align}  
Like in the proof of Theorem \ref{th:minmax}, we have
\begin{align*}
\QQ_{\bm{z}}(\lambda_{\rm min}^*>x_1,\,\lambda_{\rm max}^*\le x_2) &= \exp\Biggl\{-\Biggl(\myp\sum_{\ell\in\NN^q}^\infty -\sum_{\ell\ge\ell^*_{1}(x_1)}^\infty +\sum_{\ell\ge\ell^*_{2}(x_2)}^\infty\Biggr) \log\mynn\!\left(1+z_1^{\ell}z_2 \right) \Biggr\}.
\label{eq:ell>ell*}
\end{align*}
Applying Lemmas \ref{lm:sum2} and \ref{lm:sum2+} (with $\gamma=-\log z_1$ and $\eta=z_2$) and using the asymptotic relation \eqref{eq:z2/gamma}, we obtain
\begin{equation}
\label{eq:tail-exp}
-\log \QQ_{\bm{z}}(\lambda^*_{\rm min}> x_1,\,\lambda^*_{\rm max}\le x_2)\sim \frac{\braket{M}}{\Gamma(1/q)}\left(\int_{0}^{\gamma\myp\ell^*_{1}(x_1)}\!+\int_{\gamma\myp\ell^*_{2}(x_2)}^\infty \right) u^{1/q-1}\mypp\rme^{-u}\,\rmd{u}.
\end{equation}
Integrating by parts and using the asymptotic relation \eqref{eq:ell*gamma1}, we obtain
\begin{equation}\label{eq:int1min}
\int_{0}^{\gamma\myp\ell^*_{1}(x_1)}\mynn u^{1/q-1}\mypp\rme^{-u}\,\rmd{u}=q\int_{0}^{\gamma\myp\ell^*_{1}(x_1)}\!\rme^{-u}\,\rmd{(u^{1/q})}\sim q\mypp\bigl(\gamma\mypp\ell^*_{1}(x_1)\bigr)^{1/q}\sim\frac{\Gamma(1/q)\,x_1^{1/q}}{\displaystyle \braket{M}}.
\end{equation}
Next, using \eqref{eq:ell*gamma2} we get
\begin{align}
\notag
\int_{\gamma\myp\ell^*_{2}(x_2)}^\infty \mynn u^{1/q-1}\mypp\rme^{u}\,\rmd{u}  &\sim \bigl(\gamma\mypp\ell^*_{2}(x_2)\bigr)^{1/q-1}\mypp \rme^{-\gamma\myp\ell^*_{2}(x_2)}\\
&\sim B_q^{1/q-1}\mypp\rme^{-(B_q+x_2)}\sim\frac{\Gamma(1/q)\,\rme^{-x_2}}{\braket{M}} .
\label{eq:int1max}
\end{align}
Hence, substituting \eqref{eq:int1min} and \eqref{eq:int1max} into \eqref{eq:tail-exp} yields
\begin{align*}
-\log \QQ_{\bm{z}}(\lambda^*_{\rm min}> x_1,\,\lambda^*_{\rm max}\le x_2)\sim x_1^{1/q}+\rme^{-x_2},
\end{align*}
which completes the proof of the theorem.
\end{proof}

\begin{remark}
The necessary use of the intrinsic calibration parameter $\gamma=-\log z_1$ in Theorem \ref{th:maxM} may be a little disappointing. This can be easily improved under a slightly stronger condition on slow growth of $\langle M\rangle$ than in Assumption \ref{as2}; namely, $\kappa^{1/q}\log\langle M\rangle =o(1)$, that is, $\langle M\rangle^{q+1} \mypp(\log\mynn\langle M\rangle)^q/\langle N\rangle =o(1)$. In this case, the normalization \eqref{eq:lambda1*} can be written more explicitly by replacing $\gamma$ with $\gamma_0=\langle M\rangle/(q\mypp\langle N\rangle)$ (see~\eqref{eq:gamma}). \end{remark}

\begin{corollary}\label{cor:5.8}
Under the hypotheses of Theorem \ref{th:maxM}, the following law of large numbers holds, 
\begin{equation}\label{eq:LLN-maxmed}
\frac{\langle M\rangle\myp\lambda_{\rm max}}{q\mypp \langle N\rangle \log\mynn \langle M\rangle}\stackrel{\mathrm{p}}{\longrightarrow} 1,
\end{equation}
where the symbol $\stackrel{\mathrm{p}}{\longrightarrow}$ indicates convergence in $\QQ_{\bm{z}}$-probability. 
\end{corollary}

\begin{proof}
Theorem \ref{th:maxM} implies that $\gamma\myp\lambda_{\rm max}/B_q\stackrel{\mathrm{p}}{\longrightarrow} 1$, and the claim \eqref{eq:LLN-maxmed} follows by noting that $\gamma\sim\gamma_0=\langle M\rangle/(q\mypp\langle N\rangle)$ and $B_q\sim\log\mynn\langle M\rangle$.
\end{proof}

\begin{remark}\label{rm:max-slow}
Theorem \ref{th:maxM} indicates that, under the condition of slow growth of $\braket{M}$,
the smallest part $\lambda_{\mathrm{min}}$ of a $\QQ_{\bm{z}}$-typical partition $\lambda\in\check{\varLambda}^q$ ``lives'' on the scale $A_*= (\gamma\mypp b_q)^{-1}\propto \langle N\rangle/\langle M\rangle^{q+1}=\kappa^{-1}$. On the other hand, Corollary \ref{cor:5.8} shows that the scale of variation of the largest part $\lambda_{\mathrm{max}}$ is 
given by $A^*\mynn=B_q\mypp \gamma^{-1}\!\propto \!\langle N\rangle\log\mynn\langle M\rangle/\langle M\rangle$.
This is to be compared with the typical behavior in the bulk of the partition ``spectrum'', where the scale of variation is given by $A\sim\gamma^{-1}\propto \langle N\rangle/\langle M\rangle$.
\end{remark}

\begin{remark}
Continuing an asymptotic linkage between the cases of fixed or slowly growing parameter $\langle M\rangle$, observed above in Remark \ref{rm:link1}, the limiting distributions of Theorem \ref{th:minmax} formally conform to Theorem~\ref{th:maxM} in the limit as $\langle M\rangle \to\infty$. Indeed, using \eqref{eq:max1} we have
\begin{align*}
-\log G_{\rm max}(x+B_q)&=\frac{\braket{M}\myp\Gamma(1/q,x+B_q)}{\Gamma(1/q)}\\
&\sim \frac{\braket{M}}{\Gamma(1/q)}\,(x+B_q)^{1/q-1}\mypp\rme^{-x-B_q}\\
&\sim \frac{\braket{M}}{\Gamma(1/q)}\,(\log\mynn\langle M\rangle)^{1-1/q}\,\rme^{-x}\cdot\rme^{-B_q}=\rme^{-x},
\end{align*}
according to the definition of $B_q$ in \eqref{eq:ab+}. Similarly, using \eqref{eq:min1} and \eqref{eq:int1min} we have
\begin{align*}
-\log G^c_{\rm min}(x/b_q)&=\braket{M}\!\left(1-\frac{\Gamma(1/q,x/b_q)}{\Gamma(1/q)}\right)\\
&\sim \frac{\braket{M}\myn q}{\Gamma(1/q)}\left(\frac{x}{b_q}\right)^{1/q}=x^{1/q},
\end{align*}
by the definition of $b_q$ in \eqref{eq:ab+}.
\end{remark}

\begin{remark}[Possible generalization to non-integral powers]
The reader may have observed that our main results (including Theorems \ref{thm1} and \ref{th:CLT} about joint limiting laws of weight $N_\lambda$ and length $M_\lambda$; Theorems \ref{th:minmax} and \ref{th:maxM} about the joint asymptotics of extreme parts; the limit shape Theorems \ref{th:LS} and \ref{th:Gauss}; and the cardinality Theorem \ref{th:ID}) continue to make sense \emph{for any real power $q>0$}, although our proofs proceed from the assumption that $q$ is integer, which makes the partition model combinatorically well defined. The interest in non-integral powers is not new but was mostly motivated by applications in statistical physics  \cite{Agarwala,Comtet,Roccia}. 
A mathematically meaningful interpretation of partitions with non-integral power parts may be \strut{}based on the idea recently proposed by Lipnik et al.\ \cite{Lipnik}, whereby the multiplicative Boltzmann structure is introduced on the hidden ``substrate'' $\NN=\{k\}$, from which the $q$-power parts of the random partition $\lambda=(\ell_i)$ are formed as $\ell=\lfloor k^q\rfloor$, with any $q\in\RR_+$. We will review and extend our results under this approach in a separate paper. 
\end{remark}

\section{Application to random sampling} \label{sec:6}
\emph{Boltzmann sampling} is a powerful technique conceptualized, streamlined and popularized by Duchon et al.\ \cite{Duchon} in the context of single-parameter combinatorial structures (for multi-parametric extensions, see Bendkowski et al.\ \cite{BBD} and the references therein). 
Random integer partitions with controlled expected weight and length provide an ``exactly soluble'' instance of a two-parametric combinatorial structure, where the issues of Boltzmann sampling implementation and efficiency can be analyzed in some depth. 

Specifically, in this section we discuss sampling from the Boltzmann distribution on partition spaces $\check{\varLambda}^q$ (i.e., into distinct $q$-power parts), calibrated under the predefined hyper-parameters $\braket{N}$ and $\braket{M}$, which have the meaning of the expected weight and length, respectively. The two controlling parameters in question are $z_1$ and $z_2$, which are amenable to asymptotic analysis as was shown in Section~\ref{sec:3.3}. Once these parameters are fixed, due to the mutual independence of the multiplicities $(\nu_\ell)$ (see Proposition \ref{pr:mult} and Lemma~\ref{distnu}), the Boltzmann sampling is essentially reduced to an iterated independent testing of potential parts $\ell=j^q$ via  dichotomous (Bernoulli) random trials with 
success probabilities $\QQ_{\bm{z}}(\nu_{\ell}=1)=z_1^\ell z_2\mypp(1+z_1^\ell z_2)^{-1}$. 
The practical implementation of such sampling algorithms thus relies on a random number generator $\mathtt{Ber}(p)$, in each call producing an independent pseudo-random value $1$ or $0$ with probabilities $p$ and $1-p$, respectively.

It is convenient to distinguish between the \emph{free samplers} and the \emph{rejection samplers}, with the former just producing independent random  realizations of partitions under the said Boltzmann distribution,
and the latter 
comprising one or more rejection loops that 
iterate a free Boltzmann sampler 
until the desired
targets are met. We discuss these two versions in Sections \ref{sec:6.1} and \ref{sec:6.2}, respectively.

Computer codes were 
implemented using the programming language C and 
Intel\textsuperscript{\textregistered} \texttt{oneAPI\;DPC}$\texttt{++}$ compiler,
and run on a desktop CPU Intel\textsuperscript{\textregistered} Core\textsuperscript{\tiny TM} i5-10600 (processor base frequency 3.30 GHz, turbo boost frequency 4.80 GHz). Numerical calculations were carried out using  Maple\textsuperscript{\tiny TM} (Release 2022.1, licensed to the University of Leeds).

\subsection{Free sampler}\label{sec:6.1}
In this subsection, we delineate a free Boltzmann sampler (see Algorithm~\ref{sampler} below) 
under the calibration through the hyper-parameters $\braket{N}$ and $\braket{M}$. It should be noted that, despite an intuitive appeal of iterated Bernoulli-type tests, there are some implementation concern that have to be addressed. We discuss them below before presenting the algorithm.

\subsubsection{Correcting the bias}\label{sec:6.1.1}
The first issue to consider is that of choosing the control parameters $z_1$ and $z_2$ to ensure that the sampler is unbiased, that is, $\EE_{\bm{z}}(N_\lambda)=\braket{N}$ 
and $\EE_{\bm{z}}(M_\lambda)=\braket{M}$. Unfortunately, we can solve this set of equations only asymptotically (see Lemma~\ref{distnu}). In a ``crude'' version of Algorithm \ref{sampler}, we use the leading terms in the asymptotics by setting (cf.\ \eqref{eq:z1} and \eqref{eq:z2})
\begin{equation}\label{eq:z12-crude}
z_1=\rme^{-\gamma_0},\qquad z_2=\frac{\braket{M}\gamma_0^{1/q}}{\Gamma(1+1/q)},
\end{equation}
where $\gamma_0=\langle M \rangle/(q\myn\braket{N})$. Inevitably, this causes a bias in the resulting expectations. 
More precisely, the first source of this bias clearly comes from dropping the (positive) remainder terms $R_1(\bm{z})$ and $R_2(\bm{z})$ in the approximate series representations of the aforementioned expected values (see equations \eqref{eq:<M>} and \eqref{eq:<N>}). A further error occurs when replacing the resulting series with the corresponding integrals, using Lemma~\ref{lm:sum1}.

In fact, one can show that the overall bias due to \eqref{eq:z12-crude} is always negative. Indeed, recalling that $\Delta_0(\gamma)<0$ (see \eqref{eq:Delta0}\footnote{The inequality $\Delta_0(\gamma)<0$ can also be seen directly by monotonicity of the function $x\mapsto \rme^{-\gamma\myp x^q}$.}), we have
\begin{align*}
\EE_{\bm{z}}(M_\lambda)=\sum_{\ell\in\NN^q}\frac{z_1^\ell z_2}{1+z_1^\ell z_2}
<z_2\sum_{\ell\in\NN^q}z_1^\ell& =z_2\sum_{j=1}^\infty \rme^{-\gamma_0\myp j^q}\\
&<z_2\int_0^\infty\! \rme^{-\gamma_0\myp x^q}\rmd{x}=\frac{z_2\,\Gamma(1/q)}{q\,\gamma_0^{1/q}}
=\langle M\rangle,
\end{align*}
according to the parameter choice \eqref{eq:z12-crude}. Turning to  $\EE_{\bm{z}}(N_\lambda)$, recall from \eqref{eqR2} that the error term $-R_2(\bm{z})$ is negative, and furthermore, 
\begin{equation}\label{eq:R2dom}
R_2(\bm{z})\sim z_2^2\sum_{\ell\in\NN^q} \ell\myp z_1^{2\ell}\sim
\frac{z_2\, q\myn\braket{M}\myn\gamma_0^{1/q}}{\Gamma(1/q)}\cdot\frac{\Gamma(1+1/q)}{q\,\gamma_0^{1+{1/q}}}=\frac{z_2 \braket{M}}{q\,\gamma_0}.
\end{equation}
On the other hand, the error due to replacing the sum $\sum_{\ell} \ell\myp z_1^{\ell}$ in \eqref{eq:<N>} by the corresponding integral is bounded, according to \eqref{eq:2nd+}, by $z_2\,O\bigl(\gamma_0^{-1+1/q}\bigr)$. Since $\braket{M}$ is bounded away from zero, it follows that the $R_2$-term \eqref{eq:R2dom} is dominant and, therefore, the overall bias in targeting $\braket{N}$ is negative.

A practical recipe towards correcting the bias may be to move the  error terms $R_1(\bm{z})$ and $R_2(\bm{z})$ to the left-hand side of equations in \eqref{eq:<M>} and \eqref{eq:<N>}, for simplicity using their integral approximations. Effectively, this  amounts to redefining the hyper-parameters,
\begin{align*}
\braket{\tilde{N}}&:=\braket{N}+z_2^2\sum_{\ell\in\NN^q} \ell\myp z_1^{2\ell}\approx \braket{N}+z_2^2\int_0^\infty\! x^{q}\,\rme^{-2\myp\gamma_0\myp x^q}\myp\rmd{x}=\braket{N}+\frac{\braket{M}^2 \gamma_0^{1/q}}{2^{1/q}\,\Gamma(1/q)},\\
\braket{\tilde{M}}&:=\braket{M}+z_2^2\sum_{\ell\in\NN^q} z_1^{2\ell}\approx \braket{M}+z_2^2\int_0^\infty\! \rme^{-2\myp\gamma_0\myp x^q}\myp\rmd{x}=\braket{N}+\frac{q\myn\braket{M}^2 \gamma_0^{1/q}}{2^{1/q}\,\Gamma(1/q)}.
\end{align*}
Accordingly, we redefine $\tilde{\gamma_0}=\langle \tilde{M} \rangle/(q\myn\braket{\tilde{N}})$ and (cf.\ \eqref{eq:z12-crude})
\begin{equation}\label{eq:z12-tilde}
\tilde{z}_1=\rme^{-\tilde{\gamma_0}},\qquad \tilde{z}_2=\frac{q\myn\braket{\tilde{M}}\tilde{\gamma_0}^{1/q}}{\Gamma(1/q)}.
\end{equation}
A numerical illustration of the proposed modification is presented in Table \ref{table1}, showing a significant reduction of bias. The expected values were computed from the exact series expansions \eqref{eq:<M>} and \eqref{eq:<N>} using Maple.
\begin{table}[H]
\caption{Expected values of $N_\lambda$ and $M_\lambda$ for $q=1$ and $q=2$ under two choices of the calibrating parameters: using the leading asymptotic terms \eqref{eq:z12-crude} and after a heuristic correction \eqref{eq:z12-tilde}.}\label{table1}
\small
\vspace{1pc}\centering\begin{tabular}{|c|c||c|c|}
 \hline
\multicolumn{2}{|c||}{$q=1$}&\multicolumn{2}{|c|}{$q=2$}\\
\hline
$\vphantom{\int^{)}}\braket{N}=10^6$&$\braket{\tilde{N}}\doteq1{,}\myp002{,}\myp499.50$&$\braket{N}=10^7$&$\braket{\tilde{N}}\doteq10{,}\myp315{,}\myp391.57$\\
\hline
$\vphantom{\int^{)}}\braket{M}=100$&$\braket{\tilde{M}}\doteq100.4999$&$\braket{M}=50$&$\braket{\tilde{M}}\doteq53.1539$\\
 \hline  
 $z_1\doteq0.9999000$& $\tilde{z}_1=0.9998998$& $z_1\doteq0.9999975$& $\tilde{z}_1\doteq0.9999974$\\
 \hline  
 $z_2\doteq0.010000$& $\tilde{z}_2\doteq0.0100750$& $z_2\doteq0.0892062$& $\tilde{z}_2\doteq0.0962720$\\\hline
 $\EE_{\bm{z}}(N_\lambda)\doteq997{,}\myp510.70$&$\EE_{\tilde{z}}(N_\lambda)\doteq999{,}\myp985.73$&$\EE_{\bm{z}}(N_\lambda)\doteq9{,}\myp699{,}\myp070.63$&$\EE_{\tilde{z}}(N_\lambda)\doteq9{,}\myp981{,}\myp802.71$\\
 \hline
$\EE_{\bm{z}}(M_\lambda)\doteq99.498326$&$\EE_{\tilde{z}}(M_\lambda)\doteq99.992019$&$\EE_{\bm{z}}(M_\lambda)\doteq47.018388$&$\EE_{\tilde{z}}(M_\lambda)\doteq49.754495$\\
 \hline \end{tabular}
 \end{table}

If the remaining (small) bias is still an issue, a further recalibration can be carried out by a suitable refinement of the solution $\bm{z}=(z_1,z_2)$ to the equations \eqref{eq:<M>} and \eqref{eq:<N>}, for instance, by using a two-dimensional Newton--Raphson method. For a general approach to the multidimensional tuning of parameters based on convex optimization, see Bendkowski et al.\ \cite{BBD}.

\subsubsection{Truncation of the parts pipeline}\label{sec:6.1.2}

We deal with a finitary computation, so should rule out the risk of indefinite processing. That is to say, the algorithm must have a well-defined stopping rule that would guarantee a finite-time termination. In a free sampler, the sequence of productive outcomes in successive Bernoulli trials (i.e., with sample multiplicities  $\nu_\ell=1$ corresponding to  non-zero parts) is $\QQ_{\bm{z}}$-a.s.\ finite (see Lemma~\ref{lm:lemma2}). More precisely, the last successful trial selects the largest part $\lambda_{\rm max}$, after which the testing 
settles down to pure idling. 
The $\QQ_{\bm{z}}$-distribution of  $\lambda_{\rm max}$ is given by (cf.\ Sections \ref{sec:4.3} and \ref{sec:5.4})
\begin{align}
\notag
\QQ_{\bm{z}}(\lambda_{\rm max}\le L)&=\QQ_{\bm{z}}(\nu_{\ell}\equiv 0\ \text{for all}\  \ell > L)\\
&=\prod_{\ell >L} \frac{1}{1+z_1^{\ell}z_2}=\frac{1}{F(\bm{z})}\prod_{\ell \le L} (1+z_1^{\ell}z_2),
\label{eq:oracle-max}
\end{align}
where $F(\bm{z}) = \prod_{\ell\in\NN^q} (1+z_1^{\ell}z_2)$ is the generating function of the partition space $\check{\varLambda}^p$ (see \eqref{eq:F_q}). 
It is also easy to see that conditioning on $\lambda_{\rm max}=j_0^q$ does not change the distribution of the preceding multiplicities $\{\nu_{j^q},1\le j\le j_0-1\}$, that is, they remain mutually independent and with Bernoulli distributions \eqref{eq:Bern}. Thus, if the numerical value of 
$F(\bm{z})$ can be calculated in advance, which is a common convention in computing known as an \emph{oracle} (see, e.g., \cite{Duchon,Flajolet-Fusy-Pivoteau,BBD}), then we can sample the random value $\lambda_{\rm max}$ using formula \eqref{eq:oracle-max} and then sample independently the preceding candidate parts via the respective Bernoulli trials. Unfortunately, this approach embeds a computational error through the numerical calculation of $F(\bm{z})$, so it is not quite ``exact''; besides, convergence of the infinite product may not be fast, given that the parameter $z_1$ is close to $1$ (see \eqref{eq:z1}). Specifically, using Lemma \ref{lm:sum2+} one can check that the truncation error arising from a partial product up to $\ell_*$ is of order $\braket{M} (\gamma\myp\ell_*)^{1/q-1}\myp\rme^{-\gamma\myp\ell_*}$, which dictates that $\ell_*$ be chosen much bigger than $\gamma^{-1}\sim q\mypp\langle N\rangle/\langle M\rangle$.

An alternative idea is to truncate the pipeline of potential parts $\ell\in\NN^q$ subject to testing at an appropriate threshold $L$ (see Section~\ref{sec:2.5}), so that the Bernoulli testing only runs over $\ell\le L$.
A simple pragmatic solution is to choose the threshold $L$ so that the probability of exceeding it in an indefinite free sampler is small enough, that is, 
\begin{equation}\label{eq:<L}
\QQ_{\bm{z}}(\lambda_{\mathrm{max}}>L)\le\delta,
\end{equation}
where the confidence tolerance $\delta>0$ can be chosen in advance to be as small as desired. Then the corresponding threshold $L=L(\delta)$ can be determined from a suitable limit theorem for the largest part, namely, Theorem \ref{th:minmax} if $\braket{M}>0$ is fixed, or Theorem \ref{th:maxM} for slow growth of $\braket{M}$. 
In the former case, threshold $L$ is determined by the asymptotic equation (see~\eqref{eq:max1})
\begin{equation}\label{eq:L1}
\Gamma(1/q,\gamma_0 L)=\frac{\Gamma(1/q)}{\braket{M}}\cdot \log\frac{1}
{1-\delta},
\end{equation}
where, as before, $\gamma_0=\braket{M}\mynn/(q\myn\braket{N})$.
In the latter case, we obtain from \eqref{eq:max-lim}
\begin{equation}\label{eq:L2}
L = \frac{1}{\gamma_0}\left(B_q -\log \log \frac{1}{1-\delta}\right)\!,
\end{equation}
where (see \eqref{eq:ab+})
$$
B_q=\log\myn \langle M\rangle-\bigl(1-\tfrac{1}{q}\bigr)\log \log\myn\langle M\rangle-\log\Gamma(1/q).
$$
Note that for $q=1$ the bounds \eqref{eq:L1} and \eqref{eq:L2} coincide, reducing to
\begin{equation}\label{eq:L0}
L = \frac{1}{\gamma_0}\left(\log\myn \langle M\rangle -\log \log \frac{1}{1-\delta}\right)\!.
\end{equation}

An illustration of evaluation of the threshold $L$ is presented in Table \ref{table2} for $q=1$ and $q=2$. The equation \eqref{eq:L1} was solved numerically using Maple. One can observe from the table that while the value $\delta=10^{-k}$ ($k=1,2,\dots$)  is decreasing geometrically, the growth of threshold $L$ is only about linear. Intuitively, this is explained by the fact that the confidence probability $1-\delta$ enters  expressions \eqref{eq:L1} and \eqref{eq:L2} under the logarithm, that is, as $\log\left(1-\delta\right)$. More precisely, it readily follows from \eqref{eq:L2} that
$$
L-\frac{B_q}{\gamma_0}\sim\frac{1}{\gamma_0}\log\frac{1}{\delta}\qquad (\delta\to0^+).
$$
Likewise, equation \eqref{eq:L1} asymptotically solves to yield
$$
L-\frac{\log\braket{M}}{\gamma_0\,\Gamma(1/q)}\sim\frac{1}{\gamma_0}\log\frac{1}{\delta}\qquad (\delta\to0^+).
$$

\begin{table}[H]
\caption{Threshold $L$ for the largest part $\lambda_{\rm{max}}$ with confidence probability $1-\delta$, calculated from expressions \eqref{eq:L0} ($q=1$) and \eqref{eq:L1} or \eqref{eq:L2} ($q=2$) and rounded down to the nearest $q$-th power.}
 \label{table2}
\small
\vspace{1pc}\centering\begin{tabular}{|c|c||c|c|c|}
 \hline
\multicolumn{2}{|c||}{$\vphantom{\int^{t}}q=1$, $\braket{N}=10^6$, $\braket{M}=100$}&\multicolumn{3}{|c|}{$q=2$, $\braket{N}=10^7$, $\braket{M}=50$}\\
\hline
$\vphantom{\int^{t}}\delta$& $L$  \eqref{eq:L0}&$\delta$& $L$  \eqref{eq:L1}& $L$  \eqref{eq:L2}\\
 \hline 
$0.1$& $68{,}\myp555$&$0.1$& $1{,}\myp890{,}\myp625=1375^2$&$1{,}\myp962{,}\myp801=1401^2\vphantom{\int^{t}}$\\
 \hline 
$0.01$&  $92{,}\myp053$&$0.01$& $2{,}\myp762{,}\myp244=1662^2$ & $2{,}\myp900{,}\myp209=1703^2\vphantom{\int^{t}}$\\
 \hline 
$0.001$& $115{,}\myp124$&$0.001$& $3{,}\myp636{,}\myp649=1907^2$&$3{,}\myp825{,}\myp936=1956^2\vphantom{\int^{t}}$ \\
\hline 
$0.0001$& $138{,}\myp154$ &$0.0001$& $4{,}\myp515{,}\myp625=2125^2$ &$4{,}\myp743{,}\myp684=2178^2\vphantom{\int^{t}}$ \\
\hline  \end{tabular}
\end{table}
According to Lemma~\ref{lm:B-cutoff0} (with $\tilde{\varLambda}^\dag=\check{\varLambda}_L$), the output of a truncated sampling algorithm follows the Boltzmann distribution with a smaller source set $\mathbb{A}_L=\{\ell\in\NN^q\colon \ell\le L\}$, which nonetheless approximates well the target Boltzmann distribution $\QQ_{\bm{z}}$ (see Lemma~\ref{lm:TV-appr}). One should be wary though that truncation contributes to the negative bias (see \eqref{eq:ExpN<} and \eqref{eq:ExpM<}), which may require a refined calibration through the parameters $z_1$ and $z_2$. 

Note that the confidence guarantee $1-\delta$ as discussed above is valid only in the case of a single output instance. If the purpose of the free algorithm is to produce an independent sample of, say, $k$ random Boltzmann partitions, then the overall confidence probability is approximately  given by $(1-\delta)^k$, which may be  exponentially small if $k$ is large while  $\delta$ stays fixed. A simple upper bound for the error probability is based on the Bernoulli inequality, yielding $1-(1-\delta)^k\le k\delta$. This motivates the well-known Bonferroni correction, which amounts to choosing the  individual error probability $\delta_0=\delta/k$ in order to ensure the overall error probability not exceeding $\delta$. As an example, if $k=1{}\myp 000$ and we would like to guarantee the overall error probability bound $\delta=0.1$, then the individual error probability should be taken as $\delta_0=0.0001$. The approximation is quite accurate here, as the exact solution is $\delta_0\doteq 0.000105355$. Clearly, switching from $\delta$ to $\delta_0$ leads to a higher threshold $L$. For instance, Table \ref{table2} shows that in the case $q=1$ the suitable threshold $L$ needs to double. In general, the increase of $L$ due to multiple errors is not really dramatic because of the logarithmic dependence on $\delta$ mentioned above.

Another unwelcome outcome of the Bernoulli testing is that it may not return any parts at all, which has a positive $\QQ_{\bm{z}}$-probability even in the infinite sequence of tests (see Remark~\ref{rm:==0}). This is not critical, as the sampling cycle can be repeated if necessary. However, it may be wasteful and can be easily rectified by adopting a similar approach based on confidence. Specifically, one can set a lower cutoff $L_0$ such that the run of the sampler is terminated, and the cycle is repeated, if the Bernoulli tests fail to select at least one (non-zero) part $\ell\le L_0$. To this end, we choose $L_0$ in such a way that $\QQ_{\bm{z}}(\lambda_{\rm min}>L_0)\le \delta$, where $\delta>0$ is small enough (cf.\ \eqref{eq:<L}). Again referring to the limit theorems regarding the smallest part, we obtain from Theorem \ref{th:minmax} (cf.\ \eqref{eq:L1})
\begin{equation}\label{eq:L1'}
\Gamma(1/q)-\Gamma(1/q,\gamma_0 L_0)=\frac{\Gamma(1/q)}{\braket{M}}\mypp \log\frac{1}
{\delta},
\end{equation}
and from Theorem \ref{th:maxM} (cf.\ \eqref{eq:L2})
\begin{equation}\label{eq:L2'}
L_0=\frac{1}{\gamma_0}\left(\frac{\Gamma(1/q)}{q\myn\braket{M}}\log \frac{1}{\delta}\right)^q\!.
\end{equation}

A numerical illustration of the confident lower threshold $L_0$ is presented in Table \ref{table3} for $q=1$ and $q=2$. The equation \eqref{eq:L1'} was solved numerically using Maple. The match between the results produced via equations \eqref{eq:L1'} and \eqref{eq:L2'} is quite close, especially for $q=2$. One should also observe a significant difference between the thresholds $L_0$ and $L$, which underpins a considerable computational saving due to the lower cutoff, activated whenever the sampler fails to produce at least one positive part up to $L_0$.

\begin{table}[H]
\caption{Asymptotic threshold $L_0$ for the smallest part $\lambda_{\rm{min}}$ with confidence probability $1-\delta$, calculated from expressions \eqref{eq:L1'} or \eqref{eq:L2'}, and rounded down to the nearest $q$-th power.}
 \label{table3}
\small
\vspace{1pc}\centering\begin{tabular}{|c|c|c||c|c|c|}
 \hline
\multicolumn{3}{|c||}{$\vphantom{\int^{t}}q=1$, $\braket{N}=10^6$, $\braket{M}=100$}&\multicolumn{3}{|c|}{$q=2$, $\braket{N}=10^7$, $\braket{M}=50$}\\
\hline
$\vphantom{\int^{t}}\delta$& $L_0$ \eqref{eq:L1'}& $L_0$ \eqref{eq:L2'} &$\delta$& $L_0$  \eqref{eq:L1'}& $L_0$  \eqref{eq:L2'}\\
 \hline 
$0.1$&$232$ &$230$&$0.1$& $625=25^2$&$625=25^2\vphantom{\int^{t}}$\\
 \hline 
$0.01$&  $471$&$460$&$0.01$& $2{,}\myp601=51^2$ & $2{,}\myp601=51^2\vphantom{\int^{t}}$\\
 \hline 
$0.001$& $715$&$690$&$0.001$&$5{,}\myp929=77^2$ &$5{,}\myp929=77^2\vphantom{\int^{t}}$ \\
\hline 
$0.0001$& $966$ &$921$&$0.0001$& $10{,}\myp816=104^2$ &$10{,}\myp609=103^2\vphantom{\int^{t}}$ \\
\hline  \end{tabular}
\end{table}

Finally, if both cutoffs $L$ and $L_0$ are exercised as described above then, due to the asymptotic independence of $\lambda_{\rm max}$ and $\lambda_{\rm min}$, the overall confidence probability is (asymptotically) given by $(1-\delta)^2=1-2\myp\delta+\delta^2$, hence the resulting error probability is bounded by $2\myp\delta$.

\subsubsection{Free sampling algorithm}\label{sec:6.1.3}

A free Boltzmann sampler is presented below in pseudocode as Algorithm~\ref{sampler}. For simplicity, the algorithm incorporates only the upper threshold $L$ selected in advance for a given confidence probability $1-\delta$. As discussed in Section \ref{sec:6.1.2}, this is essential to ensure  termination of the code, but for the sake of optimization a lower cutoff $L_0$ can also be included without difficulty. As explained above, the confidence probability should be chosen carefully to match a possibly multiple output.

\begin{figure}[ht!]
\begin{algorithm}[H]
\label{sampler}
\caption{{\tt FreeSampler}\mypp($q,
\langle N\rangle,
\langle M\rangle, L$)} 
\KwInput{integer $q$,
real $\langle N\rangle,\langle M\rangle, L$}
\KwOutput{partition $\lambda\in\check{\varLambda}^{q}_{L}$, 
weight $N_\lambda$, length $M_\lambda$}
integer array $\lambda_{[\ ]}$\; real 
$z_1, z_2, \gamma_0, q\myp$\;
\strut{}$\gamma_0\leftarrow \langle M \rangle/(q\myn\braket{N})$\;
$z_1 \leftarrow \rme^{-\gamma_0}$, 
$z_2 \leftarrow q\myn\braket{M}\myn\smash{\gamma_0^{1/q}}\mynn/\myp\Gamma(1/q)$\;
integer $j^*$, $j$, $N$, $M$\;
$j^*\leftarrow \lfloor L^{1/q}\rfloor$\;
$N \leftarrow 0$, $M \leftarrow 0$\;
       \strut{} \For{$j$ from $j^*$ to $1$ by $-1$}
{$p\leftarrow z_1^{j^q}\! z_2\mypp(1+z_1^{j^q}\! z_2){\displaystyle^{-1}}$\; 
\If{$\mathtt{Ber}
(p)=1$}
{
		$N \leftarrow N+ j^q$\;
		$M \leftarrow M+1$\;
	$\lambda_{M}\leftarrow j^q$\;
}
}
$N_\lambda\leftarrow N$, $M_\lambda\leftarrow M$\;
\Return{$(\lambda, N_{\lambda}, M_{\lambda})$}
\end{algorithm}
\end{figure}

The code structure is fairly straightforward and consists in a single cycle of sequential Bernoulli tests over potential parts $\ell\in\NN^q$. It is convenient to do this via downward scoping in view of our convention to enumerate the partition parts in decreasing order. Because the resulting length of the output partition $\lambda=(\lambda_i)$ is unknown in advance, the space for the corresponding  integer array is defined in the code as $\lambda_{[\ ]}$, that is, through a dynamically allocated memory. Finally, the calibration parameters $z_1$ and $z_2$ are specified using the leading-term formulas \eqref{eq:z12-crude}; if desired, these can be replaced by the bias-correcting values \eqref{eq:z12-tilde} or by any other, more refined choices.

By design, the output of Algorithm \ref{sampler} is a random partition $\lambda\in\check{\varLambda}^q_L=\check{\varLambda}^q\cap\{(\lambda_i)\colon\lambda_{\rm{max}}\le L\}$. It has a Boltzmann distribution on the space $\check{\varLambda}^q_L$, with expected values of weight $N_\lambda$ and length $M_\lambda$ close to the predefined hyper-parameters $\braket{N}$ and $\braket{M}$, respectively. As discussed in Section \ref{sec:2.5}, this distribution approximates (in total variation) the Boltzmann distribution on the infinite partition space $\check{\varLambda}^q$, which may suffice for the sampling purposes at hand.

\subsubsection{Validation of Algorithm \ref{sampler}}
The output performance of the code in Algorithm~\ref{sampler} was visually monitored via the marginal histograms for the sample weight $N_\lambda$ and length $M_\lambda$ (Figure~\ref{fig4}), as well as by the bivariate histograms and frequency level plots of the sample pairs  $(N_\lambda,M_\lambda)$ (Figure~\ref{fig5}). The numerical illustration was carried out in the case of square parts, $q=2$ (selected for computational convenience in order to reduce the completion time), and in two different regimes with regard to the hyper-parameter $\braket{M}$, that is, ``fixed'' and ``slow growth'', illustrated by  $\braket{M}=5$ ($\braket{N}=12{,}\myp500$) and $\braket{M}=50$ ($\braket{N}=10^7$), yielding for the parameter $\kappa=\langle M\rangle^3\myn/\langle N\rangle$ values $\kappa=0.01$ and $\kappa=0.0125$, respectively (cf.\ Assumption~\ref{as2}). The algorithm was run at a very low  confidence tolerance (error probability) $\delta = 10^{-8}$ and with the corresponding truncation value $L$ calculated using formulas \eqref{eq:L1} or \eqref{eq:L2} according to the regime at hand, yielding $L=299^2=89,401$ and $L=2{,}\myp903^2=8{,}\myp427{,}\myp409$, respectively (cf.\ Table~\ref{table2}).

\begin{figure}[ht!]
\centering
\subfigure[$\braket{M}=5$, $\braket{N}=12{,}\myp500$ \mypp($\kappa=0.01$). A dot at the origin on the left panel shows a discrete atom $\pi_0=\rme^{-\braket{M}}\doteq 0.00674$
(see \eqref{eq:mix}).]{\includegraphics[width=0.4\textwidth]{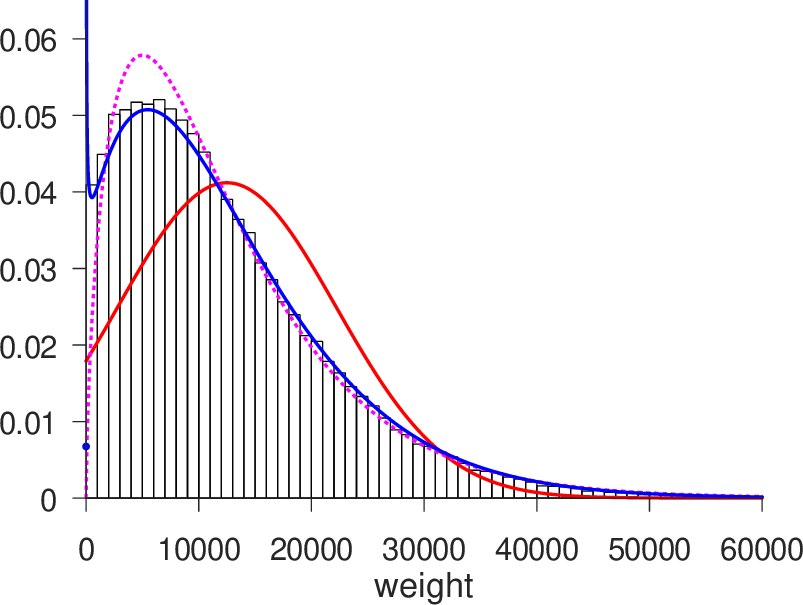}\hspace{2pc} \includegraphics[width=0.4\textwidth]{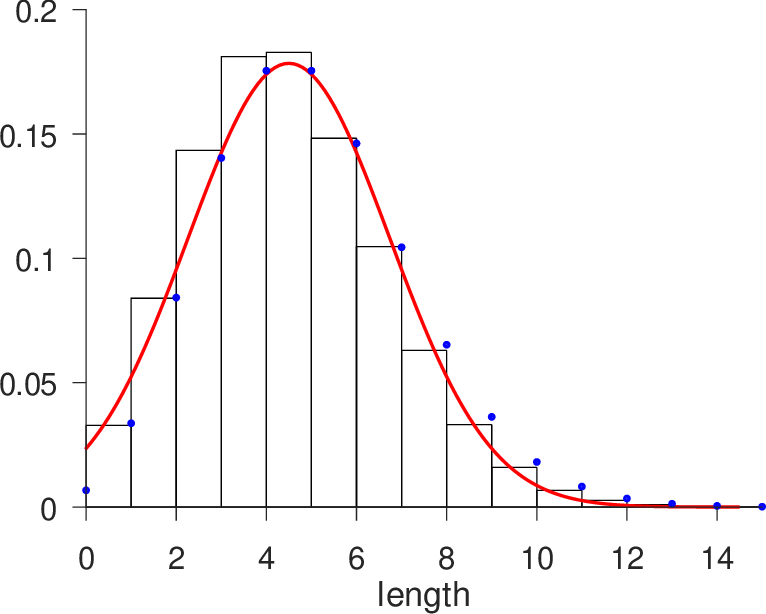}}\\
\subfigure[$\braket{M}=50$, $\braket{N}=10^7$ \myp($\kappa=0.0125$). The atom $\pi_0=\rme^{-\braket{M}}\doteq2\cdot 10^{-22}$ is too small to be visible.]{\includegraphics[width=0.4\textwidth]{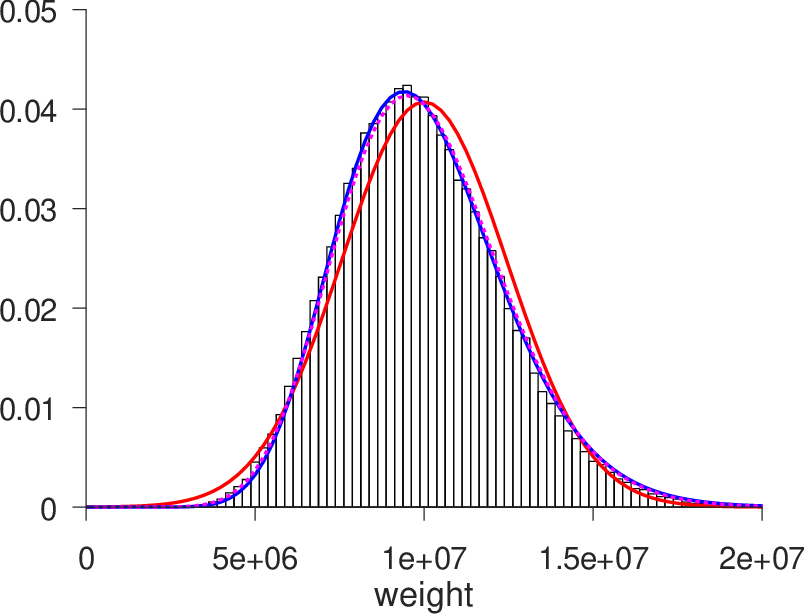}\hspace{2pc} \includegraphics[width=0.4\textwidth]{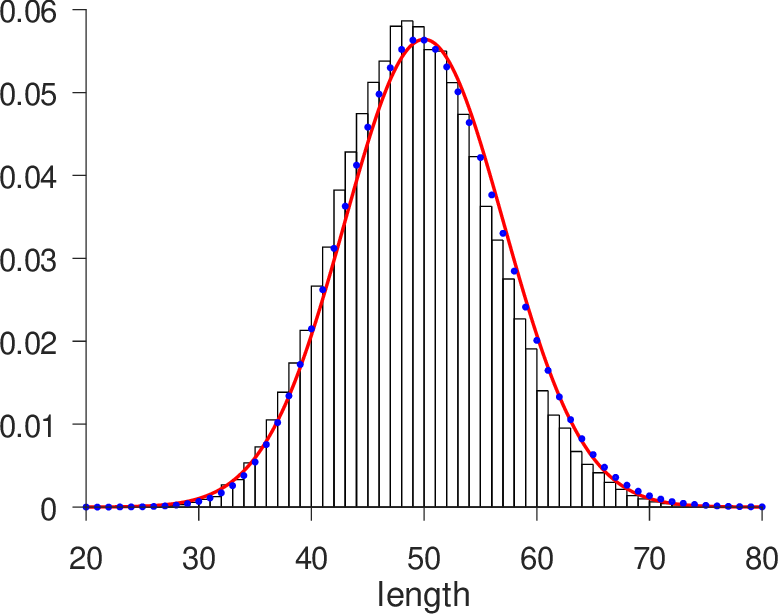}}
\caption{Marginal histograms for the weight $N_\lambda$ (left) and length $M_\lambda$ (right) for random samples (of size $10^5$ each) from the partition space $\check{\varLambda}^q$ ($q=2$),
simulated using a free Boltzmann sampler as set out in Algorithm~\ref{sampler}. Color coding (online version): \emph{blue} designates the limiting distributions under the ``fixed'' regime, that is, compound Poisson-Gamma (left) and Poisson (right); \emph{red} indicates a normal approximation; \emph{magenta} depicts a mean-gamma approximation (see Section ~\ref{sec:4.1}). 
In the black-and-white version, the normal curves on the left are identifiable by a noticeable positive shift.}
\label{fig4}
\end{figure}

The empirical results (with $10^5$ output partitions in both cases) were compared with the theoretical predictions from Theorems \ref{thm1} and \ref{th:CLT}. The marginal histogram plots for $N_\lambda$ and $M_\lambda$ shown in Figure \ref{fig4} depict bell-shaped unimodal empirical distributions with the sample modes noticeably shifted to the left of the calibration hyper-parameters $\braket{N}$ and $\braket{M}$, respectively. The discrepancy between the modes and the means is observed especially well in the weight plots (the more so for smaller $\braket{M}$), which is essentially due to the fact that the underlying gamma and Poisson distributions are right-skewed; for example, the mode of $\mathrm{Gamma}\mypp(\alpha)$ is given by $\max\{\alpha-1,0\}$ whereas the mean is $\alpha$. Another (minor) reason for a negative bias is due to a certain miscalibration, as was pointed out in Section~\ref{sec:6.1.1}. In line with theoretical predictions, this mismatch is vanishing with the growth of the hyper-parameters $\braket{N}$ and $\braket{M}$, together with the improving accuracy of the normal approximations, both for $N_\lambda$ and $M_\lambda$. It is also interesting to note that the  mean-gamma approximation $\mathrm{Gamma}\mypp(\braket{M})$ (see Section \ref{sec:4.1}) nearly perfectly matches the exact compound  Poisson-Gamma distribution for $\braket{M}=50$ (see Figure \ref{fig4}(b), left); for smaller values of $\braket{M}$, this approximates is rather crude, however it still captures well the mode of the Poisson-Gamma distribution and also its right shoulder (see Figure \ref{fig4}(a), left).

A remarkable exception to the unimodality of the plots in Figure \ref{fig4} is the compound Poisson-Gamma plot for the weight $N_\lambda$, with a relatively small value of $\braket{M}=5$ (see Figure \ref{fig4}(a), left), where one can clearly see a singularity at zero (cf.\ Section \ref{sec:4.1}). This theoretical prediction is supported by the empirical results, with an obvious excess of smaller weights. With $\braket{M}=5$ and $\braket{N}=12{,}\myp500$, the local minimum of the theoretical density $g(x)$ defined in   \eqref{eq:g-density} occurs at $x_0\doteq
0.10340$
with value\footnote{Note that the asymptotic formula \eqref{eq:pdf-as} gives a pretty accurate approximation $g(x_0)\approx 0.14334$.}
$g(x_0)\doteq 0.19632$, 
which corresponds to weight $n_0=\lceil x_0/\gamma_0\rceil 
=517$. If the density $g(x)$ continued to decay to the left of $x_0$, this would predict the (asymptotic) probability of getting weights smaller than $n_0$ (together with an empty partition) loosely bounded by $x_0\mypp g(x_0)+\pi_0\doteq 0.02704$.
But the actual compound Poisson-Gamma probability is higher, $G(x_0)\doteq0.03120$ (see \eqref{eq:mix}).
The excess of ``small'' partitions is reminiscent of a partition interpretation of the Bose--Einstein condensation (see \cite{Vershik2}). As already mentioned in Section \ref{sec:4.1}, this is a truly finite-length phenomenon, which vanishes as $\braket{M}\to\infty$ (cf.\ Figure \ref{fig4}(b), where $\braket{M}=50$).

\begin{figure}[ht!]
\centering
\subfigure[Bivariate frequency plot.]{\raisebox{1.2pc}{\includegraphics[width=0.45\textwidth]{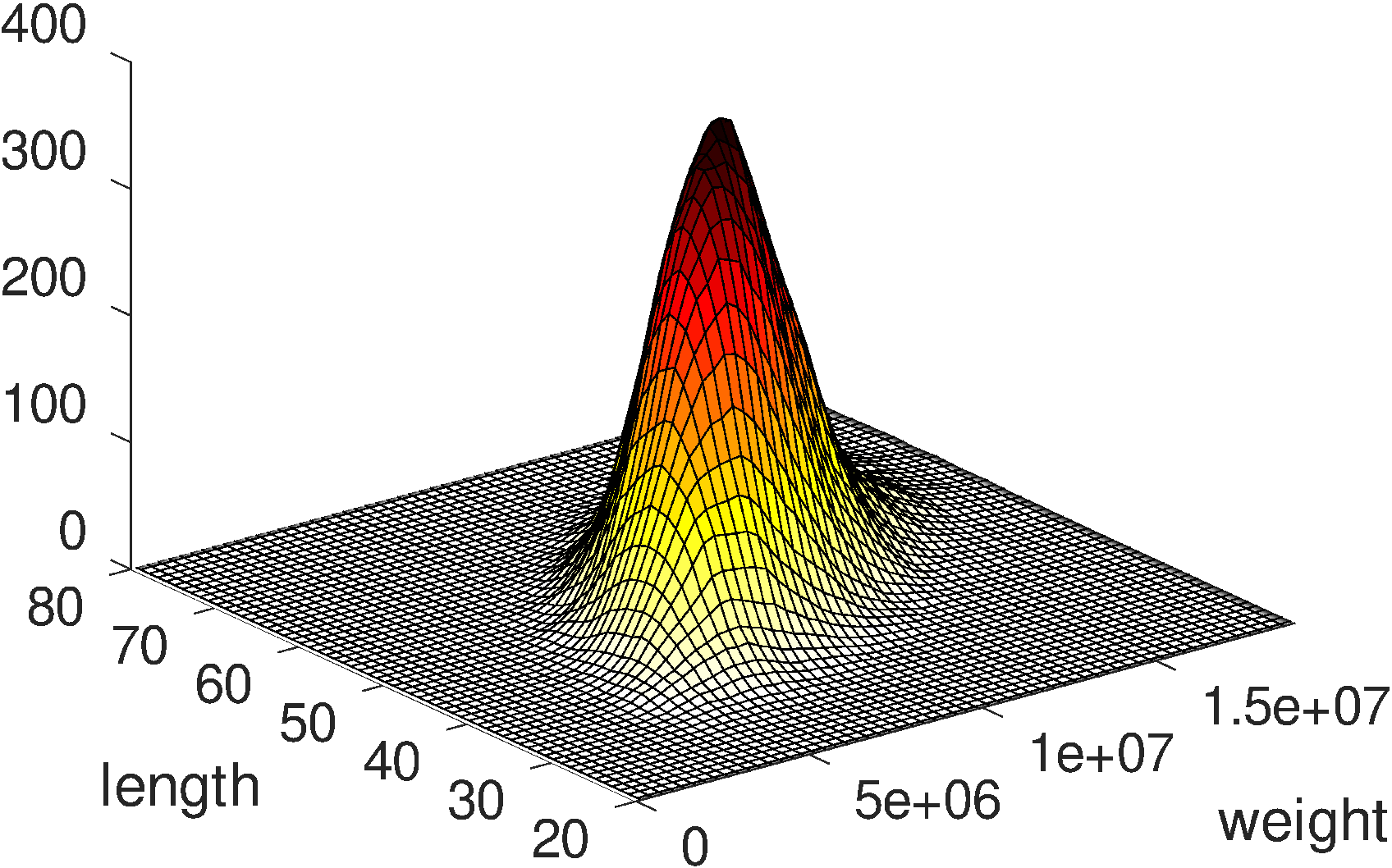}}}\hfill\subfigure[Frequency level  sets.]{\includegraphics[width=0.45\textwidth]{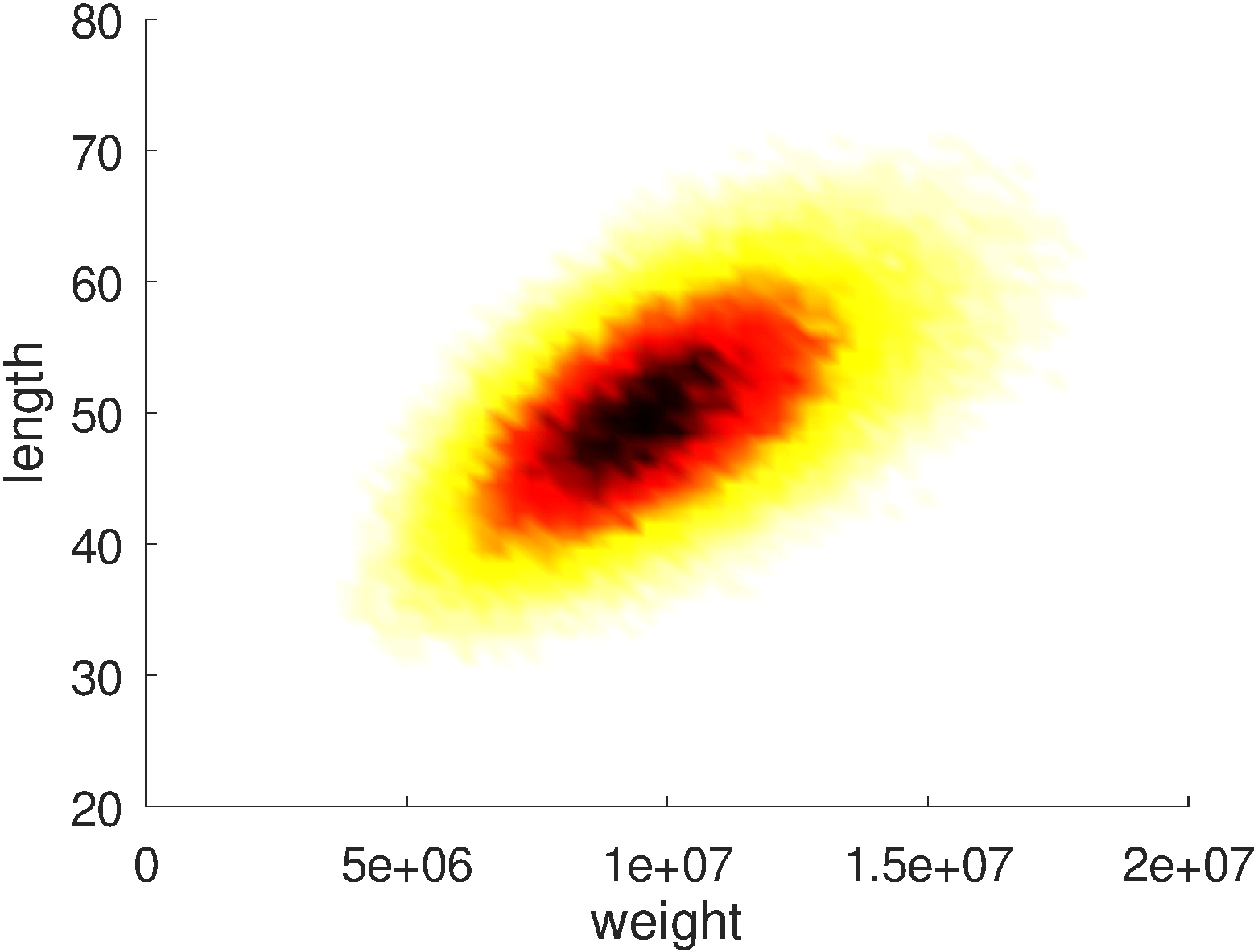}}\\
\caption{Joint sampling distribution of weight $N_\lambda$ and length $M_\lambda$ for $q=2$, $\braket{M}=50$ and $\braket{N}=10^7$ (cf.\ marginal plots in Figure \ref{fig4}(b)). A random Boltzmann sample of partitions $\lambda\in\check{\varLambda}^q$ (of size $10^5$) was simulated using Algorithm~\ref{sampler}.\label{fig5}} 
\end{figure}

The bivariate plots in Figure \ref{fig5} (for $\braket{M}=50$) appear to be approximately consistent with the asymptotically predicted (standardized) confidence ellipses of the form
\begin{equation}\label{eq:ellipse}
\mathcal{L}_\alpha=\{\bm{x}\in\RR^2\colon \bm{x}\myp \bm{K}_q^{-1}\bm{x}^\top\le \chi^2_2(1-\alpha)\},
\end{equation}
where $\bm{K}_q^{-1}$ is the inverse covariance matrix \eqref{eq:K}, and $\chi^2_2(1-\alpha)$ is the quantile of the chi-squared distribution with two degrees of freedom,  corresponding to confidence probability $1-\alpha$. The latter distribution simplifies to an exponential distribution with mean $2$, hence $\chi^2_2(1-\alpha)=2\log\left(1/\alpha\right)$. According to Theorem \ref{th:CLT}, a sample point $(N_\lambda^*,M_\lambda^*)$ belongs to the ellipse \eqref{eq:ellipse} approximately with probability $1-\alpha$, where the standardized values $N_\lambda^*,M_\lambda^*$ are defined in \eqref{eq:MN-norm}. The inverse of $\bm{K}_q$ is easily computed,
$$
\bm{K}_q^{-1}=\frac{q+1}{q}\begin{pmatrix}
1 &\displaystyle \frac{-1}{\sqrt{q+1}}\\
\displaystyle \frac{-1}{\sqrt{q+1}} & 1
\end{pmatrix}\!,
$$
and the confidence ellipse \eqref{eq:ellipse} specializes as follows,
$$
N_\lambda^{*2}-\frac{2\myp N_\lambda^*M_\lambda^*}{\sqrt{q+1}}  +M_\lambda^{*2}\le \frac{2\myp q}{q+1}\log\frac{1}{\alpha}\myp.
$$

A closer inspection of the level plots in Figure \ref{fig5}(b) reveals some elongation of the frequency level sets towards bigger values of the weight $N_\lambda$, thus indicating a bit of discrepancy with the predicted elliptical shape. This observation is confirmed by comparison of the marginal histograms of $N_\lambda$ and $M_\lambda$ in Figure \ref{fig4}, where the latter is reasonably symmetric while the former is noticeably skewed to the right. A heuristic explanation of such an effect may be based on noticing from formulas \eqref{eq:NM-constr} that, while the length $M_\lambda$ is built by summation of multiplicities $\nu_\ell$, the weight $N_\lambda$ involves size-biased terms $\ell\myp\nu_\ell$, which pinpoint skewing the distribution to the right.

A well-localized unimodal nature of the distributions behind the outputs $N_\lambda$ and $M_\lambda$ is sometimes referred to as a \emph{bumpy type} \cite{Duchon}, characterized by an asymptotically large signal-to-noise ratio (SNR) in response to a large signal,\footnote{The ``bumpy'' property is especially helpful for sampling with rejection designed to achieve certain targets for weight and length, due to a guaranteed asymptotically fast delivery of the output (see Section~\ref{sec:6.2}).}
\begin{equation}\label{eq:SNR}
\mathrm{SNR}(X):=\frac{[\EE(X)]^2}{\Var(X)}\to\infty,\qquad \EE(X)\to\infty. 
\end{equation}
Here, the notation $X$ designates a random output in question (such as the size), which has a large expected value.  Following definition \eqref{eq:SNR} and applying Theorem \ref{th:VarCov},
we readily get
$$
\mathrm{SNR}(N_\lambda)\sim\frac{\braket{N}^2}{(q+1)\braket{N}^2\!/\!\braket{M}}=\frac{\braket{M}}{q+1},\qquad \mathrm{SNR}(M_\lambda)\sim\frac{\braket{M}^2}{\braket{M}}=\braket{M} ,
$$
so that under Assumption \ref{as2} (slow growth of $\braket{M}$) each of the marginal SNRs tends to infinity.

In the multivariate case, the SNR is usually defined in the literature as a scalar value,
\begin{equation}\label{eq:SNR-m}
\mathrm{SNR}(\bm{X}):=\bm{\mu}\mypp \bm{K}^{-1}\myn\bm{\mu}^\top,\qquad \bm{\mu}:=\EE(\bm{X}),\ \ 
\bm{K}:=\Cov(\bm{X},\bm{X})
\end{equation}
(see, e.g., \cite[Eq.\,(1), p.\,511]{Polcari}). Again using Theorem \ref{th:VarCov}, we find the asymptotic inverse of the covariance matrix,
$$
\bm{K}^{-1}(\bm{z})\sim \frac{1}{q\myn\braket{N}^2}\begin{pmatrix}
\braket{M} & -\braket{N}\\
-\braket{N} & \displaystyle\frac{(q+1)\mynn\braket{N}^2}{\braket{M}}
\end{pmatrix}\!,
$$
and hence
$$
\mathrm{SNR}(N_\lambda,M_\lambda)\sim\frac{(\langle N\rangle,\langle M\rangle )}{q\myn\braket{N}^2}\begin{pmatrix}
\braket{M} & -\braket{N}\\
-\braket{N} & \displaystyle\frac{(q+1)\mynn\braket{N}^2}{\braket{M}}
\end{pmatrix}\begin{pmatrix}\braket{N}\\[.4pc]\braket{M}\end{pmatrix}=\braket{M}\to\infty.
$$
However, a scalar definition \eqref{eq:SNR-m} is not entirely satisfactory\,---\,for instance, it cannot detect whether the individual  components of $\bm{X}$ are of bumpy type. As an alternative, we propose the following matrix definition, 
\begin{equation}\label{eq:SNR-multi}
\mathbf{SNR}(\bm{X}): = (\bm{\mu}\myp \bm{K}^{-1/2})^{\mynn\top}(\bm{\mu}\myp \bm{K}^{-1/2}) = \bm{K}^{-1/2}\myp(\bm{\mu}^{\myn\top}\mynn  \bm{\mu})\mypp \bm{K}^{-1/2},
\end{equation}
where $\bm{K}^{-1/2}$ is the (unique) positive definite square root of the inverse covariance matrix $\bm{K}^{-1}$, that is, $\bm{K}^{-1/2}\bm{K}^{-1/2}=\bm{K}^{-1}$ \cite[Theorem 7.2.6, p.\,439]{Horn}. In our case, the exact expression for  $\bm{K}^{-1/2}(\bm{z})$ is cumbersome (although available), but its asymptotic version under Assumption~\ref{as2} simplifies to
$$
\bm{K}^{-1/2}(\bm{z})\sim \frac{\sqrt{\braket{M}}}{\sqrt{q\mypp(q+1)}\braket{N}}\begin{pmatrix}
\sqrt{q}& -1\\[.4pc]
-1 & \displaystyle \frac{(q+1)\braket{N}}{\braket{M}}
\end{pmatrix}\!.
$$
Hence, after straightforward 
calculations we obtain from \eqref{eq:SNR-multi}
$$
\mathbf{SNR}(N_\lambda,M_\lambda) 
\sim \frac{\braket{M}}{q+1}\begin{pmatrix}
1&\sqrt{q}\\[.4pc]
\sqrt{q}& q
\end{pmatrix}\!,
$$
which tends to infinity in matrix sense.

\subsection{Rejection sampler}\label{sec:6.2}

The idea of a rejection sampler discussed in this section is to run a free sampler in a loop until a prescribed target is met. For example, if the target is set in terms of the required partition length as $M_\lambda=m$, with a fixed $m\in\NN$, then the free sampler is iterated until a partition of exact length $m$ is obtained. Likewise, if the target is set for the partition weight, $N_\lambda = n$, with a fixed $n\in\NN$, then the free sampling loop runs until a partition of exact weight $n$ is found. These two targets can be imposed simultaneously, $M_\lambda=m$ and $N_\lambda = n$; here, it is natural to design the rejection algorithm as a juxtaposition of two loops of the free sampler, such that the internal loop runs until the length target is met and, every time this happens, the resulting partition is checked with regards to the weight target and is either rejected, whereby the internal loop starts afresh, or accepted, in which case the algorithm stops. 

A more general approach, leading to the so-called approximate algorithms, is to relax the exact targets to suitable intervals (brackets). In Algorithm \ref{Resampler} presented below in Section \ref{sec:6.2.3}, we give an example of a Boltzmann rejection sampler aiming to sample a partition $\lambda\in\check{\varLambda}^q$ satisfying two conditions,  $M_\lambda = m$ and $n\le N_\lambda\le \theta\kern0.03pc n$, for some predefined tolerance factor $\theta\ge 1$. Of course, if $\theta=1$, the approximate algorithm is reduced to an exact one. An approximate target can also be considered for length, $m\le M_\lambda\le \theta' m$, and furthermore, such approximate targets can be combined if desired. 

Before delineating  Algorithm \ref{Resampler}, we discuss a few implementation issues arising therein. 

\subsubsection{Parameter calibration and truncation of parts}\label{sec:6.2.1} 

To start with, the choice of the calibrating parameters $z_1$ and $z_2$ now follows a slightly different logic as compared to the case of a free sampler. If only one target is set, such as $M_\lambda=m$, then according to formula \eqref{eq:uniM} the conditional Boltzmann measure $\QQ_{\bm{z}}(\cdot\,|\mypp\check{\varLambda}^q(\sbullet\myp,m))$ does not depend on the parameter $z_2$, whereas the other parameter, $z_1$, can be chosen with a view on a desired expected value $\braket{N}$ of the output weight $N_\lambda$, as was the case with the free sampler. Nonetheless, in order to maximize efficiency of the sampling algorithm, the ``free''  parameter $z_2$ should still be chosen in line with the mean conditions \eqref{braket} subject to the specification  $\braket{M}=m$, thus aiming to benefit from the bumpy nature of the distribution $\QQ_{\bm{z}}(\cdot\,|\mypp\check{\varLambda}^q(\sbullet\myp,m))$ (cf.\  Section~\ref{sec:6.1.3}). Similarly, if the condition $N_\lambda=n$ is targeted then the conditional Boltzmann measure $\QQ_{\bm{z}}(\cdot\,|\mypp\check{\varLambda}^q(n,\myn\sbullet))$ does not depend on $z_1$, however both parameters $z_1$ and $z_2$ are chosen to match the mean conditions \eqref{braket} with $\braket{N}=n$. Furthermore, if both targets are imposed, $N_\lambda=n$ and $M_\lambda=m$, then by Lemma~\ref{uniformity} the measure $\QQ_{\bm{z}}(\cdot\,|\myp\check{\varLambda}^q(n,m))$ is reduced to the uniform distribution on $\check{\varLambda}^q(n,m)$ regardless of the parameters $z_1$ and $z_2$. But, as explained above, it is worthwhile to calibrate them in line with the mean conditions \eqref{braket}. With these targets in mind, in what follows (including Algorithm \ref{Resampler}) we use the specific hyper-parameters $\braket{N}=n$ and $\braket{M}=m$. Moreover, since avoiding bias is no longer a concern (unlike in the case of the free sampler), using the leading-term expressions \eqref{eq:z12-crude} is a perfectly satisfactory option. Of course, the same recipe applies to the approximate sampling.

Another issue to be addressed is whether any truncation of the source of parts is needed in the algorithm (cf.\ Section \ref{sec:6.1.2}). As long as the weight target is involved, $N_\lambda = n$, one can use a natural majorant $L^*=n$ as a call parameter $L$ in Algorithm \ref{sampler} (see Section~\ref{sec:6.1.3}), which clearly causes no loss in confidence (i.e., $\delta=0$, see \eqref{eq:<L}). The same is true in the case of an  approximate target, $N_\lambda\in[n,\theta\kern0.04pc n]$, by choosing the majorant $L^*=\theta\kern0.06pc n$. However, if only the length target is in place 
then no such majorant is available and, therefore, confidence considerations must be deployed, as discussed in Section~\ref{sec:6.1.2}.
    
\subsubsection{Censoring of iterations}\label{sec:6.2.2}

As was pointed out in the Introduction, in contrast to the special value $q=1$, in the general case with $q\ge 2$ and arbitrary $m\ge1$ there is no guarantee for a given natural number $n\in\NN$ to be partitionable into a required number $m$ of $q$-power parts (unless it is covered by a solution of the Waring problem \cite{Waring}). The requirement that the parts be distinct adds to the complexity of the question. Therefore, the space $\check{\varLambda}^q(n,m)$ may well be empty and, not knowing this in advance, the task of sampling from such a space may be ``mission impossible''. 

To be specific,  consider sampling  subject to the joint targets $M_\lambda=m$ and $N_\lambda=n$. As already indicated at the beginning of Section \ref{sec:6}, a general design of the corresponding sampling algorithm is based on the two nested loops according to the separated targets, internal for $M_\lambda$ and external for $N_\lambda$. While the internal loop is certain to produce a random partition in $\check{\varLambda}^q$ with exactly $m$ parts (see more about this below), the external loop contains an inherent loose end due to its potential failure to satisfy the weight requirement, simply because there may be no such partitions. That is to say, although the successful completion of the external loop will take some time by repeatedly querying the internal loop, it may be pointless to wait for too long as there is no certainty if that is not wasteful.

We propose to resolve this difficulty by an appropriate ``censoring'' of processing time, that is, by setting a limit $t^*$ on the waiting time, chosen to ensure sufficiently high confidence in the algorithm's ability to deliver a successful completion within the allocated time limit, of course provided that the task is feasible (i.e., required partitions exist). More precisely, given a confidence tolerance (significance level) $\delta \in (0,1)$, the threshold $t^*$ should be such that, if the target space is non-empty, the probability that the algorithm would not succeed by time $t^*$ does not exceed $\delta$. 
Such an approach is akin to statistical testing of the null hypothesis ``the target is non-empty''. Under this hypothesis, the test (implemented as a sampling algorithm) fails to produce a required partition (Type I error) with rate bounded by $\delta$.

The choice of threshold $t^*$ is determined by the sampling task at hand. A few examples of interest are as follows (assuming the number of parts $m$ to be fixed):
\begin{itemize}
   \item[(T1)] \emph{Exact sampling:} For a given $n$, attempt to sample $\lambda\in\check{\varLambda}^q(n,m)$ (in other words, check partitionability of $n$).
    \item[(T2)] \emph{Multiple exact sampling:} For a given $n$ and some $\theta>1$, attempt to sample $\lambda\in\check{\varLambda}^q(k,m)$ for each integer $k$ in the range $k\in [n,\theta\kern0.02pc n]$ (that is, test partitionability of each of these numbers).
\item[(T3)] \emph{Approximate sampling:} Same as in task (T2) but attempting to sample $\lambda\in\bigcup_{n\le k\le \theta\kern0.02pc n} \check{\varLambda}^q(k,m)$ (that is, to find at least one partitionable number in the said range).
\end{itemize}

\begin{remark}
If required, tasks (T2) and (T3) could be modified to a two-sided version, such as $\theta^{-1}n\le k\le \theta\kern0.03pc n$ or, more generally, $\theta_1 n\le k\le \theta_2\myp n 
$, with some $0<\theta_1<1< \theta_2$.
\end{remark}

First, let us look at how the internal loop performs towards its task of sampling a partition $\lambda\in \check{\varLambda}^q_{L}(\sbullet\myp,m)=\check{\varLambda}^q(\sbullet\myp,m)\cap \{\lambda\colon\lambda_{\rm max}\le L\}$ (i.e., the source of parts is truncated to $\{\ell\le L\}$, see Section \ref{sec:2.5}). According to Theorem \ref{thm1<L}(a), if $\langle M\rangle>0$ is fixed and $L\sim\theta\myp\langle N\rangle$ then the distribution of $M_\lambda$ conditional on $\lambda_{\rm  max}\le L$ converges to a Poisson law with mean $\mu_\theta=\braket{M} G_{1/q}(a_\theta)$, where $a_\theta=\theta\myp\langle M\rangle\myn/q$.  
A stronger result concerns a Poisson approximation (cf.\ Remark \ref{rm:Poisson}) under a suitable metric, such as the total variation distance between distributions. Namely, the said conditional distribution of $M_\lambda$ can be replaced by a Poisson distribution with mean
\begin{equation}\label{eq:TV-mean}
\EE_{\bm{z}}(M_\lambda\,|\,\lambda_{\rm{max}}\le L)=\sum_{\ell\le L} \frac{z_1^{\ell}z_2}{1+z_1^{\ell}z_2}\sim \mu_\theta,
\end{equation}
with the error in total variation bounded by (see \cite[Theorem 1, p.\,474]{Barbour-Hall})
\begin{equation}\label{eq:TV-error}
\frac{1}{\mu_\theta}\sum_{\ell\le L}\left( \frac{z_1^{\ell}z_2}{1+z_1^{\ell}z_2}\right)^{\mynn2}=O(z_2)=O(\kappa^{1/q})\to 0,
\end{equation}
according to \eqref{eq:sqsum}. The advantage of such an  approximation is that it holds true even if $\braket{M}$ is slowly growing, whereby the error estimate \eqref{eq:TV-error} is still valid. 

Returning to the analysis of the internal loop, 
with $\braket{N}=n$ and $\braket{M}=m$ 
we have 
\begin{equation}\label{eq:Qm}
\QQ_{\bm{z}}(M_\lambda=m\,|\,\lambda_{\rm max}\le \theta\kern0.03pc n)\sim \frac{\mu_\theta^m\,\rme^{-\mu_\theta}}{m!} ,\qquad \mu_\theta=m\,G_{1/q}(\theta\myp m/q).
\end{equation}
Hence, the probability \eqref{eq:Qm} is bounded away from zero, and since the attempts within the internal loop are independent, the number of internal runs until success has geometric distribution, with the expected time to success being bounded by a constant (depending on $m$). If $m\to\infty$ (with $\kappa=m^{q+1}\mynn/n=o(1)$), then $\mu_\theta\sim m$ and, according to \eqref{eq:TV-mean} and \eqref{eq:TV-error}, we have
\begin{equation}\label{eq:Qm'}
\QQ_{\bm{z}}(M_\lambda=m\,|\,\lambda_{\rm max}\le \theta\kern0.03pc n)=\frac{\mu_\theta^m\,\rme^{-\mu_\theta}}{m!}+O(\kappa^{1/q})\sim \frac{m^m\,\rme^{-m}}{m!}\sim \frac{1}{\sqrt{2\pi\myp m}},
\end{equation}
with the help of the Stirling formula. In turn, formula \eqref{eq:Qm'}  implies that the expected number of runs of the internal loop is of order $O\bigl(\sqrt{m}\myp\bigr)$, which is not particularly large for practical implementation.

Let us now turn to tasks (T1)\,--\,(T3) and focus on probabilistic analysis of runs of the external loop, taking for \strut{}granted that $M_\lambda=m$ and, automatically, $\lambda_{\rm max}\le L^*$, where $L^*=\theta\kern0.03pc n$ is a majorant in $\bigcup_{n\le k\le \theta n} \check{\varLambda}^q(k,m)$. As stipulated above, we use the hyper-parameters $\braket{N}=n$ and $\braket{M}=m$, and \strut{}specify the calibrating parameters $z_1$ and $z_2$ according to the leading-term expressions \eqref{eq:z12-crude}.
Denote for short
\begin{equation}\label{eq:q*k'}
p_k^* := \QQ_{\bm{z}}(N_\lambda=k\,|\mypp M_\lambda =m, \myp\lambda_{\rm max}\le \theta\kern0.03pc n),\qquad n\le k \le \theta\kern0.03pc n. 
\end{equation}
for simplicity omitting reference to $n$ and $m$. Of course, if $\check{\varLambda}^q(k,m)=\varnothing$ then $p_k^*=0$. The probability \eqref{eq:q*k'} can be interpreted as the probability of successfully sampling a partition of weight $k$ in a single run of the external loop. Due to independence of successive runs, the number $T_k$ of attempts until success for a targeted weight value $k$ follows geometric distribution,
\begin{equation}\label{eq:Q*t'0}
\QQ_{\bm{z}}(T_k>t\,|\mypp M_\lambda=m, \lambda_{\rm max}\le n)= \left(1-p_k^*\right)^{t},\qquad t\in\NN_0.
\end{equation} 
This includes the case $p_k^*=0$, whereby $T_k=\infty$ ($\QQ_{\bm{z}}$-a.s.). Also note that $(T_k)$ are mutually independent for different $k$. 

\begin{itemize}
\item[(T1)] 
Here, $\theta=1$, so $L^*=n$. Suppose that  $\check{\varLambda}^q(n,m)\ne \varnothing$
and let $\lambda_* \in \check{\varLambda}^q(n,m)$, so that $N_{\lambda_*}=n$ and $M_{\lambda_*}=m$. Then we can write
\begin{equation}\label{eq:q*}
\QQ_{\bm{z}}(\lambda_*\mypp|\mypp M_\lambda =m, \myp\lambda_{\rm max}\le n)=\frac{\QQ_z(\lambda_*\myp|\mypp \lambda_{\rm max}\le n)}{\QQ_z(M_\lambda =m\,|\,\lambda_{\rm max}\le n)} .
\end{equation}
Starting with the  numerator, we have 
\begin{equation*}
\QQ_z(\lambda_*\myp|\mypp \lambda_{\rm max}\le n)=z_1^nz_2^m\prod_{\ell\le n}\frac{1}{1+z_1^{\ell}\myn z_2}.
\end{equation*}

Substituting formulas \eqref{eq:z12-crude}, we obtain 
\begin{equation*}
z_1^nz_2^m\sim \rme^{-m/q}\,\frac{q^{2m} (m/q)^{m+m/q}}{\bigl(\Gamma(1/q)\bigr)^m\mypp n^{m/q}},
\end{equation*}
while Lemma \ref{lm:sum2+}, with the help of the asymptotic relation \eqref{eq:z2/gamma}, yields
\begin{equation*}
\prod_{\ell\le n}\frac{1}{1+z_1^{\ell}\myn z_2}
\sim \exp\bigl(-m\,G_{1/q}(m/q)\bigr).
\end{equation*}

Furthermore, by Theorem \ref{thm1<L}(a) the denominator in \eqref{eq:q*} is asymptotically given by
\begin{equation}
\QQ_z(M_\lambda =m\,|\,\lambda_{\rm max}\le n)\sim \frac{m^m\myp \bigl(G_{1/q}(m/q)\bigr)^m \exp\bigl(-m\,G_{1/q}(m/q)\bigr)}{m!}.
\end{equation}
Hence, returning to \eqref{eq:q*} we obtain
\begin{equation}\label{eq:q>}
p_n^*\ge \QQ_{\bm{z}}(\lambda_*\mypp|\mypp M_\lambda=m, \lambda_{\rm max}\le n)\sim\frac{1}{C_1(m,q)\,n^{m/q}},
\end{equation}
where
\begin{equation}\label{eq:C1}
C_1(m,q):=\frac{1}{m!}
\left(\frac{\rme^{1/q}\,\Gamma(1/q)\,G_{1/q}(m/q)}{q^{1-1/q}\myp m^{1/q}}\right)^{\mynn m}\!.
\end{equation}
Combining \eqref{eq:Q*t'0} and \eqref{eq:q>}, we have, asymptotically, \begin{equation}\label{eq:Q*t''}
\QQ_{\bm{z}}(T_n>t\,|\mypp M_\lambda=m, \lambda_{\rm max} \le n)= \left(1-p_n^*\right)^{t}
\le\left(1-\frac{1+o(1)}{C_1(m,q)\,n^{m/q}}\right)^{\myn t}\!.
\end{equation} 
Thus, for the probability \eqref{eq:Q*t''} not to exceed a predefined (small) confidence tolerance $\delta>0$, it suffices to choose the threshold $t=t_n^*$ as follows,
\begin{equation}\label{eq:t*new'}
t_n^*\simeq\frac{\displaystyle \log\delta}
{\displaystyle\log\left(1-\frac{1+o(1)}{C_1(m,q)\,n^{m/q}}\right)}\sim C_1(m,q)\,n^{m/q} \log\frac{1}{\delta}\myp.
\end{equation}

\begin{remark}\label{rm:better1}
The bound \eqref{eq:t*new'} is very conservative due to a crude estimate \eqref{eq:q>} leveraging just one instance $\lambda_*\in \check{\varLambda}^q(n,m)$. If more information was available about the size of the space $\check{\varLambda}^q(n,m)$, the bound \eqref{eq:t*new'} could be reduced accordingly. For example, if $q=1$ then it is known that $\#\check{\varLambda}(n,m)\sim n^{m-1}(m!\,(m-1)!)^{-1}$ (see \eqref{eq:q=1-size}). Hence, formula \eqref{eq:t*new'} for the time threshold in the case $q=1$ is replaced by a much better and more realistic estimate,
\begin{equation}\label{eq:t*1corr}
\tilde{t}_n^*\simeq \frac{C_1(m,1)}{\#\check{\varLambda}(n,m)}\,n^{m} \log\frac{1}{\delta}\sim\frac{\rme^{m}\left(1-\rme^{-m}\right)^m (m-1)!}{m^m}\,n\myp \log\frac{1}{\delta}\myp.
\end{equation}
Likewise, for $q=2$ we get, using \eqref{eq:Lq=2} (with $C_m=1$),
\begin{equation}\label{eq:t*2corr}
\tilde{t}_n^*\simeq \frac{C_1(m,2)}{\#\check{\varLambda}^2(n,m)}\,n^{m/2} \log\frac{1}{\delta}\sim\left(\frac{2\mypp\rme }{m}\right)^{\!m/2}\Gamma(m/2)\,\bigl(G_{1/2}(m/2)\bigr)^m\,n\myp \log\frac{1}{\delta}\myp,
\end{equation}
which again grows only linearly in $n$. 
\end{remark}

\item[(T2)] For the multiple exact sampling in the range $k\in[n,\theta\kern0.02pc n]$,
we can just repeat the procedure in task (T1) for each $k$ in that range. 
According to \eqref{eq:Q*t'0}, 
the probability that the number of attempts until success, $T_k$, exceeds a threshold $t_k^*$ is given by (cf.\  \eqref{eq:Q*t'0})
$$
\QQ_{\bm{z}}\bigl(T_k>t_k^*\,|\mypp M_\lambda=m, \lambda_{\rm max} \le k\bigr)=(1-p_k^*)^{t_k^*},\qquad n\le k\le \theta\kern0.03pc n.
$$
Taking into account only partitionable numbers $k\in[n,\theta\kern0.02pc n]$ (i.e., such that $\check{\varLambda}^q(k,m)\ne \varnothing$ and, therefore, $p_k^*>0$)
and using a Bonferroni-type inequality, the probability of Type I error for task (T2) (i.e., that the external loop fails for at least one such $k$) is bounded as follows,
\begin{equation}
\label{eq:Q*n}\QQ_{\bm{z}}\Biggl(\,\bigcup_{k=n}^{\lfloor \theta n\rfloor}\{t_k^*<T_k<\infty\}\,\Bigl|\mypp M_\lambda=m, \lambda_{\rm max} \le \theta\kern0.03pc n\Biggr)
\le \sum_{k\colon\mynn p_k^*>0}(1-p_k^*)^{t_k^*}.
\end{equation}

Motivated by formula  \eqref{eq:t*new'}, we can look for the time limits $t_k^*$ in the form 
\begin{equation}\label{eq:tk*}
t_k^*\sim c\mypp k^{m/q}.
\end{equation}
Then from \eqref{eq:Q*n} using \eqref{eq:q>} we get, asymptotically,
\begin{align*}
\sum_{k\colon\mynn p_k^*>0}(1-p_k^*)^{t_k^*}&\lesssim \sum_{k=n}^{\lfloor \theta n\rfloor}\left(1-\frac{1}{C_1(m,q)\,k^{m/q}}\right)^{c\myp k^{m/q}.}\\
&\le (\theta-1)\, n\mypp \exp\left(-\frac{c}{C_1(m,q)}\right)\le \delta.
\end{align*}
Solving this inequality for $c$ and returning to \eqref{eq:tk*} ultimately yields
\begin{equation}\label{eq:t*2}
\displaystyle t_k^*\simeq C_1(m,q)\,k^{m/q} \log\frac{ (\theta-1)\myp n}{\delta},\qquad n\le k\le \theta\kern0,03pc n.
\end{equation}
Thus, the time bound \eqref{eq:t*2} follows the same formula as in a single test (cf.\ \eqref{eq:t*new'}) but with a Bonferroni-type adjustment of the significance level in order to offset the multiple testing.

\begin{remark}\label{rm:better2} The same comment as in Remark \ref{rm:better1} applies to task (T2). Specifically, for $q=1$ and $q=2$ the improved formulas for the thresholds $t^*_k$ are given, respectively, by
\begin{align}\label{eq:t*1corrT2}
\tilde{t}_k^*&\simeq \frac{\rme^{m}\left(1-\rme^{-m}\right)^m (m-1)!}{m^m}\,k\myp \log\frac{ (\theta-1)\myp n}{\delta},\\
\label{eq:t*2corrT2}
\tilde{t}_k^*&\simeq \left(\frac{2\mypp\rme }{\pi m}\right)^{\!m/2}\Gamma(m/2)\,\biggl(\int_{0}^{m/2} \!u^{-1/2}\mypp\rme^{-u}\,\rmd{u}\biggr)^m\,k\myp \log\frac{ (\theta-1)\myp n}{\delta}.
\end{align}
\end{remark}

\item[(T3)] 
In a single attempt, the external loop gets a partition $\lambda\in \bigcup_{n\le k\le\theta\kern0.02pc n}\check{\varLambda}^q(k,m)$ with probability 
\begin{align}
\QQ_{\bm{z}}\bigl(n\le N_\lambda\le \theta\kern0.02pc n\:|\mypp M_\lambda=m,\lambda_{\rm max}\le \theta\kern0.02pc n)&=\sum_{k=n}^{\lfloor \theta n\rfloor} p_k^*\to G_{1/q}^{\star m}(a_\theta\myp|\mypp a_\theta)-G_{1/q}^{\star m}(a_1|\mypp a_\theta).
\label{eq:sum-q*}
\end{align}
The limit \eqref{eq:sum-q*} is due to Theorem \ref{thm1<L}(b), where $a_\theta=\theta\myp m/q$ and $G_{1/q}^{\star m}(x\mypp|\mypp a_\theta)$ stands for the $m$-convolution of the $a_\theta$-truncated gamma distribution $G_{1/q}(x\mypp|\mypp a_\theta)$. 

To circumvent the trouble of computing such a convolution, observe that in the range $0\le x\le a$ the distribution function $G^{\star m}_{\alpha}(x\mypp|\mypp a)$ coincides with $G_{m\myp\alpha}(x)$ up to the normalization factor $G_{\alpha}(a)^m$. This is obvious for $m=1$, and the general case can be seen by induction over $m$ using the convolution formula. Indeed, denoting the corresponding densities by  $g^{\star m}_{\alpha}(x\mypp|\mypp a)$ and $g_{m\alpha}(x)$, respectively, we have by definition $g_{\alpha}(x\mypp|\mypp a)=g_\alpha(x)/G_\alpha(a)$ ($0\le x\le a$), and the induction step is carried out as follows, 
\begin{align*}
g^{\star m}_{\alpha}(x\mypp|\mypp a)&=\int_0^x\! g^{\star(m-1)}_\alpha(u\mypp|\mypp a)\,g_{\alpha}(x-u\mypp|\mypp a)\,\rmd{u}\\
&=\frac{1}{G_\alpha(a)^m}\int_0^x\! g_{(m-1)\myp\alpha}(u)\,g_{\alpha}(x-u)\,\rmd{u}=\frac{g_{m\myp\alpha}(x)}{\bigl(G_\alpha(a)\bigr)^m},\qquad 0\le x\le a,
\end{align*}
due to the convolution property of the gamma distribution. 
Thus, formula \eqref{eq:sum-q*} simplifies to 
\begin{equation}
\label{eq:sum-q*1}
\sum_{k=n}^{\lfloor \theta n\rfloor} p_k^*\to \frac{G_{m/q}(\theta\myp m/q)-G_{m/q}(m/q)}{{\bigl(G_{1/q}(\theta\myp m/q)\bigr)^m}}=:C_3(m,q,\theta).
\end{equation}

Now, since individual runs of the external loop are independent,
the probability that the number of attempts until success, $T$, exceeds a threshold $t$ is given by (cf.\  \eqref{eq:Q*t'0})
\begin{equation}\label{eq:Q*t''1}
\QQ_{\bm{z}}(T>t\,|\mypp M_\lambda=m, \lambda_{\rm max}\le \theta\kern0.02pc n)=\Biggl( 1-\sum_{k=n}^{\lfloor \theta n\rfloor} p_k^*\Biggr)^{\!t}\simeq \bigl(1-C_3(m,q,\theta)\bigr)^t,
\end{equation} 
on account of \eqref{eq:sum-q*1}. Hence, in order that this probability be bounded by a confidence tolerance 
$\delta>0$, we may choose the threshold $t=t^*$ as follows, 
\begin{equation}\label{eq:omni-t*'}
t^*\simeq\frac{\log \delta}{\log \left(1-C_3(m,q,\theta)\right)}. 
\end{equation}

\begin{remark}
According to formula \eqref{eq:omni-t*'}, the threshold $t^*$ does not depend on $n$. It is of interest to look at how it depends on the growth of the power $q$. To this end, by a direct analysis of the gamma distribution (see \eqref{eq:gamma-cdf}) one can verify that
$$
G_\alpha(\theta\myp\alpha)=(\theta\myp\alpha)^\alpha\bigl(1+O(\alpha)\bigr)\qquad (\alpha\to0^+).
$$
Hence, from formula \eqref{eq:sum-q*1} we get 
$$
C_3(m,q,\theta)\sim 1-\theta^{-m/q}\sim\frac{m\log\theta}{q}\qquad (q\to\infty),
$$
and then \eqref{eq:omni-t*'} gives 
$$
t^*\simeq\frac{q\mypp\log\mynn (1/\delta)}{m\myp \log \theta}. 
$$
\end{remark}

\end{itemize}

\subsubsection{Rejection sampling algorithm}\label{sec:6.2.3}

A stylized example of rejection sampler is presented below in pseudocode as Algorithm \ref{Resampler}. It is set out in a flexible way so as to be usable in exact and approximate sampling alike, as determined by the tasks (T1)--(T3) described in Section \ref{sec:6.2.2}. In particular, the range parameter $\theta$ is allowed to take the value $\theta=1$, in which case the algorithm would work towards the exact sampling task (T1) (i.e., with a specific weight target $N_\lambda=n$). As explained in Section \ref{sec:6.2.1}, the hyper-parameters are adapted to the desired targets, $\braket{N}=n$ and $\braket{M}=m$, and the calibrating parameters $z_1$ and $z_2$ are set according to the simplified expressions \eqref{eq:z12-crude}. A predefined time bound $t^*$ for the external loop is selected according to the task at hand, as discussed in Section \ref{sec:6.2.2}, and on account of the required confidence probability $1-\delta$.

As briefly indicated at the start of Section \ref{sec:6.2}, Algorithm \ref{Resampler} comprises an external loop that iterates the free sampler in an internal loop (i.e., Algorithm \ref{sampler}, with the majorant $L=\theta\kern0.02pc n$), which delivers, in each productive cycle, a partition $\lambda$ that meets the length target $M_\lambda=m$. This continues until the trial partition $\lambda$ meets the weight target (e.g., $N_\lambda=n$ in task (T1) or $n\le N_\lambda\le \theta\kern0.03pc n$ in task (T3)). However, if the limit of attempts $t^*$ is reached with no success then the algorithm terminates, returning a message `VOID'. It remains to add that for task (T2) involving multiple exact sampling, the algorithm should be run in an additional loop to scan all weight values in the range $k\in[n,\theta\kern0.02pc n]$.

\begin{figure}[ht!]
\begin{algorithm}[H]
\label{Resampler}
\caption{{\tt ReSampler}\mypp($q,n, m,\theta, t^*$)}
\KwInput{integer $q,n, m$, real $\theta\ge 1$, $t^*$ 
}
\KwOutput{partition
$\lambda\in\check{\varLambda}^{q}(\sbullet\myp,m)$ with 
$N_\lambda\in[n,\theta\kern0.02pc n]$, otherwise `VOID'} 
integer array $\lambda_{[\ ]}$\;
integer $N, M, t$\;
real 
$L$\;
$L\leftarrow \theta\kern0.02pc n$\;
$N \leftarrow 0$, $M \leftarrow 0$, $t \leftarrow 0$\;
\While{$N\notin[n,\theta\kern0.02pc n]$ and $t\le t^*$\label{outloop}}{
	\While{$M\ne m$\label{inloop} }
	{
		$(\lambda, N, M) \leftarrow {\tt FreeSampler}(q,n,m,L)$
	}
	$t\leftarrow t+1$\;
}
\uIf{$t \le t^*$} {$N_\lambda\leftarrow N$\;
\Return $(\lambda,N_\lambda)$}
\Else  {\Return \texttt{`VOID'}
}
\end{algorithm}
\end{figure}

Algorithm \ref{Resampler} can be optimized in a number of ways. Since the weight of a valid output $\lambda$ should not exceed $\theta\kern0.03pc n$, it is clear that the run of the internal loop can be terminated prior to collecting the required number of parts $m$ if the next candidate part is too large, so that the incremented weight will certainly exceed the majorant. Furthermore, if the number of collected parts has already reached the target value $m$ then there is no need to keep scanning the remaining values in the range $\ell\le L$ and the current run of the internal loop may be stopped without any loss. However, to avoid bias and maintain the Boltzmann distribution of the output, the corresponding proposal $\lambda\in\check{\varLambda}_L^q(\sbullet\myp,m)$ must be accepted only if the remaining candidate parts in the range $\ell\le L$ were to be rejected by the respective Bernoulli checks. Since individual such checks are mutually independent, their multitude can be replaced by a single Bernoulli trial with the corresponding product probability of failure. An additional benefit of such an aggregated Bernoulli check is that this will reduce the number of calls of the random number generator and hence improve the efficiency of the sampler.

Another improvement of the code implementation in the multiple testing task (T2) proceeds from the observation that the sequential procedure based on separate testing of each target in the range $[n,\theta\kern0.02pc n]$ (see Section \ref{sec:6.2.2}) is apparently wasteful, because a partition of some $k'\in[n,\theta\kern0.02pc n]$ obtained whilst looking for partitions of a different number $k$ would be discarded in that cycle of the external loop, whereas keeping it would have helped to achieve success if an earlier search with target $k'$ failed, or to save time on a duplicate job when the algorithm moves to the new target $k'$. In practice, all partitions (at least, the new ones) obtained in every run of the external loop should be stored as long as they fit into the range $[n,\theta\kern0.02pc n]$, thus leaving dynamically fewer targets to address.

For the sake of presentational clarity, Algorithm \ref{Resampler} embeds iterated calls of the free sampler (Algorithm \ref{sampler}), but this means that the calibrating parameters $z_1$ and $z_2$ are recalculated at every such call, which is of course wasteful. This drawback can be easily amended by writing out the code explicitly. Note, however, that such an  improvement would have no significant bearing on the asymptotic estimation of the code complexity.

\begin{remark}
It would be interesting to explore if the performance of our sampling scheme can be improved via a probabilistic divide-and-conquer method  
proposed by Arratia and DeSalvo \cite{ADS}. 
\end{remark}

\subsubsection{Complexity and performance}
Building on the probabilistic analysis of the internal and external loops carried out in Section \ref{sec:6.2.2}, it is  straightforward to estimate the \emph{time complexity} of Algorithm \ref{Resampler}, understood as the \emph{expected number of elementary runs to completion}. 

Starting with the internal loop, in its crude (non-optimized) version each internal run comprises $\lfloor (\theta\kern0.03pc n)^{1/q}\rfloor$ checks of available parts $\ell\in \NN^q$ not exceeding $L^*\mynn=\theta\kern0.03pc n$. 
Combined with the estimate \eqref{eq:Qm} of the probability to collect $m$ parts in a single run and the corresponding geometric distribution of the number of attempts, the complexity of the internal loop is bounded by 
\begin{equation}\label{eq:M-comp}
\mu_\theta^{-m}\myp m!\,\rme^{\myp\mu_\theta}\myp (\theta\kern0.03pc n)^{1/q}.
\end{equation}

As for the external loop, its complexity depends on the task at hand. If $\theta=1$ (which corresponds to task (T1) of exact sampling with the weight target $N_\lambda=n$), then the time to completion, $T_n$, has geometric distribution with parameter $p_n^*$ (see \eqref{eq:Q*t'0}). For simplicity dropping a time bound $t^*$ (but still assuming that the space $\check{\varLambda}^q(n,m)$ is non-empty, so that $p_n^*>0$), the expected time to completion is given by $\EE_{\bm{z}}(T_n)=1/p_n^*$.
With a time bound $t^*$, the expectation is modified as follows,
\begin{align}\label{eq:T1-cmplx-2}
\EE_{\bm{z}}(T_n; T_n\le t^*)&=\sum_{t=1}^{\,t^*\!} t\mypp(1-p_n^*)^{t-1}p_n^*+t^*(1-p_n^*)^{t^*}=\frac{1-(1-p_n^*)^{t^*}}{p_n^*}<\frac{1}{p_n^*}.
\end{align}
However, the reduction in \eqref{eq:T1-cmplx-2} is not significant, because under our confidence-based choice of the time limit (see \eqref{eq:q>}), we always have
$(1-p_n^*)^{t^*}\mynn\!\le \delta$. Thus, combining formulas  \eqref{eq:M-comp} and \eqref{eq:T1-cmplx-2}, the total complexity guarantee for task (T1) is estimated by 
\begin{equation}\label{eq:MN-compT1}
\frac{m!\,\rme^{\myp\mu_1}}{\mu_1^{m}}\mypp O\bigl(n^{1/q}\myn/p_n^*\bigr),\qquad \mu_1=m\,G_{1/q}(m/q).
\end{equation}
Further specification depends on the informative lower bound for the probability $p_n^*$. For example, a crude estimate 
\eqref{eq:q>} gives a more explicit estimate for the complexity,
\begin{equation}\label{eq:MN-compT1'}
\left(\frac{\exp\bigl(G_{1/q}(m/q)+1/q\bigr)\,\Gamma(1/q)}{q^{1-1/q}\myp m^{1+1/q}}\right)^{\mynn m}\mypp O\bigl(n^{(m+1)/q}\bigr).
\end{equation}
For $q=1$ and $q=2$, this estimate can be significantly improved by using asymptotically exact cardinalities \eqref{eq:q=1-size} and \eqref{eq:Lq=2}, respectively, yielding the estimates
\begin{equation}\label{eq:MN-compT1q=1}
\frac{(m!)^2\, \rme^{2m}}{m^{2m+1}} \, O(n^2)=O(n^2)
\end{equation}
and
\begin{equation}\label{eq:MN-compT1q=2}
\frac{2^{m/2}\myp m!\:\Gamma(m/2)\,\rme^{3m/2}}{m^{3m/2}}\mypp O\bigl(n^{3/2}\bigr)=O(n^{3/2}).
\end{equation}
Interestingly, the asymptotic bounds \eqref{eq:MN-compT1q=1} and  \eqref{eq:MN-compT1q=2} do not depend on the number of parts $m$.

For task (T2) (with some $\theta>1$), the above estimates just need to be multiplied by the number of targeted weights, $\lfloor (\theta-1)\mypp  n\rfloor+1=O(n)$. 
Finally, for task (T3) we can use formula \eqref{eq:MN-compT1}, but with $\mu_1$ changed to $\mu_\theta$ (see \eqref{eq:M-comp}) and with the probability $p_n^*$ of success in a single attempt replaced by the (asymptotic) probability \eqref{eq:sum-q*1} of at least one success in the range $[n,\theta\kern0.003pc n]$, yielding
$$
\frac{m!\,\exp\bigl(m\,G_{1/q}(\theta\myp m/q)\bigr)}{m^m\mypp\bigl(G_{m/q}(\theta\myp m/q)-G_{m/q}(m/q)\bigr)}\,O(n^{1/q})=
O\bigl(m\myp n^{1/q}\bigr).
$$

\begin{table}[b!h!t]
\caption{Confident time thresholds $t^*$ for the external loop in tasks (T1), (T2) and (T3), calculated for $q=1$ and $q=2$ with various values of confidence tolerance $\delta$ using formulas \eqref{eq:t*new'}, \eqref{eq:t*2} and \eqref{eq:omni-t*'}. In both cases, the chosen values of $n$ and $m$ yield $\kappa=m^{q+1}\mynn/n=0.01$. For tasks (T2) and (T3), the testing range is set with the factor $\theta=1.1$. For comparison, corrected values $\tilde{t}^*$ for tasks (T1) and (T2) are calculated from  formulas \eqref{eq:t*1corr}, \eqref{eq:t*2corr} and \eqref{eq:t*1corrT2}, \eqref{eq:t*2corrT2}, respectively.}
 \label{table4}
\small
\vspace{1pc}\centering\begin{tabular}{|c||c|c||c|c||c|}
 \hline
\multicolumn{6}{|c|}{$\vphantom{\int^{t}}q=1$, $n=2{,}\myp500$, $m=5$}\\
\hline
$\vphantom{\int_y^{t}}\delta$& $t_n^*$\,(T1)  &$\tilde{t}_n^*\vphantom{\int^{T}}$\,(T1)  &$t_n^*$\,(T2) &$\tilde{t}_n^*$\,(T2) &$t^*$\,(T3)\\
 \hline 
$0.1$&$8.603518\cdot 10^{13}\vphantom{\int^{t}}$&$6{,}\myp343.202$&$2.923424\cdot10^{14}$&$21{,}\myp553.82$& $26.01955$
\\
 \hline 
$0.01$&$1.720704\cdot 10^{14}$&$12{,}\myp686.40$&$3.783776\cdot10^{14}\vphantom{\int^{t}}$&$27{,}\myp897.02$&  $52.03911$
\\
 \hline 
$0.001$&$2.581056\cdot 10^{14}$&$19{,}\myp029.61$&$4.644128\cdot10^{14}\vphantom{\int^{t}}$&$34{,}\myp240.22$&$78.05866$ \\
\hline 
$0.0001$&$3.441407\cdot 10^{14}$&$25{,}\myp372.81$&$5.504479\cdot10^{14}\vphantom{\int^{t}}$&$40{,}\myp583.43$&$104.0782$ \\
\hline
\multicolumn{6}{c}{}\\[-.5pc]
\hline
\multicolumn{6}{|c|}{$q=2$, $n=12{,}\myp500$, $m=5$}\\
\hline
$\vphantom{\int_y^{t}}\delta$&$t_n^*$\,(T1) &$\tilde{t}_n^*$\,(T1) &$t_n^*$\,(T2) &$\tilde{t}_n^*$\,(T2) &$t^*$\,(T3)\\
 \hline 
$\vphantom{\int^{t}}0.1$&$198{,}\myp687{,}\myp146$ &$41{,}\myp485.61$&$814{,}\myp003{,}\myp357$&$169{,}\myp962.8
$&$34.94282$
\\
 \hline 
$\vphantom{\int^{t}}0.01$&$397{,}\myp374{,}\myp292$ &$82{,}\myp971.22$&$1{,}\myp012{,}\myp690{,}\myp503$&$211{,}\myp448.4$&  $69.88564$
\\
 \hline 
$\vphantom{\int^{t}}0.001$&$596{,}\myp061{,}\myp438$&$124{,}\myp456.8$&$ 1{,}\myp211{,}\myp377{,}\myp649$&$252{,}\myp934.0$&$104.8285$ \\
\hline 
$\vphantom{\int^{t}}0.0001$&$794{,}\myp748{,}\myp583$&$165{,}\myp942.4$&$1{,}\myp410{,}\myp064{,}\myp795$&$294{,}\myp419.7$&$139.7713$ \\
\hline\end{tabular}
\end{table}

%
%
%

To evaluate real time performance of the rejection sampler, we first need to take a practical look at the censoring time limits $t^*$ in tasks (T1)--(T3) proposed in Section \ref{sec:6.2.2}. These are numerically illustrated in Table \ref{table4} for $q=1$ and $q=2$, with various values of $n$ and $m$. Observe that the crude bounds for tasks (T1) and (T2) calculated via formulas \eqref{eq:t*new'} and \eqref{eq:t*2} appear to be very high, especially for $q=1$ (of order $10^{14}$), casting doubt on whether such limits are usable. In real terms, since each run of the external loop is a simple check if $N_\lambda=n$ (see line \ref{outloop} in Algorithm \ref{Resampler}), we can assume for simplicity that it needs a single tick of the CPU clock. If the algorithm is executed on a contemporary mid-range desktop PC (say, with processor base frequency 3.30 GHz, which we used) then, under the estimate \eqref{eq:t*new'} for task (T1) with $q=1$ and a fairly low confidence tolerance $\delta=0.001$, the external loop alone may require up to $2.581056 \cdot 10^{14}/(3.30\cdot 10^9\cdot60\cdot 60)\doteq 21.72606\approx 22$ hours until completion, which is unpleasantly long but not entirely unrealistic. This estimate drops dramatically for $q=2$ to less than $1$ second. A steep decreasing trend continues with larger powers;\footnote{Keeping $m$ and $\kappa=m^{q+1}\mynn/n$ fixed, from formulas \eqref{eq:C1} and \eqref{eq:t*new'} we find $\lim _{q\to\infty} t_n^*=(m^m\myn /m!)\log\mynn(1/\delta)$. For example, for $m=5$ as in Table \ref{table4} and $\delta=0.001$, this limiting value specializes to $179.8895$} for example, for $q=3$, $n=62{,}\myp500$, $m=5$ and $\delta=0.001$, formula \eqref{eq:t*new'} gives $t_n^*=3{,}\myp358{,}\myp531$, leading to the maximum execution time of up to $0.001$ second. Thus, the sampler becomes progressively more efficient for larger $q$, even under a crude time bound. On the other hand, as pointed out in Remarks \ref{rm:better1} and \ref{rm:better2}, additional information about the size of the corresponding partition spaces would allow a significant reduction of the estimated bound as illustrated in Table \ref{table4} (by a factor $10^{10}$ for $q=1$ and about $4{,}\myp790$ for $q=2$).

Let us now look at the real time computational cost due to the internal loop. As mentioned before (cf.\ \eqref{eq:M-comp}), the expected numbers of internal runs until collecting exactly $m$ parts is asymptotically given by $\mu_1^{-m}\myp m!\,\rme^{\myp\mu_1}$ (with $\theta=1$), where $\mu_1=m\,G_{1/q}(m/q)$. Using for numerical illustration the same values of $q$, $n$ and $m$ as in Table \ref{table4}, this formula yields $5.6993$
($q=1$) and $5.7043$
($q=2$). The average computing time for each of such attempts is inversely proportional to the CPU base frequency (such as 3.30 GHz), but it involves many other important aspects such as the operational efficiency of a random number generator, design of memory allocation and data storage, numerical precision, coding implementation and compiler used, and the overall architecture of the computer (e.g., the number of cores and whether or not parallel processing was utilized). Thus, it is impossible to estimate the actual computing time without real benchmarking.

To test the performance of the internal loop, we implemented the algorithm on a desktop CPU as described at the beginning of Section \ref{sec:6}, for simplicity using a single core. Since internal runs are independent from each other and the computational costs due to the multi-core design are negligible, we can simply divide the average execution time on a single core by the number of cores at disposal.

With the same values of $q$, $n$ and $m$ used above and in Table \ref{table4} (and with $\theta = 1$), the average number of sampling attempts (starting at line \ref{inloop} of Algorithm \ref{Resampler}) was $5.6956$ for $q=1$ and $5.6404$ for $q=2$; note that these sample averages match the expected values calculated above. Furthermore, the program took on average $2.1036\cdot 10^{-3}\myn$ seconds ($q=1$) and  $0.9780\cdot 10^{-4}\myn$ seconds ($q=2$) per single successful completion of the internal loop. The corresponding number of ticks of the CPU clock per elementary check of a candidate part $\ell\le n$ (see formula \eqref{eq:M-comp}) is evaluated as $(2.1036\cdot 10^{-3}\myn/(5.6956\cdot 2{}\mypp 500))\cdot 3.30 \cdot 10^9\doteq 487.5258$ ($q=1$) and $(0.9780\cdot 10^{-4}\myn/(5.6404\,\sqrt{12{}\mypp500}\mypp))\cdot 3.30\cdot 10^9=511.7854$, so it stays in the range about $450\div 550$. 

However, there is a problem: if we combine the physical times benchmarked for the internal loop with the time bounds $t^*_n$ for the external loop given in Table \ref{table4} (say, with tolerance $\delta=0.001$), then for $q=1$ we obtain, by converting seconds to minutes, hours, days and years, $2.1036\cdot 10^{-3}\myn\cdot 2.581056\cdot 10^{14}\myn/(60\cdot 60\cdot 24\cdot 365)\approx17{,}\myp217$ years\,(!), which is clearly impractical. For $q=2$, a similar calculation gives a more reasonable estimate,  $0.9780\cdot 10^{-4}\myn\cdot 596{}\, 061{}\mypp 438 /(60\cdot 60)\approx16$ hours. But with the improved time bounds $\tilde{t}_n^*$ (see Table \ref{table4}), we obtain much more satisfactory estimates, $2.1036\cdot 10^{-3}\myn\cdot 19\, 029.61\approx40$ seconds ($q=1$) and $0.9780\cdot 10^{-4}\cdot 124\,456.8 \approx12$ seconds ($q=2$).

As a concluding remark, 
Algorithm \ref{Resampler} could be used as an experimental tool for searching satisfactory instances in additive problems of number theory such as variants of the Waring problem. Here, iterated sampling attempts to find a suitable instance subject to certain constraints\,---\,in the lack of prior knowledge if such instances even exist\,---\,may be interpreted as statistical testing of existence as the null hypothesis, under which a suitable confident time limit $t^*$ can be determined. The sampling approach based on bounded (although high) confidence bears similarity with primality testing procedure such as the Miller--Rabin algorithm \cite{Rabin}. It can also be helpful in verification of conjectures about the density of representable numbers via an experimental analysis of the success rates. We will address such applications in another paper.

\vspace{-.3pc}
\section*{Acknowledgments}
J.C.P.\ was supported by an EPSRC Doctoral Training Partnership scholarship at the School of Mathematics, University of Leeds (grant number 2106382). The authors have benefited from the useful discussions with Yuri V.\ Yakubovich.

%



\end{document}